\documentclass[a4paper]{article}

\usepackage[utf8]{inputenc}
\usepackage[T1]{fontenc}
\usepackage{csquotes}
\usepackage[english]{babel}

\usepackage{bm}

\usepackage[backend=bibtex,style=alphabetic,giveninits=true,isbn=false,doi=false,url=false]{biblatex}
\usepackage{lmodern}

\usepackage{amssymb,amsmath,amsthm,mathrsfs}

\usepackage{todonotes}
\usepackage{mathtools,bbm,bm}
\usepackage[normalem]{ulem}
\usepackage{cancel}
\usepackage{verbatim}
\usepackage{tikz}
\usetikzlibrary{matrix}
\pgfdeclarelayer{background} %%%% creating two layers for tikz pictures, also line below
\pgfsetlayers{background,main} %%
\usepackage[shortlabels]{enumitem}

\usepackage{hyperref}

\usepackage{a4}
\usepackage{multirow}
\usepackage[footnotesize,bf,centerlast]{caption}
\usepackage{xcolor,colortbl}
\definecolor{Gray}{gray}{0.80}
\definecolor{LightGray}{gray}{0.90}
\definecolor{darkpastelgreen}{rgb}{0.01, 0.75, 0.24}

\usepackage{tikz}  % richard: for my diagrams 
\usetikzlibrary{matrix} % richard: matrices in diagrams
\pgfdeclarelayer{background} %%%% creating two layers for tikz pictures, also line below
\pgfsetlayers{background,main} %%

\newcommand{\cB}{\mathcal{B}}

\newcommand{\cD}{\mathcal{D}}
\newcommand{\cE}{\mathcal{E}}
\newcommand{\cF}{\mathcal{F}}
\newcommand{\cG}{\mathcal{G}}
\newcommand{\cH}{\mathcal{H}}

\newcommand{\cO}{\mathcal{O}}
\newcommand{\cP}{\mathcal{P}}

\newcommand{\cX}{\mathcal{X}}

\newcommand{\bN}{\mathbb{N}}

\newcommand{\bR}{\mathbb{R}}

\newcommand{\dd}{ \mathrm{d}}

\DeclareMathOperator{\grad}{grad}

\renewcommand{\epsilon}{\varepsilon}
\newcommand{\epsal}{\varepsilon_{\alpha}}
\newcommand{\epsaln}{\varepsilon_{\alpha_n}}

\newcommand{\vn}[1]{\left| \! \left| #1\right| \!\right|}

\newcommand{\ip}[2]{\langle #1,#2\rangle}

\numberwithin{equation}{section}

\newtheorem{theorem}{Theorem}[section]
\newtheorem{lemma}[theorem]{Lemma}
\newtheorem{proposition}[theorem]{Proposition}

\theoremstyle{definition}
\newtheorem{definition}[theorem]{Definition}
\newtheorem{remark}[theorem]{Remark}

\newtheorem{assumption}[theorem]{Assumption}

\newcommand{\R}{\mathbb{R}}
\newcommand{\De}{\mathrm{d}}

\newcommand{\geo}[2]{\bm{\zeta}^{{#1}\to{#2}}}
\newcommand{\geod}[2]{\dot{\bm{\zeta}}^{{#1}\to{#2}}}

\title{Hamilton--Jacobi equations for controlled gradient flows: the comparison principle}

\author{Conforti G. \thanks{CMAP, Ecole Polytechnique, Route de Saclay, 91128, Palaiseau Cedex, France. \emph{E-mail address}: giovanni.conforti@polytechnique.edu. Research supported by the ANR project  ANR-20-CE40-0014.} , Kraaij  R. C. \thanks{Delft Institute of Applied Mathematics, Delft University of Technology, Mekelweg 4, 2628 CD Delft, The Netherlands. \emph{E-mail address}: r.c.kraaij@tudelft.nl} , Tonon D.\thanks{Dipartimento di Matematica "Tullio Levi-Civita", Universit\`a degli Studi di Padova, via Trieste 63, 35121 Padova, Italy.}}
\date{\today}

\begin{document}

\maketitle

\tableofcontents

%
%\begin{abstract}

%	
%\noindent \emph{Keywords: Non-linear resolvent \and Hamilton-Jacobi equations \and Large deviations; \and Markov processes} 
%
%
%\noindent \emph{MSC2010 classification: primary: ; secondary:} 
%\end{abstract}
%

% primary
%	60F10  	Large deviations for stochastic processes
%	47H20  	Semigroups of nonlinear operators

% secondary 49J45, 49Q20, 58B20.
%   60J25  	Continuous-time Markov processes on general state spaces
%	60J35  	Transition functions, generators and resolvents
%	49L25  	Viscosity solutions for Hamilton-Jacobi equation

 \begin{abstract}

     Motivated by recent developments in the fields of large deviations for interacting particle systems and mean field control, we establish a comparison principle for the Hamilton--Jacobi equation corresponding to linearly controlled gradient flows of an energy function $\cE$ defined on a metric space $(E,d)$. 
     Our analysis is based on a systematic use of the regularizing properties of gradient flows in evolutional variational inequality (EVI) formulation, that we exploit for constructing rigorous upper and lower bounds for the formal Hamiltonian at hand and, in combination with the use of the Tataru's distance, for establishing the key estimates needed to bound the difference of the Hamiltonians in the proof of the comparison principle. Our abstract results apply to a large class of examples only partially covered by the existing theory, including gradient flows on Hilbert spaces and the Wasserstein space equipped with a displacement convex energy functional $\cE$ satisfying McCann's condition.

 \end{abstract}

\paragraph{Data availability} Data sharing not applicable to this article as no datasets were generated or analysed during the current study.

\section{Introduction}

The study of Hamilton--Jacobi (HJ) and related equations on infinite dimensional spaces is a flourishing research field. Such equations arise naturally in a great number of situations, including but certainly not limited to mean--field (or McKean--Vlasov) control problems, mean--field games and large deviation theory. This article is concerned with a specific class of infinite dimensional Hamilton--Jacobi equations having a common geometric structure that is typically encountered in the study of abstract versions of the so called Schr\"odinger problem (see {\cite{VR12, FJVR14,FN12,gentil2020dynamical,monsaingeon2020dynamical}} for some motivating examples) and in connection with large deviations theory \cite{FK06}. At the formal level, given a metric space $(E,d)$ where the metric $d$ is generated by a Riemannian metric $\ip{\cdot}{\cdot}$, the equation writes as
\begin{equation} \label{eqn:formal_HJ_eq}
    f- \lambda H f = h, \quad Hf:=-\ip{\grad f}{ \grad \cE} + \frac{1}{2} \vn{\grad f}^2
\end{equation}
where $\grad$ is the gradient associated with $\ip{\cdot}{\cdot}$. A fundamental example where equation \eqref{eqn:formal_HJ_eq} arises naturally in applications is that of the Wasserstein space $(E,d)=(\cP_2(\R^d),W_2(\cdot,\cdot))$ equipped with an energy functional $\cE$ satisfying McCann's condition. In this case, the underlying formal Riemannian metric is the so called Otto metric \cite{Ot01}. Equation \eqref{eqn:formal_HJ_eq} is expected to characterize the value function of the control problem 
\begin{equation}
    \sup\left\{ \int_0^{+\infty} e^{-\lambda^{-1} t}[\lambda^{-1} h(\rho^u(t))-\frac{1}{2}\vn{u(t)}^2\big]\De t: \dot{\rho}^{u} = -\grad \cE(\rho^{u}) + u, \,\rho^u(0)=\rho_0\right\},
\end{equation}
which can be interpreted as the problem of steering the gradient flow 
\begin{equation*}
    \dot{\rho} = -\grad \cE(\rho)
\end{equation*}
in such a way that an optimal balance is struck between the cost of controlling, modeled through the term $-\frac{1}{2}\vn{u(t)}^2$, and the reward obtained, modeled by the term $\lambda^{-1}h(\rho^u(t))$.
The above control problem can be written in the equivalent form
\begin{multline*}
    \sup\left\{ \int_0^{+\infty} \lambda^{-1} e^{-\lambda^{-1} t}\left[h(\rho^u(t))- \int_0^t \frac{1}{2}\vn{u(s)}^2 \De s \right] \De t \, \middle| \right. \\
    \left. \phantom{\int} \dot{\rho}^{u} = -\grad \cE(\rho^{u}) + u, \,\rho^u(0) = \rho_0 \right\}.
\end{multline*}
that gains a natural interpretation in relation to the corresponding semigroup. 

{%\color{purple}
%REPLACEMENT OF LAST PARAGRAPH?

In this manuscript we prove a comparison principle for viscosity solutions of \eqref{eqn:formal_HJ_eq} that holds under mild assumptions, the most relevant one being the existence of a gradient flow for the energy functional $\cE$ in Evolutional Variational Inequality (EVI) formulation, see \eqref{item:ass_EVI} below. Since in most examples of interest one cannot make sense of $\grad \cE$ and the Riemannian metric cannot be rigorously constructed, following  \cite{Ta92,Ta94,CrLi94,Fe06,FeKa09,AmFe14,GaSw15,FeMiZi21} we argue, using \eqref{item:ass_EVI},  that the Hamilton-Jacobi equation \eqref{eqn:formal_HJ_eq} can be replaced by two equations in terms of two operators $H_\dagger$ and $H_\ddagger$ that serve as upper and lower bounds for the formal Hamiltonian in \eqref{eqn:formal_HJ_eq}. 

We then state a comparison principle in terms of the upper and lower bounds $H_\dagger$ and $H_\ddagger$ (see Definition \ref{definition:HdaggerHddagger}). Following \cite{Ta92,Ta94,CrLi94,Fe06} the test functions in the domains of $H_\dagger$ and $H_\ddagger$, contain, next to the squared metric, the non-regular Tataru distance. This distance is not easy to handle when proving the existence of viscosity solutions, nevertheless the comparison principle we state is already of large interest. A refinement of the comparison principle presented here, that will be helpful for the existence of solutions, and the existence of solution  itself will be published in subsequent articles.  We also present some meaningful examples of applications of our main result in particular to controlled gradient flows in the Wasserstein space. Further applications to controlled gradient flows in Riemannian manifolds and Hilbert spaces are also discussed.
}

%\paragraph{Schr\"odinger problem(s)}

\paragraph{Hamilton--Jacobi equations in infinite dimensional spaces}

%\giov{Relevant papers: Ambrosio Feng 13, Gangbo Mesz 20, Gangbo Tudor, Gangbo Schwiech (2x), Feng Katsoulaki, Feng Kurtz, Feng Schwiech, Feng Mikami Zimmer, Fang AoP, Lions Carda Delarue, Gozlan Roberto Samson, Wu Zhang(Ass 3.1), Furhman Cosso Bandini, Pham, papers by Richard}
The theory of viscosity solutions for Hamilton--Jacobi equations on infinite dimensional spaces was initiated by Crandall and Lions in a series of papers \cite{CrLi84,CrLi86,CL86,CrLi90,CrLi91,CrLi94} in the setting of Hilbert spaces or Banach spaces possessing the Radon-Nikodym property. Recent applications in large deviations \cite{FK06}, functional inequalities \cite{GoRoSa14}, statistical mechanics \cite{BeDSGaJLLa02,BeDSGaJLLa15}, and McKean-Vlasov control \cite{CaDe15} have motivated the development of a theory of viscosity solutions for Hamilton--Jacobi equations on metric spaces that are not necessarily Hilbert, and in particular over the space of probability measures endowed with a transport--like distance. 

A first approach to Hamilton--Jacobi equations on the space of probability measures exploits the possibility of lifting the space of probability distributions to the space of square integrable random variables in order to take advantage of the Hilbertian structure of the latter: we refer to \cite{BaCoFuPh19,PhWe18,BeGrYa20,CoGoKhPhRo21} for some results recently obtained following this method. A second approach is more intrinsic and consists of working directly at the level of the space of probability measures and developing all the relevant notions therein. One can perform this using either the linear derivative, as shown in \cite{burzoni2020viscosity} in the context of McKean-Vlasov control for jump processes, or relying on the notion of subdifferential provided by optimal transport \cite{AmGiSa08}. The connections between the intrinsic approach and the extrinsic notion of derivative obtained through the above mentioned lifting procedure have been clarified in \cite{GaTu19}. 

In this manuscript, we follow the intrinsic approach and in particular we build on the achievements of the research program carried out by Feng and his coauthors \cite{FK06,FeKa09,FeMiZi21}, who developed a technique to deal with equations whose geometric structure is the same as \eqref{eqn:formal_HJ_eq} in terms of upper and lower bounds. We combine this intrinsic approach with the use of the Tataru distance function, as a penalization function in Ekeland's variational principle. Such idea has been introduced in \cite{Ta92,Ta94} and then further refined in \cite{CrLi94,Fe06}. To the best of our knowledge, in this work we provide the first systematic implementation of Tataru's method in metric spaces that are not Hilbert: as a result, we can remove compactness assumptions on the sublevel sets of the energy $\cE$ and for metric balls. While postponing to the next paragraph a more accurate comparison of our results with the above mentioned works, we stress that several other important contributions \cite{AmFe14,gangbo2015existence,GaSw15,GaTu19,WuZh20} adopt the intrinsic approach to show well--posedness of Hamilton--Jacobi equations on metric spaces. In all these works it is assumed that the variations of the Hamiltonian w.r.t. the measure argument can be locally controlled by the metric $d$. Since we require very little from the energy functional $\cE$ beyond the existence of an EVI gradient flow, this assumption is systematically violated in most of the instances of \eqref{eqn:formal_HJ_eq} that we consider. This happens already in the basic example when $\cE$ is the relative entropy and $(E,d)$ is the Wasserstein space. It is worth mentioning that operating the formal change of variable $\tilde{f}=f-\cE$ and setting $\lambda =1$ allows to rewrite formally \eqref{eqn:formal_HJ_eq} in the form
 \begin{equation}\label{eq: change of var}
     f(\pi) - \frac 12 \vn{\grad f(\pi)}^2 +\cF(\pi) =0,  
 \end{equation}
$\cF(\pi)=\frac 1 2 \vn{\grad \cE(\pi)}^2+\cE(\pi)-h(\pi).$ This equation has been often studied in the literature on infinite dimensional Hamilton--Jacobi equations. However, our main geometrical assumption, that is formally equivalent to the semiconvexity of $\cE$, does not give the control on the growth of $\frac{1}{2}\vn{\grad \cE(\pi)}^2$ needed to successfully apply the techniques developed in the above mentioned references \cite{AmFe14,gangbo2015existence,GaSw15,GaTu19,WuZh20}.

%\todo[inline]{Remove the paragraph below, to replace by a much longer discussion below?}
%Our proof of the comparison principle combines some rather classical ingredients such as Ekeland's perturbed optimization principle and the duplication of variables technique, which becomes here a quadruplication to get around the fact that $\cE$ may well be infinite, together with some more original ones. A first one worth mentioning is the use of the Tataru distance (\cite{Ta92,Ta94}) as a penalization function in Ekeland's principle, that allows us to remove compactness assumptions both for the level sets of the energy functional $\cE$ and for metric balls: applications to the infinite dimensional setting seem to be limited to Hilbert spaces \cite{CrLi94,Fe06}. A second original element of our proof strategy is the systematic use of the properties of EVI gradient flows, in particular of their regularizing properties that include energy dissipation and distance contraction estimates. Indeed, gradient flows play a crucial role in: \emph{a)} Defining suitable upper and lower bounds for the formal Hamiltonian that depend on $\cE$ and $d$ only; \emph{b)} the construction of the Tataru distance and \emph{c)} developing all the necessary estimates for the proof of the main result, in particular to bound the difference of the Hamiltonians in the proof of the comparison principle (see e.g. Lemma \ref{lemma:key_estimate_drift} and Lemma \ref{lemma:key2}).

 \paragraph{Master equation and Mean Field Games} 
 
 The recent years have witnessed fundamental advances on the understanding of the master equation arising in the theory of  Mean Field Games, see \cite{CaDeLaLi19} and the recent works \cite{WuZh20,GaMe20,GaMeMoZh21,CaCiPo20} for a sample of the recent progresses. Such equation aims at characterizing the limiting behavior of Nash equilibria in the many players regime and it has been noticed \cite{bensoussan2013mean} that the master equation shares some properties with infinite dimensional Hamilton--Jacobi equations, and in particular with those characterizing the value function of McKean--Vlasov control problems. However, these two equations remain conceptually different as explained in \cite{carmona2013control}. For example, despite some analogies between the ``monotonicity" assumption that is typically imposed on the coefficients of the master equation and the geodesic convexity of the energy functional $\cE$ that underlies all our computations, these two geometrical assumptions are not directly related and enter the coefficients of the respective equations in a different way. In the recent article \cite{GaMe20}, the authors get past the classical monotonicity assumption and indeed obtain well posedness for the master equation by means of displacement convexity. Still, the equation considered there and \eqref{eqn:formal_HJ_eq} have a different nature.

\paragraph{Contribution of this work}

 Our methods are  largely inspired by ideas and techniques put forward in \cite{Ta92,Ta94,CrLi94,FK06,FeKa09,Fe06,DFL11,KrReVe19,FeMiZi21}, where comparison principles for \eqref{eqn:formal_HJ_eq} have been proven in different contexts. Apart from \cite{DFL11,KrReVe19}, in which exploiting a Riemannian viewpoint they prove comparison principles in finite dimensional setting, we summarize here the contributions of the other papers in treating infinite--dimensional versions of \eqref{eqn:formal_HJ_eq}.
\begin{itemize}
    \item The works \cite{Ta92,Ta94,CrLi94,Fe06} deal with quadratic Hamiltonians on Hilbert spaces where the drift is not given by a gradient flow, but rather by a maximally dissipative operator $C$. (The subgradient of a proper lower semi-continuous convex functional is maximally dissipative, connecting the two equations.) We formally have
    \begin{equation} \label{eqn:Hamiltonian_Hilbert_intro}
        Hf(\pi) :=\ip{\grad f(\pi)}{ C\pi} + \frac{1}{2} \vn{\grad f(\pi)}^2
    \end{equation}
    Due to the non-compactness of the space, optimizers in the proof of the comparison principle are produced using Ekeland's variational principle. As the drift term arises from a (non-continuous) operator, the standard Hilbertian norm is not suitable to be used as a penalization function in Ekelands principle. Thus, a new metric-like object is introduced that is constructed from the norm in combination with the flow generated by $C$. A second innovation in this collection of papers concerns how to deal with $C$ in giving rigorous understanding to the Hamiltonian in \eqref{eqn:Hamiltonian_Hilbert_intro}. Working for test functions of the type $f(\pi) =  \frac{a}{2}\vn{\pi-\mu}^2$, the drift term equals
    \begin{equation} \label{eqn:intro_dissipativity_ub_1}
        \ip{\pi-\mu}{C\pi}
    \end{equation}
    which is ill-defined if $\pi$ is not in the domain of $C$. However, using the dissipativity of $C$, this term can be upper bounded by 
    \begin{equation}\label{eqn:intro_dissipativity_ub_2}
        \ip{\pi-\mu}{C\mu}
    \end{equation}
    which is well defined as long as $\mu \in \cD(C)$. This leads to a candidate upper bound for $H$, as long as one restricts to test-functions of metric type with $\mu \in \cD(C)$. A similar procedure can be carried out to obtain a lower bound. Working with test functions with restrictions on $\mu$ makes it necessary to replace the standard duplication of variables by a quadruplication, where the two new variables take their values in $\cD(C)$. The inclusion in $\cD(C)$ for these two new variables is enforced by the addition of two new penalization terms. This procedure is to some extent analogous to the procedure that, in finite dimensional cases, forces the variables to take their values in a compact set. 
    \item Building upon the works above, \cite{FeKa09} introduces a more intrinsic point of view replacing $C$ by the gradient of some energy functional $\cE$. In \cite{FeKa09} this is carried out for an energy functional $\cE$ with compact sub-levelsets on a metric space. The inclusion in the domain of the gradient of $\cE$ is now achieved by penalization with $\cE$, whereas in the papers above, considered in the context that $C = - \grad \cE$, the penalization can be interpreted as the square root of a Fisher information. This geometric interpretation effectively leads to much cleaner estimates.
    A second notable difference to the papers above is that the quadruplication is replaced by a duplication of variables. This comes at the cost of working with less-regular test functions in the domain of the Hamiltonians. To obtain existence of solutions, one typically starts out with more regular test-functions. It was shown in e.g. the three examples of Section 13 of \cite{FK06} that for well-posedness one can connect the regular and non-regular Hamiltonians by performing an inf- and sup-convolutions on sub- and supersolutions respectively. This is reminiscent of the techniques used in the proof of the comparison principle for second order equations on finite dimensional spaces, see e.g. \cite{CIL92}, and implies that the full procedure to obtain the comparison principle can be seen as one that involves a quadruplication just like in the papers mentioned above.
In the example of Section 13.3 of \cite{FK06}, studying the controlled heat flow in the Wasserstein space, it is observed that the upper and lower bound that in \eqref{eqn:intro_dissipativity_ub_1} and \eqref{eqn:intro_dissipativity_ub_2} were given by the use of the dissipativity of the operator are now replaced by the use of an inequality that we recognize in our more general context as the evolutional variational inequality.
    
    %\sout{It is immediately assumed that the upper and lower bound are given in terms of test functions that include the energy functional to enforce optimizing variables to be in some compact set. This is an improvement over the works above, where in the context that $C = - \grad \cE$, the penalization can be interpreted as a penalization in terms of the square root of the Fisher information. The geometric structure, and the use of $\cE$ as a penalization, thus leads to much cleaner estimates.}
    
    %^\sout{The paper, in contrast to the papers above, takes a slightly different starting-point. It is assumed that the test-functions in the upper and lower bound already include the penalization $\cE$, and that the estimate on the difference of Hamiltonians is assumed to be sufficiently tight to obtain a comparison principle. The assumed bound is inspired by formal computations and, hence, expected to hold in a wide range of examples.}
    
    %\sout{Thus, to obtain a full comparison principle, starting from test functions build up from the metric as in the first set of papers, one needs to perform a sup- and inf-convolution first, which is reminiscent to the procedures in the context of second order equations on finite dimensional spaces, see e.g. \cite{CIL92}. In addition, one needs to verify the estimate on the difference of the Hamiltonians in specific contexts. These steps are non-trivial and use variational inequalities, as seen e.g. in Chapter 13 of \cite{FK06}. This effectively turns the procedure in a variant of the quadruplication method mentioned above.}

    \item In \cite{FeMiZi21}, the authors study the controlled Carleman equation. In this context the Hamiltonian is associated to the gradient flow of the entropy on the space of probability measures considered as a subset of an inverse Sobolev space.  In this paper, a combination of the ideas above has been put to work, the procedure that involves a quadruplication, as above, in the sense of a standard duplication in combination with sup- and inf-convolutions, uses compactness of the sublevelsets of the energy. Also in this work, an inequality is used that we recongize as the evolutional variational inequality with contractivity constant $\kappa = 0$. %\ri{\sout{The upper and lower bound that in \eqref{eqn:intro_dissipativity_ub_1} and \eqref{eqn:intro_dissipativity_ub_2} were given by the use of the dissipativity of the operator are now replaced by the use of an inequality that we, in this paper, recognize in our more general context as the evolutional variational inequality with contractivity constant $\kappa = 0$.}}
\end{itemize}

In view of the above works, we combine their strengths  and assemble the key idea's in a single unifying framework:
\begin{itemize}
    \item We work with a geodesic metric space, where $\cE$ and $d$ do not necessarily have compact sublevel sets. In fact, we will allow $\cE$ that are unbounded from below. 
    \item We replace the variational inequalities used in the papers above by the systematic use of the evolutional variational inequality \eqref{item:ass_EVI}. This inequality is the generalization of the one used in \cite{FeMiZi21} and in a Hilbertian context implies the dissipativity of the operator $C$. Interpreting the variational inequalities used in the literature in the context of EVI, they correspond to the evolutional variational inequality with contractivity constant $\kappa = 0$. We will allow for negative $\kappa$ also. 
    \item We generalize the Tataru distance from Hilbert to general metric spaces and modify the distance to allow its application to gradient flows satisfying EVI with a negative contractivity constant $\kappa$.
    \item Instead of establishing the comparison principle via the duplication of variables combined with sup- and inf-convolutions, we perform the quadruplication of variables in a single go and introduce an argument generalizing the classical Lemma 3.1 of \cite{CIL92}.
\end{itemize}

To summarize: the key innovation in our proof strategy is the systematic use of the properties of EVI gradient flows, in particular of their regularizing properties that include energy dissipation and distance contraction estimates. Indeed, gradient flows play a crucial role in: \emph{a)} Defining suitable upper and lower bounds for the formal Hamiltonian that depend on $\cE$ and $d$ only; \emph{b)} the construction of the Tataru distance and \emph{c)} developing all the necessary estimates for the proof of the main result, in particular to bound the difference of the Hamiltonians in the proof of the comparison principle (see e.g. Lemma \ref{lemma:key_estimate_drift} and Lemma \ref{lemma:key2}). Apart from our key assumptions on the properties of the geodesic metric space $E$ and the existence of a gradient flow satisfying the evolution variational inequality, which can be considered to be standard in the literature, we assume in Assumption \ref{assumption:regularized_geodesics} that the directional derivative of the energy functional along 'regularized geodesics' can be controlled by the local slope of the energy.
Thanks to the rather soft assumptions needed for our main results to apply, we are able to cover natural situations that, to the best of our knowledge and understanding, fall out of the range of existing techniques. Leaving all precise statements to section \ref{sec: examples} below, we would like to mention  that one novelty is that we can treat the case of the Wasserstein space equipped with a R\'eny entropy as energy functional: in this setting the underlying gradient flow is the porous medium equation \cite{Ot01}. Even if we restrict to the more classical setting where the energy functional is the sum of the Boltzmann entropy, a potential energy and an interaction energy, existing results (see e.g. \cite{FK06,FeKa09}) require the confining potential to grow superquadratically at infinity in order to be applied, and several further restrictions are imposed on the interaction potential. Here, we allow for much more flexibility on both potentials. It is also plausible that the class of distances introduced in \cite{dolbeault2009new} leads to Hamilton-Jacobi equations whose uniqueness can be established by means of Theorem \ref{theorem:comparison_principle_tildeoperators} though we leave it to future work to validate this conjecture, as well as to enlarge the range of applications of the comparison principle proven in this paper.

%\sout{This property seems to be non-standard, and we do not fully oversee its consequences. In key examples, like in Riemannian manifolds or Hilbert spaces it seems to be implied by the at-most single-valuedness of the subgradient of the energy. As we do not assume that the subgradient exists everywhere, our assumption is significantly weaker than continuous differentiability, but due to the relation to single-valuedness, it carries a certain resemblance to differentiability. More research to work out the exact meaning of this assumption would be worthwile.}

\paragraph{Organization} The article has the following structure: in Section \ref{sec: mainres} we state our hypothesis and then proceed to the presentation of our main results. In section \ref{sec: proof of comparison} we prove Theorem \ref{theorem:comparison_principle_tildeoperators}, that is the comparison principle for the upper and lower bounds $H_{\dagger}$ and $H_{\ddagger}$.  Section \ref{sec: examples} is devoted to examples of applications, whereas Section \ref{EVI-Tataru} reports on the fundamental properties of EVI gradient flows and the Tataru distance.   Appendix \ref{appendix A} contains some background material on Ekeland's principle and Hamilton Jacobi equations.
\paragraph{Frequently used notation}
 \begin{itemize}
 \item $B_R(\rho)$ the ball of radius $R$ centered at $\rho;$
     \item $\bar{\mathbb{N}} =\mathbb{N}\cup \{+\infty\}$;
     \item $USC(E),LSC(E)$: space of upper semi-continuous and lower semi-continuous functions over $E$;
     \item $C(E)$ continuous and bounded functions over $E$;
     %\item $\mathcal{C}_{\infty,b}((0,\infty)^{h};\mathbb{R})$ $C^\infty$ bounded functions from $(0,\infty)^{h}$ to $\R$.  
 \end{itemize}
 
\section{The comparison principle}\label{sec: mainres}
Our aim is to establish a comparison principle for viscosity solutions of equations  of the form
\begin{equation} \label{eqn:formal_HJ_eq1}
    f(\pi) - \lambda H f (\pi) = h (\pi), \quad \pi\in E
\end{equation}
%and the evolutionary version
%\begin{equation*}
 %   \partial_t f(t,\pi) - H f (t,\pi) = 0, \quad t \in \R, \pi\in E,
%\end{equation*}
where $(E,d)$ is a complete  metric space, $\lambda>0$ is a constant, $h$ a real function on $E$ and the action of the formal Hamiltonian $H$ is given by
\begin{equation} \label{eqn:formal_H}
Hf(\pi) = -\ip{\grad_\pi f(\pi)}{ \grad_\pi \cE(\pi)} + \frac{1}{2} \vn{\grad_\pi f(\pi)}^2,
\end{equation}
where $\cE: E \to (-\infty, +\infty]$ is some energy functional and gradients are taken w.r.t. a formal Riemannian structure on $E$. Various issues arise with the definition of $H$ due to the presence of $\grad_\pi$. Indeed a precise notion of gradient for $\cE$ is difficult or impossible to give. For example, when $(E,d)$ is the Wasserstein space $(\cP_2(\R^d),W_2(\cdot,\cdot))$, in typical situations of interest, $\cE$ is worth $+\infty$ on a dense set and nowhere differentiable, even though the subdifferential is well defined and non empty on a subset of the domain of $\cE$. The lack of differentiability of entropic functionals is a well known issue in the theory of gradient flows and has led to the development of notions of gradient flows that do not appeal to $\mathrm{grad}_{\pi}\cE$ directly: we refer to \cite{AmGiSa08}  for a comprehensive overview. In a certain sense, we adopt a similar strategy: instead of working with $H$ directly, we construct suitable upper and lower bounds $H_{\dagger}$ and $H_{\ddagger}$, that depend on $\cE$ rather than its gradient and that are tight enough for the comparison principle to hold. To construct the upper and lower bounds we partially rely on ideas put forward in \cite{FeKa09,FeMiZi21} and draw inspiration from the EVI formulation of gradient flows which allows to put the considerations made therein on some important examples into a considerably more general framework. For example, an important with these work is that here we do not assume that the level sets of $\cE$ are compact. Let us now proceed to introduce the most important concepts needed to properly define $H_{\dagger}$ and $H_{\ddagger}$.

%To be able to  write everything down in a correct way, we will work in the general framework of a Hamilton-Jacobi problem with a drift given by a gradient flow. 

%The setting of Section \ref{section:MF_schrodinger} is our main example, but the methods below apply to more general settings.

\subsection{EVI-gradient flows and statement of the main hypotheses} \label{section:general_framework}

%\todo{Rewrite this}

%Our aim is to prove a comparison principle for an Hamilton-Jacobi equation of type \eqref{eqn:formal_HJ_eq1} stated on $(E,d),$ a complete  metric space.  We start by stating here the main necessary assumptions. 

We work on a complete metric space $(E,d)$ on which an extended functional $\cE : E\to (-\infty,+\infty]$ is defined. In the sequel, we shall refer to $\cE$ as to the energy, or entropy depending on the context. The next definition is that of local slope given in the first chapter of \cite{AmGiSa08}.
\begin{definition} {Let $\phi: E\to (-\infty,+\infty]$ be an extended functional with proper effective
domain, i.e. 
$\cD(\phi):=\{ \pi\in E: \phi(\pi)<+\infty \}\neq \emptyset$.}
Then the local slope of $\phi$ at {$\rho\in \cD(\phi)$} is defined as 
\begin{equation*}
    |\partial \phi|(\rho):= \begin{cases} \limsup_{\pi \rightarrow \rho} \frac{(\phi(\rho)-\phi(\pi))^+}{d(\rho,\pi)}, &\quad\mbox{if $\phi(\rho)<+\infty.$ }  \\
    +\infty, &\quad \mbox{otherwise.}
    \end{cases}
\end{equation*}

\end{definition}

Next, we define geodesic spaces.

\begin{definition} $(E,d)$ is a geodesic space, if for any $\rho,\pi\in E$ there exists a curve  $(\geo{\rho}{\pi}(t))_{t\in[0,1]}$ such that $\geo{\rho}{\pi}(0)=\rho,\geo{\rho}{\pi}(1)=\pi$ and for all $s,t\in[0,1]$
		\begin{equation}\label{eq: geodesicproperty}
		    d(\geo{\rho}{\pi}(s),\geo{\rho}{\pi}(t))=|t-s|d(\rho,\pi).
		\end{equation}
		Such a curve will be called \emph{geodesic}.
\end{definition}

\begin{assumption}[Metric and energy] \label{assumption:distance_and_energy}
We make the following assumptions of the complete metric space $(E,d)$ and the energy functional $\cE$.
\begin{enumerate}[(a)]
		\item   $(E,d)$ is a geodesic space. 
		\item We assume that the energy functional $\cE : E \to { (-\infty, +\infty]}$ is an extended functional such that: 
        \begin{itemize}
            \item  It has a proper effective domain, i.e. $\cD(\mathcal{E}):=\{ \pi\in E:\mathcal{E}(\pi)<+\infty \}\neq \emptyset
             $.   
            \item It is lower semi-continuous. 
        \end{itemize}
\end{enumerate}
\end{assumption}

Our second main assumption is the existence of an EVI gradient flow of $\cE$. The EVI (Evolutional Variational Inequality) formulation is the strongest formulation of gradient flows in metric spaces, we refer to the monograph \cite{AmGiSa08} and the more recent article \cite{MuSa20} for an extensive study of this notion.

\begin{definition}
    Given  $\kappa\in \R$, we define   \emph{solution of the $EV\!I_k$ inequality}   a continuous curve  $\gamma:{[}0,+\infty)\to {E}$  such that ${\gamma((0,+\infty))\subseteq \cD(\cE)}$ and 
	    %\todo{When we write EVI we should be careful about taking time derivatives: I used upper right derivative which always exists and seems to be classical way of writing EVI} 
	    for all $\rho\in E$
		 \begin{equation}\label{item:ass_EVI}\tag{$EV\!I_{\kappa}$}
		\frac{1}{2} {\frac{\dd^+}{\dd t}} \left(d^2(\gamma(t),\rho)\right) \leq \cE(\rho) - \cE(\gamma(t)) - \frac{\kappa}{2} d^2(\gamma(t),\rho),\quad \forall \rho \in \cD(\cE),t\in [0,+\infty).
		\end{equation}
		Here $\frac{\dd^+}{\dd t}$ denotes the upper right time derivative.
		
		%\todo{R: only used in Lemmas \ref{lemma:key1} and \ref{lemma:distance_contracting_under_gradient_flow}? Perhaps it thus suffices to work with $\pi$ such that $\cE(\pi) < \infty$?}
		  %\todo{R: Only used in approx of non-smooth hamiltonian by smooth ones? }
		  
		 An \emph{ $EV\!I_k$ gradient flow} of $\mathcal{E}$  defined in $D\subset \overline{\mathcal {D} (\cE)}$ is a family of continuous maps $S(t): D\to D, t\geq 0$  such that for every $\pi\in D$:
		 \begin{itemize}
	    \item The semigroup property holds 
	    \begin{equation}\label{ass:semigroup ppty and continuity} 
	    S[\pi](0)=\pi,\quad S[\pi](t+s)=S[S[\pi](t)](s) \quad \forall t,s\geq0.
	    \end{equation}
	    \item The curve $ (S[\pi](t))_{t{\geq}0}$ is a solution to \ref{item:ass_EVI}.
	    \end{itemize}
		   We shall refer to $(S[\pi](t))_{t\geq 0 }$ as  the \emph{gradient flow} of $\mathcal{E}$ started at $\pi$. %The gradient flow is completely characterized by EVI. 
		%\item \label{item:ass_flow_information} \textbf{[Gradient flow, connection to information]} 
		To lighten the notation, from now on, we will denote with $(\pi(t))_{t\geq 0}$ the gradient flow $(S[\pi](t))_{t\geq 0}$. 
\end{definition}

\begin{assumption} \label{assumption:gradientflow}
	 \textbf{[Gradient flow and  EVI]}  
	    We assume the existence of an \ref{item:ass_EVI} gradient flow of $\cE$ defined on  $D=E$. 
\end{assumption}
\begin{remark}
Note that the above assumption implies that $\overline{\mathcal {D} (\cE)}=E$.
\end{remark}

\ref{item:ass_EVI} is known to have several important consequences (see \cite{MuSa20}), including uniqueness of the gradient flow. Some of these facts, gathered at Lemma \ref{lem:EVI_properties}, play a crucial role in the proofs of our main results.
\begin{remark}
Note that the Hamiltonian is formally equivalent to 
\begin{equation} \label{eqn:lessformal_H}
Hf(\pi) =  \frac{\dd^+}{\dd t} \left( f (\pi(t)) \right)|_{t=0}+ \frac{1}{2} |\partial  f|^2(\pi).
\end{equation}
for $f:E\to (-\infty,+\infty)$ and $\pi\in E$. This representation is an important guideline for the construction of the lower and upper bounds.
\end{remark}

{For later use, we define } the information functional as the squared slope of the energy. 

\begin{definition} \label{definition:information}
    We define  the  \textit{information functional} $I: E \to [0, +\infty]$ as 
	    %$$
	    %I(\pi) : =
		   %    |\partial \cE(\pi)|^2 .
	    %$$
		\begin{equation*}
		   I(\pi) : = \left\{\begin{array}{cc}
		      |\partial \cE|^2(\pi)  & \pi\in \cD(\cE) \\
		       +\infty  & \text{otherwise}
		    \end{array} \right..
		\end{equation*}
		 %We assume that the information functional $I$ is lower semi-continuous.
\end{definition}
The information functional is closely related to the gradient flow via the energy identity
\begin{equation*}
    \cE(\pi(t))-\cE(\pi(0))=-\int_0^t I(\pi(s))\De s,
\end{equation*}
see Lemma \ref{lem:EVI_properties} for a rigorous version of the above relation.

Our final condition is of non-standard nature. We assume that any geodesic can be approximated as well as needed with a smoother curve, typically but not necessarily another geodesic,  along which the variations of $\cE$ can be controlled with the slope. This last requirement is coherent with the interpretation of the metric slope as the norm of the gradient of $\cE$. Note that, in most examples of interest, \eqref{eq: energy directional derivative} below fails to be true if we replace $\geo{\rho}{\pi}_{\theta}(t)$ with an arbitrary geodesic and that in the infinite dimensional context this assumption is considerably weaker than the existence of directional derivatives of $\cE$ along arbitrary geodesics.

\begin{assumption} \label{assumption:regularized_geodesics}
     For any $\rho,\pi\in E$ satisfying  $I(\rho)+\cE(\pi)<+\infty$, there exist a geodesic $\geo{\rho}{\pi}$ such that, for any $\theta>0$, there exists $\tau>0$ and a curve, not necessarily a geodesic, $(\geo{\rho}{\pi}_{\theta}(t))_{t\in[0,\tau]}$ , satisfying  
            \begin{equation}\label{eq: angle condition} 
            \limsup_{t \downarrow 0}  \frac{d(\geo{\rho}{\pi}_{\theta}(t),\geo{\rho}{\pi}(t))}{t} \leq \theta ,\quad 
            \end{equation} 
            and 
            \begin{equation}\label{eq: energy directional derivative}
            \liminf_{t\downarrow 0} \frac{\cE(\geo{\rho}{\pi}_{\theta}(t))-\cE(\rho)}{t}\leq |\partial \cE |(\rho)(d(\rho,\pi)+\theta).
            \end{equation}
            Note that \eqref{eq: angle condition} implies that $\geo{\rho}{\pi}_{\theta}(0)=\rho.$ 
           
\end{assumption}

We refer to \eqref{eq: angle condition} as to the \emph{angle condition}. \eqref{eq: energy directional derivative} can be interpreted as controllability of directional derivatives of regularized geodesics by the local slope.

\subsection{A first attempt at defining upper and lower bounds}

In light of the previous discussion, we can start developing a correct formulation of the Hamilton-Jacobi equation. In classical proofs of the comparison principle for first order Hamilton--Jacobi equations one needs to apply the Hamiltonian to distance--like test functions. 
In the following lines, ignoring all the technical issues, we shall derive a formal upper bound for $\pi\mapsto H d^2(\cdot,\rho)(\pi)$ arguing on the basis of \ref{item:ass_EVI} and on the following 
(formal) property of the distance
\begin{equation}  \label{eqn:ass_formal_noise_input}
		\forall \pi,\rho\in E \quad \left|\partial\left( \frac{1}{2} d^2(\cdot,\rho)\right)\right|^2(\pi) = d^2(\pi,\rho),
		\end{equation}
 	%Here, $|\partial \left( \frac{1}{2} d^2(\pi,\rho)\right)|=|\partial \left( \frac{1}{2} d^2(\cdot,\rho)\right)|(\pi)$ is the local slope of the function $\pi\mapsto \frac{1}{2} d^2(\pi,\rho),$ where $ d^2(\cdot,\rho):E\to [0,+\infty)$.  
		where $\left|\partial \left( \frac{1}{2} d^2(\cdot,\rho)\right)\right|(\pi)$ is the slope of the function $\frac{1}{2}d^2(\cdot,\rho)$ evaluated at $\pi$. Note that the above equation holds in the case of a smooth Riemaniann manifold. Let us now consider a test function $f^\dagger : E \to \R$  that is given in terms of the squared distance as $f^\dagger(\pi) = \frac{1}{2} a d^2(\pi,\rho)$ for some $\rho\in E$ and $a>0$. Applying formally the representation of $H$ from \eqref{eqn:lessformal_H} and using the property \eqref{eqn:ass_formal_noise_input} (as if $\pi\in \cD(\cE)$), we obtain that
\begin{equation*}
Hf^\dagger(\pi) =  \frac{1}{2} a\frac{\dd^+}{\dd t} \left( d^2 (\pi(t), \rho) \right)\Big|_{t=0} + \frac{1}{2} a^2 d^2(\pi,\rho).
\end{equation*}
Then, applying (formally)  Assumption \ref{assumption:gradientflow}  and being $a > 0$, we get 
\begin{equation*}
Hf^\dagger(\pi) \leq  a\left[ \cE(\rho) - \cE(\pi) \right] -  {a\frac{\kappa}{2} d^2(\pi,\rho) } + \frac{1}{2} a^2 d^2(\pi,\rho).
\end{equation*}
Let us note that this upper bound is proper as soon as $\cE(\rho)<+\infty$, so that the right hand side is well defined, even though it may take the value $-\infty$.
 Therefore, we are led to a candidate definition for a first upper bound $H_{\text{can},\dagger}$: its domain is 
\begin{equation*}
\cD(H_{\text{can},\dagger}) := \left\{ f^\dagger : E \to \R, \ f^\dagger(\pi) = \frac{1}{2} a d^2(\pi,\rho) \, \middle| \, \forall \, a > 0, \forall \, \rho\in E: \, \cE(\rho) < \infty \right\}
\end{equation*}
and for $f^\dagger(\pi) = \frac{1}{2}a d^2(\pi,\rho)$ we define our candidate Hamiltonian via
\begin{equation*}
H_{\text{can},\dagger} f^\dagger(\pi) := a\left[ \cE(\rho) - \cE(\pi)\right] - a\frac{\kappa}{2} d^2(\pi,\rho) + \frac{1}{2} a^2 d^2(\pi,\rho).
\end{equation*}

Similarly, we get a formal lower bound for a test function $f^\ddagger : E \to \R$ defined as $f^\ddagger(\mu) = - \frac{1}{2}a d^2(\gamma,\mu)$, $\gamma \in \mathcal D(\mathcal E)$. Let 
\begin{equation*}
\cD(H_{\text{can},\ddagger}) := \left\{f^\ddagger : E \to \R, \ f^\ddagger(\mu) = - \frac{1}{2} a d^2(\gamma,\mu) \, \middle| \, \, a > 0,  \, \gamma\in E: \, \cE(\gamma) < \infty \right\}
\end{equation*}
be the corresponding domain then for $f^\ddagger(\mu) = - \frac{1}{2}a d^2(\gamma,\mu)$ we set
\begin{equation*}
H_{\text{can},\ddagger} f^\ddagger(\mu) = a\left[ \cE(\mu) - \cE(\gamma)\right] + a\frac{\kappa}{2} d^2(\gamma,\mu) + \frac{1}{2} a^2 d^2(\gamma,\mu).
\end{equation*}

Thus, instead of establishing the comparison principle for  equation \eqref{eqn:formal_HJ_eq1}, we aim to show it for the upper and lower bound we found for our Hamiltonian, i.e. we would like to show that for every subsolution  $u$ (in a sense to be precised)  of
\begin{equation*}
    f - \lambda H_{\text{can},\dagger} f = h
\end{equation*}
and every supersolution  $v$ (in a sense to be precised) of
\begin{equation*}
    f - \lambda H_{\text{can},\ddagger} f = h
\end{equation*}
we have $u \leq v$.
Thanks to the formal inequalities this result would give a formal comparison principle for equation \eqref{eqn:formal_HJ_eq}.

\smallskip

The standard procedure to prove the comparison principle consists in using a  doubling  variables method. However, when doing this with our candidate Hamiltonian, we run into the known issue that optimal values are not attained, essentially because we are working in a infinite dimensional space. This  issue is usually solved via Ekeland's variational principle (a version of which, the one used in this article, is Lemma \ref{lemma:Ekeland}, in the appendix). Nevertheless, for our setting, in which the Hamiltonian contains an unbounded term, this is not enough. Indeed, once Ekeland variational principle gives us the unique optimizer, the standard procedure consists in finding good estimates for the difference of the Hamiltonians. Following \cite{CrLi94,Ta92,Ta94,Fe06}, we need to apply the Ekeland variational principle with the Tataru distance as a penalization function which, in contrast with the usual distance $d$ is Lipschitz along the gradient flow and allows for an efficient comparison of the difference between of the Hamiltonians. Let us now proceed to construct a version of the Tataru distance that is adapted to our scope.

\subsection{The Tataru distance}

The Tataru distance function, introduced in \cite{Ta92}, is given in terms of the gradient flow generated by the energy functional $\cE$ considered therein.
\begin{equation*}
    d_T(\pi,\rho) = \inf_{t \geq 0} \left\{ t + d(\pi,\rho(t))\right\}, \quad \forall \pi,\rho\in E,
\end{equation*}
where $\rho(\cdot)$ is the gradient flow of $\cE$ started at $\rho$. Note that $d_T$ is not a metric due to a lack of symmetry. The two key properties of the above Tataru distance are that $d_T$ is Lipschitz with respect to the metric $d$ and that it behaves well with respect to the corresponding gradient flow
\begin{equation*}
    \frac{d_T(\pi(r),\rho) - d_T(\pi,\rho)}{r} \leq 1, \quad \forall \pi,\rho\in E,
\end{equation*}
for all $r\in \R\setminus \{0\}$.

These properties are both based on the fact that the gradient flow  considered there was contracting with respect to the metric. In our setting, we consider  \eqref{item:ass_EVI} gradient flows and we allow negative values $\kappa$, i.e. a negatively curved space, and in this case the gradient flow is not anymore contracting. Thus, we have to work with an adjusted Tataru distance that takes care of all possible values of $\kappa$. 

%{REM: Since the gradient flow is defined only for starting points in $\overline{\cD(\cE)}$ the Tataru distance should  have a sencond entry restricted to $ \overline{\cD(\cE)}.$ (This was already the case in Tataru's and C-L papers)}

\begin{definition}
We define the Tataru distance {$d_T: E\times E \to [0,+\infty)$} with respect to the metric $d$ and energy $\cE$ as
    \begin{equation*}
    d_T(\pi,\rho) = \inf_{t \geq 0} \left\{ t + e^{\hat{\kappa} t} d(\pi,\rho(t))\right\}, \quad \forall \pi, \rho \in E,
\end{equation*}
where $\hat{\kappa} = (0 \wedge \kappa)\leq 0$.
\end{definition}

The precise statements and proofs of the main properties of Tataru distance are postponed to Section \ref{appendixTataru}.

\subsection{The comparison principle for a proper upper and lower bound}\label{bound}

Now that we have defined the Tataru distance we are ready to introduce the upper and lower bounds for $H$ for which we will actually establish the comparison principle. As we did before, we provide a heuristic argument to justify their definition. To do so, we begin by fixing a test function of the form

\begin{equation} \label{eqn:formal_computation_H_f0}
    f^\dagger(\pi) = \frac{1}{2}a d^2(\pi,\rho) + b d_T(\pi,\mu) + c
\end{equation}
for $a,b > 0$, $c \in \bR$, and $\rho,\mu\in E$. {As before, due to the presence of the  term $\frac{1}{2}a d^2(\pi,\rho)$, we will need to require that $\cE(\rho) < \infty$ in order to obtain a proper bound for the Hamiltonian.} In order to bound the action of $H$ on $f^\dagger$, we can rely again on the representation \eqref{eqn:lessformal_H} and  invoke the Lipschitzianity of $d_{T}$ along the gradient flow (Lemma \ref{lemma:estimates_Tataru}) that gives %we observe 
\begin{equation*}
  \Big| \frac{\dd^+}{\dd t} \left( d_T(\pi(t),\mu)\right)\big|_{t=0}\Big|\leq 1.
\end{equation*}
%is  formally  equal to the limit
%\begin{equation*}
 %   \lim_{r \downarrow 0} \frac{d_T(\pi(r),\mu) - d_T(\pi,\mu)}{r},
%\end{equation*}
%which by Lemma \ref{lemma:estimates_Tataru} is upper bounded by $1$. 
Similarly, as the Tataru distance is Lipschitz with respect to $d$, then any gradient of $d_T$ can be upper bounded by $1$. Using these two properties and applying formally \ref{item:ass_EVI} and \eqref{eqn:ass_formal_noise_input} as we did before to define $H_{\text{can},\dagger}$, we obtain that if $f^{\dagger}$ is as in \eqref{eqn:formal_computation_H_f0}:
\begin{align*}
Hf^\dagger(\pi) =&  \frac 1 2 a \frac{\dd^+}{\dd t} \left( d^2(\pi(t),\rho)\right)\big|_{t=0}  + b \frac{\dd}{\dd t} \left( d_T(\pi(t),\mu)\right)\big|_{t=0}
+\frac 12 \left |\partial\left(\frac{1}{2} a d^2(\cdot,\rho) + b  d_T(\cdot,\mu)\right)\right|^2(\pi)\\
\leq& a\left[ \cE(\rho) - \cE(\pi)\right] - a \frac{\kappa}{2} d^2(\pi,\rho) +b \\
& \qquad + \frac{1}{2} a^2 \left |\partial \left(\frac{1}{2} d^2(\cdot,\rho)\right)\right|^2(\pi) + \frac{1}{2} 2 a b \left |\partial\left(\frac{1}{2} d^2(\cdot,\rho)\right)\right|(\pi) \, \left |\partial d_T(\cdot,\mu)\right|(\pi) + \frac{1}{2}b^2\left |\partial d_T(\cdot,\mu)\right|^2(\pi) \\
\leq& a\left[ \cE(\rho) - \cE(\pi)\right] - a \frac{\kappa}{2} d^2(\pi,\rho) + b \\
& \qquad + \frac{1}{2} a^2 d^2(\pi,\rho) + ab  d(\pi,\rho)\left |\partial d_T(\cdot,\mu)\right|(\pi) + \frac{1}{2} b^2 \left |\partial d_T(\cdot,\mu)\right|^2(\pi) \\
 \leq &a\left[ \cE(\rho) - \cE(\pi)\right] - a \frac{\kappa}{2} d^2(\pi,\rho) + b  + \frac{1}{2} a^2 d^2(\pi,\rho) + ab d(\pi,\rho) 
+ \frac{1}{2} b^2.
\end{align*}
%In the above computation  we used formally \eqref{eqn:ass_formal_noise_input}, that $d_T$ has Lipschitz bound $1$, thus  
%\begin{equation*}
%\frac{\dd}{\dd t} \left( d_T(\pi(t),\mu)\right)|_{t=0}\leq 1,
%\end{equation*}
%and 
%begin{equation*}
%\left |\partial d_T(\pi,\mu)\right| ^2\leq 1.
%\end{equation*}

We can adapt this argument to test functions of the form 
\begin{equation*}
    f^\ddagger(\mu):= -\frac{1}{2}a d^2(\gamma,\mu) - b d_T(\mu,\pi) + c, \quad a,b>0, c\in \mathbb{R},
\end{equation*}
 by treating the term $\frac 12 \left |\partial\left(-\frac{a}{2}  d^2(\cdot,\gamma) - b  d_T(\cdot,\pi)\right)\right|^2(\mu)$ in a slightly different way, namely\footnote{In this computation we use the formal bound $|\partial (f+g)|\geq ||\partial f| -|\partial g||$. The local slope does not satisfy this property. In order to justify heuristically the upcoming calculations, it is convenient to think of it as a proxy for the norm of the gradient of $f+g$. } 
 \begin{equation*}
 \begin{split}
     \frac 12 \left |\partial\left(-\frac{a}{2}  d^2(\cdot,\gamma) - b  d_T(\cdot,\pi)\right)\right|^2(\mu)&
     \geq \frac12 \left(a\big|\partial\big(\frac{1}{2} d^2(\cdot,\gamma)\big)\big| - b  |\partial d_T(\cdot,\pi)|\right)^2(\mu)\\
     &= \frac{a^2}{2}d^2(\mu,\gamma)  -ab d(\mu,\gamma)|\partial d_T(\cdot,\pi)|(\mu)+ \frac{b^2}{2} |\partial d_T(\cdot,\pi)|^2(\mu)\\
     &\geq \frac{a^2}{2}d^2(\mu,\gamma)  -ab d(\mu,\gamma)|\partial d_T(\cdot,\pi)|(\mu)\\
     &\geq \frac{a^2}{2}d^2(\mu,\gamma)  -ab d(\mu,\gamma)
     \end{split}
 \end{equation*}
We are thus led to consider the following definition, in which we prefer to underline the fact that the Hamiltonians are operators.

\begin{definition} \label{definition:HdaggerHddagger}

\begin{enumerate}
    \item For each $a > 0, b > 0, c \in \bR$, and $\mu,\rho\in E : \, \cE(\rho) < \infty$ let $f^\dagger = f^\dagger_{a,b,c,\mu,\rho} \in C(E)$ and  $g^\dagger = g^\dagger_{a,b,c,\mu,\rho} \in USC(E)$ %\sout{$g = g_{a,b,c,\mu,\rho} \in C_b(E)$}} 
    be given for any $\pi \in E$ by
    \begin{align*}
        f^\dagger(\pi) & := \frac{1}{2}a d^2(\pi,\rho) + b d_T(\pi,\mu) + c \\
        g^\dagger(\pi) & := a\left[ \cE(\rho) - \cE(\pi)\right] - a \frac{\kappa}{2} d^2(\pi,\rho) + b  + \frac{1}{2} a^2 d^2(\pi,\rho) + ab d(\pi,\rho) 
+ \frac{1}{2} b^2.
    \end{align*}

    Then the  operator $H_\dagger \subseteq C(E) \times USC (E)$ is defined by
    \begin{equation*}
        H_\dagger := \left\{\left(f^\dagger_{a,b,c,\mu,\rho},g^\dagger_{a,b,c,\mu,\rho}\right) \, \middle| \, a , b>0,c\in \bR, \mu,\rho\in E : \, \cE(\rho) < \infty \right\}.
    \end{equation*}
    \item For each $a > 0, b > 0, c \in \bR$, and $\pi, \gamma\in E : \, \cE(\gamma) < \infty$ let $f^\ddagger = f^\ddagger_{a,b,c,\pi,\gamma} \in C(E)$ and $g^\ddagger = g^\ddagger_{a,b,c,\pi,\gamma} \in LSC(E)$  be given for any $\mu\in E$ by
    \begin{align*}
        f^\ddagger(\mu) & := -\frac{1}{2}a d^2(\gamma,\mu) - b d_T(\mu,\pi) + c \\
        g^\ddagger(\mu) & :=  a\left[ \cE(\mu) - \cE(\gamma)\right] + a \frac{\kappa}{2} d^2(\gamma,\mu) - b  + \frac{1}{2} a^2 d^2(\gamma,\mu) - ab d(\gamma,\mu){-\frac{1}{2} b^2}.
    \end{align*}

    Then the  operator $H_\ddagger \subseteq C(E) \times LSC(E)$ is defined by
    \begin{equation*}
        H_\ddagger := \left\{\left(f^\ddagger_{a,b,c,\pi,\gamma},g^\ddagger_{a,b,c,\pi,\gamma}\right) \, \middle| \, a,b>0,c\in \bR, \pi,\gamma\in E : \, \cE(\gamma) < \infty \right\}.
    \end{equation*}
\end{enumerate}

\end{definition}

We are now ready to provide the notion of solution we are going to work with, which we state for  general Hamiltonians $A_\dagger\subseteq LSC(E) \times USC (E)$ and $A_\ddagger\subseteq USC(E) \times LSC (E)$.

	\begin{definition} \label{definition:viscosity_solutions_HJ_sequences}
		 Fix $\lambda > 0$ and $h^\dagger,h^\ddagger \in C_b(E)$. Consider the equations
		\begin{align} 
		f - \lambda  A_\dagger f & = h^\dagger, \label{eqn:differential_equation_H1} \\
		f - \lambda A_\ddagger f & = h^\ddagger. \label{eqn:differential_equation_H2}
		\end{align}

		We say that $u$ is a \textit{(viscosity) subsolution} of equation \eqref{eqn:differential_equation_H1} if $u$ is bounded, upper semi-continuous and if for all $(f,g) \in A_\dagger$ there exists a sequence $(\pi_n)_{n\in \mathbb N}\in E$ such that
		\begin{gather}
		{\lim_{n \uparrow \infty} } \ u(\pi_n) - f(\pi_n)  = \sup_\pi u(\pi) - f(\pi), \label{eqn:viscsub1} \\
		{\limsup_{n \uparrow \infty}  }  \ u(\pi_n) - \lambda g(\pi_n) - h^\dagger(\pi_n) \leq 0. \label{eqn:viscsub2}
		\end{gather}
		We say that $v$ is a \textit{(viscosity) supersolution} of equation \eqref{eqn:differential_equation_H2} if $v$ is bounded, lower semi-continuous and if for all $(f,g) \in A_\ddagger$ there exists a sequence $(\pi_n)_{n\in \mathbb N}\in E$ such that
		\begin{gather*}
		{\lim_{n \uparrow \infty} } \ v(\pi_n) - f(\pi_n)  = \inf_\pi v(\pi) - f(\pi),  \\
		{\liminf_{n \uparrow \infty} } \ v(\pi_n) - \lambda g(\pi_n) - h^\ddagger(\pi_n) \geq 0.  
		\end{gather*}
		If $h^\dagger = h^\ddagger$, we say that $u$ is a \textit{(viscosity) solution} of equations \eqref{eqn:differential_equation_H1} and \eqref{eqn:differential_equation_H2} if it is both a subsolution of \eqref{eqn:differential_equation_H1} and a supersolution of \eqref{eqn:differential_equation_H2}.
		
		We say that \eqref{eqn:differential_equation_H1} and \eqref{eqn:differential_equation_H2} satisfy the \textit{comparison principle} if for every subsolution $u$ to \eqref{eqn:differential_equation_H1} and supersolution $v$ to \eqref{eqn:differential_equation_H2}, we have $\sup_E u-v \leq \sup_E h^\dagger - h^\ddagger$.
	\end{definition}

	In classical works on viscosity solutions, instead of working with the statement "there exists a sequence such that...", one has "for all optimizers one has...". 	However, when constructing our test functions in the comparison principle proof, we will work with the Ekeland variational principle, see Lemma \ref{lemma:Ekeland}. This principle will give us an optimizer that is also  unique. We will show in Lemma \ref{lemma:visc_sol_optimizers_sequence_to_point} that, for our specific test functions, we can work directly with the unique optimizer instead of passing through an optimizing sequence as if we were using the  stronger definition. On the other hand, Definition \ref{definition:viscosity_solutions_HJ_sequences} is easier to handle while showing existence of solutions.	We are ready to state the main result of this article.

\begin{theorem} \label{theorem:comparison_principle_tildeoperators}\textbf{[The comparison Principle.]} 
	Let Assumptions \ref{assumption:distance_and_energy}, \ref{assumption:gradientflow} and \ref{assumption:regularized_geodesics} be satisfied. Let $\lambda>0$ and  $h^\dagger,h^\ddagger: E\to \R$ be bounded and uniformly continuous.
	Let $u : E\to \R$ be a  viscosity subsolution to $f - \lambda H_\dagger f = h^\dagger$ and let $v: E\to \R$ be a viscosity supersolution to $f - \lambda H_\ddagger f = h^\ddagger$. Then we have 
	\begin{equation*}
	    \sup_{\pi \in E} u(\pi) - v(\pi) \leq \sup_{\pi \in E} h^\dagger(\pi) - h^\ddagger(\pi).
	\end{equation*}
	\end{theorem}

%\todo{to add there the relation to the operator without Tatarus.}

\begin{remark}
Note that we formally have
$$
H f \leq H_\dagger f \quad \text{and}\quad  H_\ddagger f  \leq H f .
$$
Thanks to these inequalities the above result will give a formal comparison principle for equation \eqref{eqn:formal_HJ_eq1}.
\end{remark}

\begin{remark} \label{remark:comparison_weakening_uniformly_continuous_h}
The assumption that $h^\dagger, h^\ddagger$ are uniformly continuous can be weakened to uniform continuity on sets of the type
\begin{equation*}
    K_{c,d}^\rho := \left\{\pi \in E \, \middle| \, d(\pi,\rho) \leq c, \cE(\rho) \leq d  \right\}.
\end{equation*}
\end{remark}

\section{Proof of Theorem \ref{theorem:comparison_principle_tildeoperators} }\label{sec: proof of comparison}
The proof of Theorem \ref{theorem:comparison_principle_tildeoperators} contains two main parts.  The first part consists in showing that, in order to establish the comparison principle, we can reduce to the usual estimation on the difference of $H_\dagger$ and $H_\ddagger$. The estimation of this difference, however, is non-trivial in the present context and we postpone to section \ref{sec:key_estimates} the proof of some of the key estimates needed there.

\begin{remark}
In Step 1 of the proof below, we first make use of the fact that $\cE$ can be bounded from below by a non-negative constant times $-d^2$. In this way, the standard quadruplication of variables, which goes with a penalization needed as we work with non-equal variables, is indeed a penalization. If $\cE$ is itself already bounded from below by $0$, we can simplify significantly the proof by choosing $c_1 = 0$. 
\end{remark}

\begin{proof}
Let $u$ be a subsolution of equation \eqref{eqn:differential_equation_H1} and $v$ a supersolution of equation \eqref{eqn:differential_equation_H2}, we have to prove that $$\sup_{\pi \in E} u(\pi) - v(\pi)$$ can be controlled by $$ \sup_{\pi \in E} h^\dagger(\pi) - h^\ddagger(\pi).$$ 
To proceed, as in the classical proof of the comparison theorem, one usually performs  the  doubling variables method, that can be done in our case using the distance function and the energy functional as penalization functions. However, the use of the energy functional and the fact that $\cE(\pi)$ could be worth $+\infty$ oblige us to introduce two additional variables, i.e. we  quadruplicate the number of variables. %Indeed, the terms containing the energy functional should appear as constant in the test functions.
This procedure is actually reminiscent of the sup-convolution procedure.

\underline{\emph{Step 1:} Quadruplication of variables and  Ekeland's principle.}

	We fix $\nu_0\in E$ such that $\cE(\nu_0) < \infty$, we need $\cE(\nu_0) < \infty$ and $c_1,c_2\in \mathbb R$ as in Lemma \ref{lem:EVI_properties} item \ref{item:EVI quadratic lower bound}, i.e. such that
    \begin{equation*}
        \inf_{\pi\in E} \cE(\pi) + \frac{c_1}{2} d^2(\pi,\nu_0) + c_2 = 0,
    \end{equation*}
    and we define
    $$
    \bar \cE(\pi):= \cE(\pi) + \frac{c_1}{2} d^2(\pi,\nu_0) + c_2.
    $$

 We fix $\alpha>0$ and $\varepsilon_{\alpha}$ small enough (this value has to be fixed according to the condition \eqref{eq:quadruplication_Ekeland_step1_1}, {i.e. $\Xi_{\alpha}(x_{\alpha,0})+\epsal < \alpha^{-1},$
	where  $x_{\alpha,0}=(\pi_{\alpha,0},\rho_{\alpha,0},\mu_{\alpha,0},\gamma_{\alpha,0})$ will be chosen later on and $\Xi_{\alpha}$ is defined as below}).

 We introduce for $x=(\pi,\rho,\mu,\gamma)\in E^4$

\begin{subequations}
    \begin{equation*}
        \Phi_{\alpha}(x) := \frac{u(\pi)}{1-\epsal}-\frac{v(\mu)}{1+\varepsilon_{\alpha}}
    \end{equation*}
    \begin{equation*}
        \Psi_{\alpha}(x):= \frac{d^2(\pi,\rho)}{2(1-\varepsilon_{\alpha})} + \frac{d^2(\rho,\gamma)}{2} +\frac{d^2(\gamma,\mu)}{2(1+\varepsilon_{\alpha})}
    \end{equation*}
    \begin{equation*}
        \Psi_{\alpha,0}(x):=\frac{1}{2(1-\epsal)}d^2(\pi,\mu)
    \end{equation*}
    \begin{equation*}
        \Xi_{\alpha}(x):= \frac{\epsal}{1-\epsal} \bar \cE(\rho)  +\frac{\epsal}{1+\epsal} \bar \cE(\gamma)
    \end{equation*}
\end{subequations}
Next, we define
\begin{subequations}
    \begin{equation}\label{eqn:Gdef}
        \cG_{\alpha}(x):=\Phi_{\alpha}(x)-\alpha \Psi_{\alpha}(x)-\Xi_{\alpha}(x), \quad M_{\alpha}:=\sup_{x\in E^4} \cG_{\alpha} (x)
    \end{equation}
    \begin{equation*}
         \cG_{\alpha,0}(x):=\Phi_{\alpha}(x)-\alpha \Psi_{\alpha,0}(x),\quad M_{\alpha,0}:=\sup_{x\in E^4} \cG_{\alpha,0}(x)
    \end{equation*}
\end{subequations}
and 
\begin{equation*} 
\begin{aligned}
    \mathcal{B}(x,\tilde x)  &  := \frac{1}{1-\epsal}d_T(\pi,\tilde \pi) +\frac{1}{1+\epsal} d_T(\mu,\tilde \mu)+ d_T(\rho,\tilde \rho)+d_T(\gamma,\tilde \gamma).
\end{aligned}
\end{equation*}

We gather the important results of this step in the following proposition, whose proof is postponed to section \ref{sec:proof_quadruplication_Ekeland}.

\begin{proposition}\label{prop:quadruplication_Ekeland}
For each $\alpha > 0$ we can find  $x_{\alpha}=(\pi_\alpha,\rho_\alpha,\mu_\alpha,\gamma_\alpha)\in E^4$ such that
\begin{enumerate}[(a)]
    \item \label{item: duplication}
    \begin{equation} \label{eqn:first_estimate_u-v}
    \sup_{\pi\in E} u(\pi) - v(\pi)   \leq \Phi_{\alpha}(x_{\alpha})+ \cO(\alpha^{-1/2}),
    \end{equation}
    \item \label{item:Ekeland_optimality} $\rho_\alpha,\gamma_\alpha \in \cD(\cE)$ and $x_\alpha$ is the unique point in $E^4$ such that
    \begin{equation} \label{eqn:Ekeland_optimality}
        \sup_{x\in E^4 } \cG_\alpha(x) - {\frac 12} \alpha^{-2} \leq \cG_\alpha(x_\alpha) = \sup_{x\in E^4 } \cG_\alpha(x) - \alpha^{-1} \cB_\alpha, 
    \end{equation}
    where 
    \begin{equation}\label{eqn:Bdef}
    \mathcal{B}_\alpha(x) := \mathcal{B}(x,x_\alpha).
    \end{equation}
    \item \label{item:Ekeland_uniqueness} If  $(x_n)_{n\in\mathbb N}\in E^4$ is such that 
    \begin{equation*}
    \lim_{n\to \infty} \cG_\alpha(x_n) - \alpha^{-1} \cB_\alpha(x_n) = \cG_\alpha(x_\alpha),
    \end{equation*}
    then $\lim_{n\to\infty} x_n = x_{\alpha}$.
    \item \label{item:CIL} We have 
    \begin{equation*}        \liminf_{\alpha \rightarrow \infty} \alpha \Psi_{\alpha}(x_{\alpha})+\Xi_{\alpha}(x_{\alpha})+ \varepsilon_{\alpha} d^2(\rho_{\alpha},\nu_0)+\varepsilon_{\alpha}d^2(\gamma_{\alpha},\nu_0)=0.
\end{equation*}
    \end{enumerate} 
 
\end{proposition}

\underline{\emph{Step 2:} Use of sub(super)solution properties.} {In the rest of the proof we consider a diverging  sequence $(\alpha_n)_{n\in\mathbb{N}}$ along which

  \begin{equation*}        \lim_{n \rightarrow \infty} \alpha_n \Psi_{\alpha_n}(x_{\alpha_n})+\Xi_{\alpha_n}(x_{\alpha_n})+ \varepsilon_{\alpha_n} d^2(\rho_{\alpha_n},\nu_0)+\varepsilon_{\alpha_n}d^2(\gamma_{\alpha_n},\nu_0)=0.
\end{equation*}
}

Consider as test functions $f^\dagger,f^\ddagger: E\to (-\infty,+\infty)$ given by
\begin{align}\label{eqn:def_f0}
        f^\dagger(\cdot) :& = -(1-\epsaln)\mathcal{G}_{\alpha_n}(\cdot,\mu_{\alpha_n},\rho_
        {\alpha_n},\gamma_{\alpha_n}) +u(\cdot) + (1-\epsaln)\alpha^{-1}_n \mathcal{B}_{\alpha_n}(\cdot,\mu_{\alpha_n},\rho_{\alpha_n},\gamma_{\alpha_n}), \\
         \nonumber f^\ddagger(\cdot) :&=(1+\epsaln)\mathcal{G}_{\alpha_n}(\pi_{\alpha_n},\cdot,\rho_{\alpha_n},\gamma_{\alpha_n})+v(\cdot)-(1+\epsaln)\alpha^{-1}_n \mathcal{B}_{\alpha_n}(\pi_{\alpha_n},\cdot,\rho_{\alpha_n},\gamma_{\alpha_n}).
\end{align}
Note that $f^\dagger,f^\ddagger$ are valid test functions, %i.e. $f^\dagger\in\mathcal{D}(H_{\dagger}),f^\ddagger\in\mathcal{D}(H_{\ddagger})$. Indeed, $f^\dagger,f^\ddagger\in C_b(E)$,
Indeed, from \eqref{eqn:Gdef},\eqref{eqn:Bdef} we have
\begin{equation*}
    \begin{split}
        f^\dagger(\pi) &= \frac{\alpha_n}{2}d^2(\pi,\rho_{\alpha_n}) + \alpha_n^{-1} d_T(\pi,\pi_{\alpha_n})+ \text{const.} ,\\
        f^\ddagger(\mu) &= -\frac{\alpha_n}{2}d^2(\mu,\gamma_{\alpha_n})- {\alpha_n}^{-1} d_T(\mu,\mu_{\alpha_n})+\text{const.},
    \end{split}
\end{equation*}
and we know that $\rho_{\alpha_n},\gamma_{\alpha_n}\in\mathcal{D}(\mathcal{E})$ by Proposition \ref{prop:quadruplication_Ekeland}-\ref{item:Ekeland_optimality}.  From the very definition of $f^{\dagger}$, we obtain
\begin{equation}\label{eqn:def_f0b}
u(\pi)-f^\dagger(\pi)=(1-\epsaln)[\mathcal{G}_{\alpha_n}-{\alpha_n}^{-1} \mathcal{B}_{\alpha_n}](\pi,\mu_{\alpha_n},\rho_{\alpha_n},\gamma_{\alpha_n}),
\end{equation} 
and $\pi_{\alpha_n}$ is the unique maximizer of $u(\pi)-f^\dagger(\pi)$ because of \eqref{eqn:Ekeland_optimality}. %indeed
%$$
%\sup_{\pi \in E} u(\pi)-f^\dagger(\pi)=\sup_{\pi \in E} (1-\epsaln)[\mathcal{G}_{\alpha_n}-{\alpha_n}^{-1} \mathcal{B}_{\alpha_n}](\pi,\mu_{\alpha_n},\rho_{\alpha_n},\gamma_{\alpha_n})\stackrel{\eqref{eqn:Ekeland_optimality}}{=} (1-\epsaln) \mathcal{G}_{\alpha_n}(\pi_{\alpha_n},\mu_{\alpha_n},\rho_{\alpha_n},\gamma_{\alpha_n})= u(\pi_{\alpha_n})-f^\dagger(\pi_\alpha). 
%$$
Analogously, we find 
$$
v(\mu)-f^\ddagger(\mu)=-(1+\epsaln)[\mathcal{G}_{\alpha_n}-{\alpha_n}^{-1} \mathcal{B}_{\alpha_n}](\pi_{\alpha_n},\mu,\rho_{\alpha_n},\gamma_{\alpha_n}),
$$
and $\mu_{\alpha_n}$ is the unique minimizer of $v(\mu)-f^\ddagger(\mu)$. Being $u$ a subsolution, there exists a sequence $(\pi_m)_{m\in\mathbb{N}}\in E$ satisfying \eqref{eqn:viscsub1} and \eqref{eqn:viscsub2}, for  $(f^\dagger,g^{\dagger})\in  H_{\dagger}$, where $g^{\dagger}$ is   given by Definition \ref{definition:HdaggerHddagger} (with $a={\alpha_n}, b={\alpha_n}^{-1}$). In the next lines, we deduce from these properties that 
\begin{equation}\label{eqn:viscosity_sub_onpi0} 
    u(\pi_{\alpha_n}) \leq  \lambda g^{\dagger}(\pi_{\alpha_n}) +  h^\dagger(\pi_{\alpha_n}).
\end{equation}
We begin by observing that
\begin{equation*}
    \begin{split}
        \lim_{m \rightarrow + \infty} (1-\epsaln)[\mathcal{G}_{\alpha_n}-{\alpha_n}^{-1} \mathcal{B}_{\alpha_n}](\pi_m,\mu_{\alpha_n},\rho_{\alpha_n},\gamma_{\alpha_n})&\stackrel{\eqref{eqn:def_f0b}}{=}\lim_{m\rightarrow \infty} \left( u-f^\dagger\right)(\pi_m)\\& \stackrel{\eqref{eqn:viscsub1}}{=} \sup_{\pi\in E} \left(u-f^\dagger\right)(\pi)\\&\stackrel{\eqref{eqn:def_f0b}+ \eqref{eqn:Ekeland_optimality}}{=} (1-\epsaln)\mathcal{G}_{\alpha_n}(x_{\alpha_n})\\&=(u-f^{\dagger})(\pi_{\alpha_n}).
    \end{split}
\end{equation*}
At this point, we can use item \ref{item:Ekeland_uniqueness} of Proposition \ref{prop:quadruplication_Ekeland} which gives that $\lim_{m\rightarrow+\infty}\pi_m=\pi_{\alpha_n}$. 

{ Now Lemma \ref{lemma:visc_sol_optimizers_sequence_to_point}, says that since there exists  $\pi_{\alpha_n}\in E$ such that  {$\lim_{m\to+\infty} \pi_m = \pi_{\alpha_n}$} and
	    \begin{equation*}
	        u(\pi_{\alpha_n}) - f^\dagger(\pi_{\alpha_n}) = \sup_\pi u(\pi) - f^\dagger(\pi).
	    \end{equation*}
	    Then we have 
	    \begin{equation*}
	        u(\pi_{\alpha_n}) - \lambda g^\dagger(\pi_{\alpha_n}) - h^\dagger(\pi_{\alpha_n}) \leq 0.
	    \end{equation*}
	    Therefore} we finally establish \eqref{eqn:viscosity_sub_onpi0}. Arguing similarly, we obtain that 
\[v(\mu_{\alpha_n}) \geq  \lambda g^{\ddagger}(\mu_{\alpha_n}) +  h^\ddagger(\mu_{\alpha_n}),  \]
for  $(f^\ddagger,g^{\ddagger})\in  H_{\ddagger}$, where $g^{\ddagger}$ is   given by Definition \ref{definition:HdaggerHddagger} (with $a={\alpha_n} , b={\alpha_n}^{-1}$).

Plugging \eqref{eqn:viscosity_sub_onpi0} and this last bound into \eqref{eqn:first_estimate_u-v} and using the fact that our choice \eqref{eq:quadruplication_Ekeland_step1_2} of $\epsaln$ implies $\epsaln\leq\alpha^{-1}_n$, we arrive at 

\begin{equation}\label{eqn:Hamiltonian_difference_1}
\begin{split}
    \sup_{\pi\in E} u(\pi)-v(\pi) \leq&  \frac{h^\dagger(\pi_{\alpha_n})}{1-\epsaln}- \frac{h^\ddagger(\mu_{\alpha_n})}{1+\epsaln} + \lambda \Big( \frac{1}{1-\epsaln}g^{\dagger}(\pi_{\alpha_n}) - \frac{1}{1+\epsaln}g^{\ddagger}(\mu_{\alpha_n}) \Big)\\ & + \cO({\alpha_n}^{-1/2}).
    \end{split}
\end{equation}

\underline{\emph{Step 3}: Upper bound on the difference of the Hamiltonians.}
Applying the definition of $g^{\dagger}$ and $g^{\ddagger}$ and with the help of Proposition \ref{prop:quadruplication_Ekeland} \ref{item:CIL} we can split the difference of the Hamiltoniains into two terms and a vanishing term, namely
\begin{equation}\label{eqn:g0pi01}
\begin{split}
    &\frac{g^{\dagger}(\pi_{{\alpha_n}})}{1-\epsaln}-\frac{g^{\ddagger}(\mu_{{\alpha_n}})}{1+\epsaln} \leq \\
    &{\alpha_n} \Big[ \frac{1}{1-\epsaln}\big( \cE(\rho_{{\alpha_n}})-\cE(\pi_{{\alpha_n}}) + \frac{\kappa}{2}d^{2}(\pi_{{\alpha_n}},\rho_{{\alpha_n}}) \big)  - \frac{1}{1+\epsaln}\big( \cE(\mu_{{\alpha_n}})-\cE(\gamma_{{\alpha_n}}) - \frac{\kappa}{2}d^{2}(\pi_{{\alpha_n}},\rho_{{\alpha_n}}) \big) \Big]\\
    &+ \frac{{\alpha_n}^2}{2(1-\epsaln)}d^2(\pi_{{\alpha_n}},\rho_{{\alpha_n}}) -\frac{{\alpha_n}^2}{2(1+\epsaln)}d^2(\mu_{{\alpha_n}},\gamma_{{\alpha_n}})\\
    &+ o(1)
\end{split}
\end{equation}

{We gather here the important estimates used in this step and that will be contained in Lemma \ref{lemma:key_estimate_drift} and Lemma \ref{lemma:key2}, whose proof is postponed to section \ref{sec:key_estimates}.

Let $(\alpha_n)_{n\in\mathbb{N}}$ be the sequence given by Proposition \ref{prop:quadruplication_Ekeland}-\ref{item:CIL}, then we have
	 \begin{equation}\label{eq:key1*}
	\begin{split}
	&\alpha_n \Big[ \frac{1}{1-\epsaln}\big( \cE(\rho_{\alpha_n})-\cE(\pi_{\alpha_n}) + \frac{\kappa}{2}d^{2}(\pi_{\alpha_n},\rho_{\alpha_n}) \big)  - \frac{1}{1+\epsaln}\big( \cE(\mu_{\alpha_n})-\cE(\gamma_{\alpha_n}) - \frac{\kappa}{2}d^{2}(\pi_{\alpha_n},\rho_{\alpha_n}) \big) \Big] \\
	&\leq   -  \frac{\epsaln}{(1-\epsaln)} I(\rho_{\alpha_n}) - \frac{\epsaln}{(1+\epsaln)} I(\gamma_{\alpha_n}) + o(1).
	\end{split}
	\end{equation}
	and
	\begin{equation}\label{eq:key2*}
	   \frac{\alpha^2_n}{2(1-\epsaln)}d^2(\pi_{\alpha_n},\rho_{\alpha_n}) -\frac{\alpha^2_n}{2(1+\epsaln)}d^2(\mu_{\alpha_n},\gamma_{\alpha_n}) \leq \frac{\epsaln}{(1-\epsaln)} I(\rho_{\alpha_n})  +\frac{\epsaln}{(1+\epsaln)}   I(\gamma_{\alpha_n})+  o(1).
        \end{equation}
}

If we now apply \ref{eq:key1*} to bound the first term and \ref{eq:key2*} to bound the second term, we obtain that 
\begin{equation*}
    \frac{g^{\dagger}(\pi_{{\alpha_n}})}{1-\epsaln}-\frac{g^{\ddagger}(\mu_{{\alpha_n}})}{1+\epsaln}\leq {o}(1).
\end{equation*}
\underline{\emph{Step 4}: Conclusion.}
Let $\omega^{\dagger}$ be a modulus of continuity for $h^{\dagger}$. Combining the conclusion of \emph{Step 3} with \eqref{eqn:Hamiltonian_difference_1} we obtain that for all $n\in\mathbb{N}$
\begin{equation*}
\begin{split}
    \sup_{\pi\in E}\, u(\pi)-v(\pi) &\leq \omega^{\dagger}(d(\pi_{{\alpha_n}},\mu_{{\alpha_n}}))+ \frac{h^{\dagger}(\mu_{{\alpha_n}})}{1-\epsaln}-\frac{h^{\ddagger}(\mu_{{\alpha_n}})}{1+\epsaln} +o(1)\\&\leq \sup_{\pi\in E}\, h^{\dagger}(\pi)-h^{\ddagger}(\pi)+\omega^{\dagger}(d(\pi_{{\alpha_n}},\mu_{\alpha_n}))+o(1)
    \end{split}
\end{equation*}
where to establish the last inequality we used the boundedness of $h^{\dagger},h^{\ddagger}$ and \eqref{eq:quadruplication_Ekeland_step1_2}. The desired conclusion follows by taking limits on both sides in the above display and invoking one last time Proposition \ref{prop:quadruplication_Ekeland}\ref{item:CIL}
 Note that item \ref{item:CIL} of Proposition \ref{prop:quadruplication_Ekeland} also implies Remark \ref{remark:comparison_weakening_uniformly_continuous_h}.

\end{proof}

\subsection{Proof of proposition \ref{prop:quadruplication_Ekeland}}\label{sec:proof_quadruplication_Ekeland}
\begin{proof}
\begin{itemize}
    \item \emph{\underline{Step 1: quadruplication of variables}}
            We first pick $(\pi_{\alpha,0},\mu_{{\alpha},0})\in E^2$ such that 
            \begin{equation}\label{eq:quadruplication_Ekeland_step1_3}
                \sup_{\pi\in E} u(\pi)-v(\pi) \leq u(\pi_{\alpha,0})-v(\mu_{\alpha,0})- \frac{\alpha}{2} d^2(\pi_{\alpha,0},\mu_{\alpha,0})+\alpha^{-1}.
            \end{equation}        
            Next, we choose $(\rho_{\alpha,0},\gamma_{\alpha,0})\in E^2$ such that 
    \begin{equation}\label{eq:quadruplication_Ekeland_step1_2}
         \cE(\rho_{\alpha,0})+ \cE(\gamma_{\alpha,0})<+\infty, \quad  d(\pi_{\alpha,0},\rho_{\alpha,0})+d(\gamma_{\alpha,0},\mu_{\alpha,0})<\alpha^{-1},
    \end{equation}
    and $\epsal\in (0,1/3)$ such that
        \begin{equation}\label{eq:quadruplication_Ekeland_step1_1}
			\Xi_{\alpha}(x_{\alpha,0})+\epsal < \alpha^{-1},
		\end{equation}
	where  $x_{\alpha,0}=(\pi_{\alpha,0},\rho_{\alpha,0},\mu_{\alpha,0},\gamma_{\alpha,0})$.	
            
    \item \emph{\underline{Step 2:  algebraic bounds on the difference of solutions }} 
    In this step we show that 
        \begin{equation}\label{eq:quadruplication_Ekeland_step2_1}
        \sup_{\pi\in E} u(\pi)-v(\pi) \leq M_{\alpha,0}+\cO(\alpha^{-1}) \leq M_{\alpha} + \cO(\alpha^{-1/2}).
    \end{equation}
    We do so by first showing that 
    \begin{equation}\label{eq:quadruplication_Ekeland_step2_2}
    \begin{split}
                \sup_{\pi\in E} u(\pi)-v(\pi)&\leq \sup_{x\in E^4}\Phi_{\alpha}(x)- \alpha \Psi_{\alpha,0}(x)+ \cO(\alpha^{-1})\\&\leq \Phi_{\alpha}(x_{\alpha,0})- \alpha \Psi_{\alpha,0}(x_{\alpha,0})+ \cO(\alpha^{-1})
                \end{split}
    \end{equation}    
 and eventually establishing that 
    \begin{equation}\label{eq:quadruplication_Ekeland_step2_3}
        \alpha\Psi_{\alpha,0}(x_{\alpha,0}) \geq \alpha\Psi_{\alpha}(x_{\alpha,0})  + \cO(\alpha^{-1/2}).
    \end{equation}
    Once these two bounds have been proven, the desired conclusion \eqref{eq:quadruplication_Ekeland_step2_1} follows at once  using \eqref{eq:quadruplication_Ekeland_step1_1}.

    Let us now proceed to the proof of \eqref{eq:quadruplication_Ekeland_step2_2}. From the boundedness of $u,v$ and using the bounds
    \begin{equation}\label{eq:quadruplication_Ekeland_step2_4}
        \Big|\frac{1}{1-\epsal}-1\Big|+  \Big|\frac{1}{1+\epsal}-1\Big| = \cO(\epsal)=\cO(\alpha^{-1})
    \end{equation}
    we get
    \begin{equation}\label{eq:quadruplication_Ekeland_step2_5}
        \sup_{\pi\in E} u(\pi)-v(\pi) \leq \Phi_{\alpha}(x_{\alpha,0})- \frac{\alpha}{2}d^2(\pi_{\alpha,0},\mu_{\alpha,0})+\cO(\alpha^{-1})
    \end{equation}
    From the choice of $(\pi_{\alpha,0},\mu_{\alpha,0})$ (see \eqref{eq:quadruplication_Ekeland_step1_3}) we deduce that 
    \begin{equation}\label{eq:quadruplication_Ekeland_step2_6}
        \frac{\alpha}{2} d^2(\pi_{\alpha,0},\mu_{\alpha,0})\leq \sup_{\pi}|u|(\pi) +\sup_{\pi}|v|(\pi) + \sup_{\pi}|u-v|(\pi)+ \alpha^{-1}=\cO(1)
    \end{equation}
    But then, using this last bound and \eqref{eq:quadruplication_Ekeland_step2_4} in \eqref{eq:quadruplication_Ekeland_step2_5} we obtain 
    \begin{equation*}
                \sup_{\pi\in E} u(\pi)-v(\pi)\leq \Phi_{\alpha}(x_{\alpha,0})- \alpha \Psi_{\alpha,0}(x_{\alpha,0})+ \cO(\alpha^{-1}),
    \end{equation*}  
    which proves the first inequality of 
     \eqref{eq:quadruplication_Ekeland_step2_2}. 
    To prove the second one, i.e.  
    \begin{equation*}
                \sup_{x\in E^4}\Phi_{\alpha}(x)- \alpha \Psi_{\alpha,0}(x)+ \cO(\alpha^{-1})\leq \Phi_{\alpha}(x_{\alpha,0})- \alpha \Psi_{\alpha,0}(x_{\alpha,0})+ \cO(\alpha^{-1}),
    \end{equation*}
    we proceed as before using the boundedness of $u,v$,  \eqref{eq:quadruplication_Ekeland_step2_4} and \eqref{eq:quadruplication_Ekeland_step2_6}, to show that 
    $$
    \sup_{x\in E^4}\Phi_{\alpha}(x)- \alpha \Psi_{\alpha,0}(x)+ \cO(\alpha^{-1}) \leq  \sup_{\pi, \mu\in E} u(\pi)-v(\mu)- \frac{\alpha}{2}d^2(\pi,\mu) + \cO(\alpha^{-1}).
    $$
    By the choice of $(\pi_{\alpha,0},\mu_{\alpha,0})$ (see \eqref{eq:quadruplication_Ekeland_step1_3}) we obtain
    $$
    \sup_{x\in E^4}\Phi_{\alpha}(x)- \alpha \Psi_{\alpha,0}(x)+ \cO(\alpha^{-1}) \leq   u(\pi_{\alpha,0})-v(\mu_{\alpha,0})- \frac{\alpha}{2}d^2(\pi_{\alpha,0},\mu_{\alpha,0}) + \cO(\alpha^{-1}), 
    $$ and, through analogous computations,  the second inequality of \eqref{eq:quadruplication_Ekeland_step2_2}.

    In order to prove \eqref{eq:quadruplication_Ekeland_step2_3}
    we begin observing that the triangular inequality give
    \begin{equation}\label{eq:quadruplication_Ekeland_step2_7}
        d(\pi_{\alpha,0},\mu_{\alpha,0}) \geq d(\rho_{\alpha,0},\gamma_{\alpha,0})-d(\pi_{\alpha,0},\rho_{\alpha,0})-d(\gamma_{\alpha,0},\mu_{\alpha,0}).
    \end{equation}
    There are two possible cases:
        \begin{itemize}
            \item $d(\rho_{\alpha,0},\gamma_{\alpha,0})<d(\pi_{\alpha,0},\rho_{\alpha,0})+d(\gamma_{\alpha,0},\mu_{\alpha,0})$. In this case, we immediately obtain from our choice of $\rho_{\alpha,0}$ and $\gamma_{\alpha,0}$ that $d(\rho_{\alpha,0},\gamma_{\alpha,0})=\cO(\alpha^{-1})$ from which we deduce that \begin{equation*}\label{eq:quadruplication_Ekeland_step2_8}
                d^2(\pi_{\alpha,0},\mu_{\alpha,0}) =  d^2(\pi_{\alpha,0},\rho_{\alpha,0})+  d^2(\rho_{\alpha,0},\gamma_{\alpha,0})+  d^2(\gamma_{\alpha,0},\mu_{\alpha,0})+\cO(\alpha^{-2}).
            \end{equation*}
            \item $d(\rho_{\alpha,0},\gamma_{\alpha,0}) \geq d(\pi_{\alpha,0},\rho_{\alpha,0})+d(\gamma_{\alpha,0},\mu_{\alpha,0}).$ In this case taking squares in \eqref{eq:quadruplication_Ekeland_step2_7} and using \eqref{eq:quadruplication_Ekeland_step1_2} we get
            \begin{equation*}
            \begin{split}
     d^2(\pi_{\alpha,0},\mu_{\alpha,0}) =&  d^2(\pi_{\alpha,0},\rho_{\alpha,0})+ d^2(\rho_{\alpha,0},\gamma_{\alpha,0})+ d^2(\gamma_{\alpha,0},\mu_{\alpha,0})\\ &+ d(\rho_{\alpha,0},\gamma_{\alpha,0})\cO(\alpha^{-1}) +\cO(\alpha^{-2}).
     \end{split}
            \end{equation*}
            An application of the triangular inequality \eqref{eq:quadruplication_Ekeland_step2_7} in combination with \eqref{eq:quadruplication_Ekeland_step1_2} and \eqref{eq:quadruplication_Ekeland_step2_6} gives that $d(\rho_{\alpha,0},\gamma_{\alpha,0})=\cO(\alpha^{-1/2})$. Plugging this into the above display yields
            \begin{equation}\label{eq:quadruplication_Ekeland_step2_9}
              d^2(\pi_{\alpha,0},\mu_{\alpha,0}) =  d^2(\pi_{\alpha,0},\rho_{\alpha,0})+ d^2(\rho_{\alpha,0},\gamma_{\alpha,0})+ d^2(\gamma_{\alpha,0},\mu_{\alpha,0})+\cO(\alpha^{-3/2}).
            \end{equation}
            \end{itemize}
        Therefore in both cases we have that \eqref{eq:quadruplication_Ekeland_step2_9} holds. Multiplying this relation on both sides by $\frac{\alpha}{2(1-\epsal)}$ and using the basic inequality $\frac{\alpha}{2(1-\epsal)}\leq \frac{\alpha}{2}\leq \frac{\alpha}{2(1+\epsal)}$
        establishes \eqref{eq:quadruplication_Ekeland_step2_3}.
    \item \emph{\underline{Step 3: Ekeland's principle and proof of item \ref{item: duplication},\ref{item:Ekeland_optimality} and \ref{item:Ekeland_uniqueness}}}
     The verification that $\cG_{\alpha}$ and $\cB$ satisfy the hypothesis of Ekeland's Lemma (Lemma \ref{lemma:Ekeland}) is done separately in Lemma \ref{lemma: tataru good for ekeland} in the Appendix. Next, we pick $\hat{x}=(\hat \pi,\hat\mu,\hat\rho,\hat\gamma)\in E^2\times {\mathcal D(\mathcal{E})}^2 $ such that 
    \begin{equation} \label{eqn:delta_optimizers} 
			\sup_{x\in E^4} \mathcal{G}_\alpha(x) -\frac{1}{2}\alpha^{-2} \leq \cG_\alpha(\hat x).
	\end{equation}
    If we now apply Lemma \ref{lemma:Ekeland} setting $\delta=\alpha^{-1}$ we immediately obtain the equality statement in  \eqref{eqn:Ekeland_optimality} thanks to item \ref{item:Ekeland2}-\ref{lemma:Ekeland}. I.e., for each $\alpha > 0$ we can find a unique $x_{\alpha}=(\pi_\alpha,\rho_\alpha,\mu_\alpha,\gamma_\alpha)\in E^2\times\mathcal D(\cE ^2)$  that attains the maximum in  $\sup_{E^4} \cG_\alpha(\cdot) - \alpha^{-1} \cB_\alpha(\cdot)$. Moreover, using item { \ref{item:Ekeland1}-\ref{lemma:Ekeland} in combination with   \eqref{eqn:delta_optimizers}} we prove the inequality statement in \eqref{eqn:Ekeland_optimality}. This concludes the proof of item \ref{item:Ekeland_optimality}. At this point, item \ref{item: duplication} is a direct consequence of equations \eqref{eqn:Ekeland_optimality}, that we have just proven, \eqref{eq:quadruplication_Ekeland_step2_1}, and the fact that $\Xi_{\alpha},\Psi_{\alpha}$ non-negative functions. Item \ref{item:Ekeland_uniqueness} also follows from item \ref{item:Ekeland_unique_optimizer_convergence}-\ref{lemma:Ekeland}.
    \item \emph{\underline{Step 4: Proof of item \ref{item:CIL}}}. 
    We have from item \ref{item:Ekeland_optimality}
    \begin{equation}\label{eq:quadruplication_Ekeland_step4_1}
            M_{\alpha}-\frac12\alpha^{-2} \leq \mathcal G_{\alpha}(x_{\alpha})
            \stackrel{\Xi_{\alpha}\geq 0}{\leq} [\Phi_{\alpha}-\alpha \Psi_{\alpha}](x_\alpha).
    \end{equation}
    Next, we observe that our choice of $\epsal$ and the boundedness of $u,v$ imply
    \begin{equation*}
        \Phi_{\alpha}(x_{\alpha})=\Phi_{\alpha/6}(x_{\alpha})+ \cO(\alpha^{-1}).
    \end{equation*}
    Moreover, using the version of Jensen's inequality \eqref{eq:Jensen_1}, proven separately in Lemma \ref{lemma:Jensen_on_distance}, with the choices $\varepsilon=\epsal,\varepsilon'=\varepsilon_{\alpha/6}$ we obtain
    \begin{equation*}
        \Psi_{\alpha}(x_{\alpha}) \geq \frac{1}{3}\Psi_{\frac{\alpha}{6},0}(x_{\alpha}).
    \end{equation*}
    But then, the right hand side in \eqref{eq:quadruplication_Ekeland_step4_1} is bounded above by
    \begin{equation*}
    \begin{split}
       M_{\alpha}-\frac12\alpha^{-2}&\leq  [\Phi_{\alpha/6}-\frac{\alpha}{6}\Psi_{\alpha/6,0}](x_{\alpha})-\frac{\alpha}{2}\Psi_{\alpha}(x_{\alpha})+\cO(\alpha^{-1}) \\
        &\leq M_{\alpha/6,0} - \frac{\alpha}{2} \Psi_\alpha(x_\alpha) + \cO(\alpha^{-1})\\
        &\stackrel{\eqref{eq:quadruplication_Ekeland_step2_1}}{\leq} M_{\alpha/6} -\frac{\alpha}{2}\Psi_{\alpha}(x_{\alpha})+\cO(\alpha^{-1/2}).
        \end{split}
    \end{equation*}
        We have thus obtained
    \begin{equation*}
        M_{\alpha}-\alpha^{-2} \leq M_{\alpha/6} -\frac{\alpha}{2}\Psi_{\alpha}(x_{\alpha})+\cO(\alpha^{-1/2}).
    \end{equation*}
    Taking $\limsup$ on both sides we get
    \begin{equation*}
        \limsup_{\alpha\to \infty} M_{\alpha} \leq \limsup_{\alpha\to \infty} M_{\alpha/6} -\frac{\alpha}{2}\Psi_{\alpha}(x_{\alpha}) \leq \limsup_{\alpha\to \infty} M_{\alpha/6} - \frac{1}{2}\liminf_{\alpha\to \infty} \alpha\Psi_{\alpha}(x_{\alpha}),
    \end{equation*}
    whence the existence of a sequence $(\alpha_n)_{n\in\mathbb N}$ such that
    \begin{equation}\label{eq:CIL}
        \lim_{n\to +\infty} \alpha_n = + \infty, \quad \lim_{n\to+ \infty}\alpha_n\Psi_{\alpha_n}(x_{\alpha_n})=0.
    \end{equation}
        To conclude the proof, we observe that thanks to item \ref{item:Ekeland_optimality} we have
        \begin{equation*}
            \mathcal{G}_{\alpha}(x_{\alpha})\geq \mathcal{G}_{\alpha}(x_{\alpha,0})-\frac12 {\alpha}^{-2},
        \end{equation*}
        whence, with the help of \eqref{eq:quadruplication_Ekeland_step1_1}
        \begin{equation*}
            \Xi_{\alpha}(x_{\alpha}) \leq [\Phi_{\alpha}-\alpha\Psi_{\alpha}](x_{\alpha})- [\Phi_{\alpha}-\alpha\Psi_{\alpha}](x_{\alpha,0})+\cO(\alpha^{{\-1/2}}).
        \end{equation*}
        Using \eqref{eq:quadruplication_Ekeland_step2_3} on $\alpha \Psi_\alpha(x_{\alpha,0})$ and Lemma \ref{lemma:Jensen_on_distance} to obtain $-\alpha \Psi_\alpha(x_\alpha) \leq - \alpha \Psi_{\alpha,0}(x_\alpha) + {5\alpha} \Psi_{\alpha}(x_\alpha)$, we obtain
        \begin{equation*}
            \Xi_{\alpha}(x_{\alpha}) \leq \cG_{\alpha,0}(x_\alpha) +{5\alpha} \Psi_\alpha(x_\alpha) - \cG_{\alpha,0}(x_{\alpha,0}) +\cO(\alpha^{-{1/2}}).
        \end{equation*}
        Since $ \mathcal{G}_{\alpha,0}(x_{\alpha})\leq M_{\alpha,0}$ and  $\mathcal{G}_{\alpha,0} (x_{\alpha,0})=M_{\alpha,0}+\cO(\alpha^{-1})$ by \eqref{eq:quadruplication_Ekeland_step2_2}, we find
        \begin{equation*}
            \Xi_{\alpha}(x_{\alpha}) \leq  {5\alpha} \Psi_\alpha(x_\alpha)  + \cO(\alpha^{-{1/2}}).
        \end{equation*}
        As a consequence of \eqref{eq:CIL}, if we choose the same sequence $(\alpha_n)_{n\in\mathbb N}$ giving \eqref{eq:CIL}  we have
        \begin{equation}\label{eq:CIL2} 
        \lim_{n\to+ \infty}\Xi_{\alpha_n}(x_{\alpha_n})=0.
    \end{equation}
        
         Finally, observing that by construction
         { $$
    \bar \cE(\pi):= \cE(\pi) + \frac{c_1}{2} d^2(\pi,\nu_0) + c_2.
    $$ for all $\pi\in E$ and some ${c}_1>0,{c}_2\in\R$ this implies by
         Lemma \ref{lem:EVI_properties} \ref{item:EVI quadratic lower bound}  that}
        \begin{equation*}
            \bar{\cE}(\pi) \geq \tilde{c}_1 d^2(\pi,\nu_0)+ \tilde{c}_2
        \end{equation*}
        for all $\pi\in E$ and some $\tilde{c}_1>0,\tilde{c}_2\in\R$,
        we deduce from \eqref{eq:CIL2} that
        \begin{equation*} 
        \lim_{n\to+ \infty} \varepsilon_{\alpha_n} d^2(\rho_{\alpha_n},\nu_0)+\varepsilon_{\alpha_n}d^2(\gamma_{\alpha_n},\nu_0)=0.
    \end{equation*}
        
\end{itemize}
\end{proof}

\subsection{Key estimates}\label{sec:key_estimates}

We now prove the two main estimates we used in the proof of the comparison principle. In the next lemma, we find an upper bound for the first term on the right-hand side in  \eqref{eqn:g0pi01} relying essentially on \ref{item:ass_EVI}. it is precisely here where the use of $d$ instead of $d_T$ in Ekeland's lemma results in weaker estimates that do not allow to conclude the proof of the comparison principle. In Lemma \ref{lemma:key2}, we find an upper bound for the second term on the right-hand side in  \eqref{eqn:g0pi01}, relying on the curves introduced in Assumption \ref{assumption:regularized_geodesics}.

The proofs of these lemmas are partially inspired by Lemma 2.5 and 2.6 of \cite{FeMiZi21}. {In both statements, we use the information functional $I = |\partial \cE|^2$ which was introduced in Definition \ref{definition:information}.}

\begin{lemma}[Estimate on drift from EVI and gradient flow]  \label{lemma:key_estimate_drift} For fixed $\alpha>0$ let $x_\alpha =(\pi_\alpha,\mu_\alpha,\rho_\alpha,\gamma_\alpha)$ and $\nu_0$ be as in the proof of Theorem \ref{theorem:comparison_principle_tildeoperators}. Then, we have that $\mathcal{E}(\pi_\alpha)+\mathcal{E}(\mu_\alpha)<+\infty$ and the following estimates hold
 \begin{equation}\label{eqn:EVI_estimate_1}
	\begin{split}
	&\alpha \left[\cE(\rho_\alpha) - \cE(\pi_\alpha)\right] + \frac{\alpha\kappa}{2} d^2(\rho_\alpha,\pi_\alpha) \\
	&\leq (1-\epsal) \alpha\left[ \cE(\gamma_\alpha) - \cE(\rho_\alpha)\right] - (1-\epsal) \alpha \frac{\kappa}{2} d^2(\rho_\alpha,\gamma_\alpha) -  {\epsal} I(\rho_\alpha) + (1-\epsal)\alpha^{-1}\\
	&\quad {+\epsal c_1 [\mathcal{E}(\nu_0)-\mathcal{E}(\rho_\alpha)] -\epsal c_1\frac{\kappa}{2}d^2(\rho_\alpha,\nu_0)};
	\end{split}
\end{equation}
	 \begin{equation}\label{eqn:EVI_estimate_2}
	\begin{split}
    &\alpha \left[\cE(\mu_\alpha) - \cE(\gamma_\alpha)\right] - \frac{\alpha \kappa}{2} d^2(\gamma_\alpha,\mu_\alpha) \\
	&\geq (1+\epsal) \alpha \left[ \cE(\gamma_\alpha) - \cE(\rho_\alpha)\right] + {(1+\epsal)}\alpha \frac{\kappa}{2} d^2(\rho_\alpha,\gamma_\alpha) +  {\epsal} I(\gamma_\alpha) - (1+\epsal)\alpha^{-1}\\
	&\quad {+\epsal c_1 [\mathcal{E}(\gamma_\alpha)-\mathcal{E}(\nu_0)] +\epsal c_1\frac{\kappa}{2}d^2(\gamma_\alpha,\nu_0)}.
	\end{split}
	\end{equation}
Moreover,  $I(\rho_\alpha) + I(\gamma_\alpha) < \infty$. %\todo{where do we need this later on?}.
	%\todo{R: at first it was necessary in the proof of Lemma \ref{lemma:key2}. To get $|\partial \cE|$ in $\rho_\alpha$. Here it seems ok without}
	
	 As a corollary,  if $(\alpha_n)_{n\in\mathbb{N}}$ is the sequence given by Proposition \ref{prop:quadruplication_Ekeland}-\ref{item:CIL}, then we have
	 \begin{equation}\label{eq:key1}
	\begin{split}
	&\alpha_n \Big[ \frac{1}{1-\epsaln}\big( \cE(\rho_{\alpha_n})-\cE(\pi_{\alpha_n}) + \frac{\kappa}{2}d^{2}(\pi_{\alpha_n},\rho_{\alpha_n}) \big)  - \frac{1}{1+\epsaln}\big( \cE(\mu_{\alpha_n})-\cE(\gamma_{\alpha_n}) - \frac{\kappa}{2}d^{2}(\pi_{\alpha_n},\rho_{\alpha_n}) \big) \Big] \\
	&\leq   -  \frac{\epsaln}{(1-\epsaln)} I(\rho_{\alpha_n}) - \frac{\epsaln}{(1+\epsaln)} I(\gamma_{\alpha_n}) + o(1).
	\end{split}
	\end{equation}
\end{lemma}

\begin{proof}
The fact that $\mathcal{E}(\pi_\alpha)<+\infty$ follows from the subsolution property \eqref{eqn:viscosity_sub_onpi0}  of $u$ and the fact that $u(\pi_\alpha), h^\dagger(\pi_\alpha),\mathcal{E}(\rho_\alpha)$ are all finite quantities. The proof that $\mathcal{E}(\mu_\alpha)<+\infty$ is analogous. Fix $s>0$. From \eqref{item:ass_EVI} and Ekeland's principle \eqref{eqn:Ekeland_optimality} we obtain that the  gradient flow started at $\rho_\alpha$  satisfies
\begin{align*}
   &\alpha \int_0^s \mathcal{E}(\rho_\alpha(r)) -\mathcal{E}(\pi_\alpha)+\frac{\kappa}{2}d^2(\rho_\alpha(r),\pi_\alpha) \, dr\\ &\leq\frac{\alpha d^2(\rho_\alpha,\pi_\alpha)}{2}-\frac{\alpha d^2(\rho_\alpha(s),\pi_\alpha)}{2}\\
   &=(1-\epsal)\Big(\frac{\alpha d^2(\rho_\alpha,\pi_\alpha)}{2(1-\epsal)}+\mathcal{G}_\alpha(x_\alpha)\Big)-(1-\epsal)\Big(\frac{\alpha d^2(\rho_\alpha(s),\pi_\alpha)}{2(1-\epsal)}+\mathcal{G}_\alpha(x_\alpha)\Big)\\
   &\leq (1-\epsal)\Big(\frac{\alpha d^2(\rho_\alpha,\pi_\alpha)}{2(1-\epsal)} +\mathcal{G}_\alpha(x_\alpha)\Big)\\ &-(1-\epsal)\Big(\frac{\alpha d^2(\rho_\alpha{(s)},\pi_\alpha)}{2(1-\epsal)} +\mathcal{G}_\alpha(\pi_\alpha,\mu_\alpha,\rho_\alpha {(s)},\gamma_\alpha)-\alpha^{-1}\mathcal{B}_\alpha(\pi_\alpha,\mu_\alpha,\rho_\alpha{(s)},\gamma_\alpha) \Big). 
\end{align*}
Recalling \eqref{eqn:Gdef}, we can rewrite the last expression as 
\begin{align}
    &(1-\epsal)\alpha\left[\frac{d^2(\rho_\alpha(s),\gamma_\alpha)}{2}-\frac{d^2(\rho_\alpha,\gamma_\alpha)}{2}\right] \label{eqn:EVI_bound1}\\
    &+\epsal[\mathcal{E}(\rho_\alpha(s))-\mathcal{E}(\rho_\alpha)]{+\epsal\frac{c_1} 2[d^2(\rho_\alpha(s),\nu_0)-d^2(\rho_\alpha,\nu_0)]} \label{eqn:energy_identity_bound1}\\
    &+ (1-\epsal)\alpha^{-1} d_T(\rho_\alpha(s),\rho_\alpha).\label{tataru_bound1}
\end{align}
Using \eqref{item:ass_EVI} in \eqref{eqn:EVI_bound1}, the energy identity \eqref{eq_energy_identity}, again \eqref{item:ass_EVI}  in \eqref{eqn:energy_identity_bound1} and Lemma \ref{lemma:estimates_Tataru} (b) in  \eqref{tataru_bound1} we obtain the upper bound
\begin{align*}
    &\int_0^s \alpha(1-\epsal)[\mathcal{E}(\gamma_\alpha)-\mathcal{E}(\rho_\alpha(r)) -\frac{\kappa}{2}d^2(\rho_\alpha(r),\gamma_\alpha)] - \epsal I(\rho_\alpha(r)) d r  + (1-\epsal)\alpha^{-1} s\\
     &{+ \int_0^s\epsal c_1 [\mathcal{E}(\nu_0)-\mathcal{E}(\rho_\alpha(r)) -\frac{\kappa}{2}d^2(\rho_\alpha(r),\nu_0)] dr }. 
\end{align*}
Dividing by $s$ and letting $s\rightarrow0$ we obtain \eqref{eqn:EVI_estimate_1}, recalling that $r\mapsto d^2(\rho_\alpha(r),\gamma_\alpha)$, $r \mapsto d^2(\rho_\alpha(r),\nu_0)$, $r\mapsto \mathcal{E}(\rho_\alpha(r))$ are continuous functions and that $r \mapsto I(\rho_\alpha(r))$ is right continuous by Lemma \ref{lem:EVI_properties} \ref{item:Irightcontinuous}. Arguing in the same way, we obtain \eqref{eqn:EVI_estimate_2}. Finally, having proved   \eqref{eqn:EVI_estimate_1}, if we observe that all terms except $I(\rho_\alpha)$ are finite, we can deduce that $I(\rho_\alpha)<+\infty$.  The proof that $I(\gamma_\alpha)<+\infty$ is completely analogous. At this point, inequality \eqref{eq:key1} follows due to Proposition \ref{prop:quadruplication_Ekeland}-\ref{item:CIL}.
\end{proof}

In the following lemma we obtain obtain an upper bound for the second term in  \eqref{eqn:g0pi01}. Here, it is the fact that $(E,d)$ is a geodesic space together with the geometric conditions \eqref{eq: angle condition} \eqref{eq: energy directional derivative} that play a crucial role.
\begin{lemma}\label{lemma:key2}
For fixed $\alpha>0$, let $x_\alpha =(\pi_\alpha,\mu_\alpha,\rho_\alpha,\gamma_\alpha)$ be as in the proof of Theorem \ref{theorem:comparison_principle_tildeoperators}. Then we have 
\begin{equation}\label{eq: dir der lemma 26}
\frac{\alpha^2}{2} d^2(\rho_\alpha,\pi_\alpha)  \leq (1-\epsal) \frac{1}{2} [\alpha^{-1}+\alpha d(\rho_\alpha,\gamma_\alpha)]^2 + \frac{\epsal}{2} [\sqrt{I(\rho_\alpha)}+ c_1 d(\rho_\alpha,\nu_0)]^2 
\end{equation}
and 
\begin{equation}\label{eq: dir der lemma 25}
   \frac{\alpha^2}{2} d^2(\gamma_\alpha,\mu_\alpha)  \geq (1+\epsal)\frac{\alpha^2}{2}d^2(\gamma_{\alpha},\rho_{\alpha})-\frac12\epsal\left(c_1d(\gamma_{\alpha},\nu_0)+\sqrt{I(\gamma_{\alpha})}\right)^2  + o(1).
\end{equation}
	As a corollary, if $(\alpha_n)_{n\in\mathbb{N}}$ is the sequence given by Proposition \ref{prop:quadruplication_Ekeland}-\ref{item:CIL}, then we have 
	\begin{equation}\label{eq:key2}
	   \frac{\alpha^2_n}{2(1-\epsaln)}d^2(\pi_{\alpha_n},\rho_{\alpha_n}) -\frac{\alpha^2_n}{2(1+\epsaln)}d^2(\mu_{\alpha_n},\gamma_{\alpha_n}) \leq \frac{\epsaln}{(1-\epsaln)} I(\rho_{\alpha_n})  +\frac{\epsaln}{(1+\epsaln)}   I(\gamma_{\alpha_n})+  o(1).
        \end{equation}
\end{lemma}

\begin{proof}
    We begin by proving \eqref{eq: dir der lemma 26}.  First note that if $d(\rho_\alpha,\pi_\alpha) = 0$, there is nothing to prove. We thus only prove the first statement in the case that $d(\rho_\alpha,\pi_\alpha) > 0$. To do so, we define the auxiliary function $\tilde{\mathcal G}_\alpha(\cdot)$ by
    \begin{align*}
        \tilde{\mathcal G}_\alpha(\cdot)&= -(1-\epsal)\mathcal G_\alpha(\pi_\alpha,\mu_\alpha,\cdot,\gamma_\alpha)\\
            &= \frac{\alpha}{2}d^2(\cdot,\pi_\alpha) + (1-\epsal)\frac{{\alpha}}{2}d^2(\cdot,\gamma_\alpha) +\epsal{\bar\cE(\cdot)} + c,
    \end{align*}
    where $c$ is a constant. We obtain from \eqref{eqn:Ekeland_optimality}, the definition of $\mathcal{B}_\alpha$ (see \eqref{eqn:Bdef}) and the Lipschitzianity of Tataru's distance that 
    \begin{equation*}
       \forall \rho\in E, \quad \tilde{\mathcal G}_\alpha(\rho_\alpha)-\tilde{\mathcal G}_\alpha(\rho) \leq (1-\epsal)\alpha^{-1} d_{T}(\rho,\rho_\alpha) \leq (1-\epsal)\alpha^{-1} d(\rho,\rho_\alpha).
    \end{equation*}
    Let us now consider a geodesic $\geo{\rho_\alpha}{\pi_\alpha}$, fix $\theta>0$ small enough, and consider the curve  $\geo{\rho_\alpha}{\pi_\alpha}_{\theta}$ given by  Assumption \ref{assumption:regularized_geodesics}. Choosing $\rho=\geo{\rho_\alpha}{\pi_\alpha}_{\theta}(s)$ in the above estimate and, dividing by $s$, and letting $s\downarrow 0$ we obtain 
    \begin{comment}
    \begin{subequations}\label{eq: dirder lemma 1}
        \begin{equation}\label{eq: dirder lemma 10}
            \liminf_{s\downarrow 0}\frac{\alpha}{2s}[d^2(\rho_\alpha,\pi_\alpha)-d^2(\geo{\rho_\alpha}{\pi_\alpha}_{\theta}(s),\pi_\alpha)]
        \end{equation} 
        \begin{equation}\label{eq: dirder lemma 11}
        \leq  \limsup_{s\downarrow 0} \frac{(1-\epsal)}{\alpha s} d(\geo{\rho_\alpha}{\pi_\alpha}_{\theta}(s),\rho_\alpha) 
        \end{equation}
        \begin{equation}\label{eq: dirder lemma 12}
        + \limsup_{s\downarrow 0}\frac{(1-\epsal)\alpha}{2s}[d^2(\geo{\rho_\alpha}{\pi_\alpha}_{\theta}(s),\gamma_\alpha)-d^2(\rho_\alpha,\gamma_\alpha)]
        \end{equation}
        \begin{equation}\label{eq: dirder lemma 13}
       +  \liminf_{s\downarrow 0}\frac{\epsal}{s}{[\bar \cE(\geo{\rho_\alpha}{\pi_\alpha}_{\theta}(s))-\bar\cE(\rho_\alpha)]}.
        \end{equation}
    \end{subequations}
    \end{comment}
    
    \begin{subequations}\label{eq: dirder lemma 1}
        \begin{align}
            & \liminf_{s\downarrow 0}\frac{\alpha}{2s}[d^2(\rho_\alpha,\pi_\alpha)-d^2(\geo{\rho_\alpha}{\pi_\alpha}_{\theta}(s),\pi_\alpha)]  \label{eq: dirder lemma 10} \\
            & \qquad \leq  \limsup_{s\downarrow 0} \frac{(1-\epsal)}{\alpha s} d(\geo{\rho_\alpha}{\pi_\alpha}_{\theta}(s),\rho_\alpha)  \label{eq: dirder lemma 11} \\
            & \qquad \qquad  + \limsup_{s\downarrow 0}\frac{(1-\epsal)\alpha}{2s}[d^2(\geo{\rho_\alpha}{\pi_\alpha}_{\theta}(s),\gamma_\alpha)-d^2(\rho_\alpha,\gamma_\alpha)] \label{eq: dirder lemma 12} \\
            & \qquad \qquad +  \liminf_{s\downarrow 0}\frac{\epsal}{s}{[\bar \cE(\geo{\rho_\alpha}{\pi_\alpha}_{\theta}(s))-\bar\cE(\rho_\alpha)]}. \label{eq: dirder lemma 13}
        \end{align}
    \end{subequations}
    We start with estimates for all the terms on the right-hand side of \eqref{eq: dirder lemma 1}. To this aim, we observe that for any $\sigma\in E$ we have, using the triangle inequality, the geodesic property and hypothesis \eqref{eq: angle condition}
    \begin{equation}\label{eq: dirder lemma 20}
    \begin{split}
        d(\geo{\rho_\alpha}{\pi_\alpha}_{\theta}(s),\sigma) & \leq d(\rho_\alpha,\sigma)+d(\geo{\rho_\alpha}{\pi_\alpha}_{\theta}(s),\rho_\alpha) \\
        & \leq d(\rho_\alpha,\sigma) + s d(\rho_{\alpha},\pi_{\alpha}) + d(\geo{\rho_\alpha}{\pi_\alpha}_{\theta}(s),\geo{\rho_\alpha}{\pi_\alpha}(s))\\
        & \leq d(\rho_\alpha,\sigma) + s d(\rho_{\alpha},\pi_{\alpha}) + s \theta(1+o(1)).
    \end{split}
    \end{equation}
     Choosing $\sigma=\rho_{\alpha}$ to bound \eqref{eq: dirder lemma 11}, $\sigma = \gamma_\alpha$ for \eqref{eq: dirder lemma 12}, and $\sigma=\nu_0$ to bound the distance term of \eqref{eq: dirder lemma 13} together with
    \begin{equation}\label{eq: dirder lemma 2}
       \liminf_{s \downarrow 0}  \frac{1}{s}[\cE(\geo{\rho_\alpha}{\pi_\alpha}_{\theta}(s))-\cE(\rho_\alpha)]  \stackrel{\eqref{eq: energy directional derivative}}{\leq} \sqrt{I(\rho_\alpha)}(d(\rho_\alpha,\pi_\alpha)+\theta)
    \end{equation}
    for the energy term of \eqref{eq: dirder lemma 13}, we obtain that the right hand side in \eqref{eq: dirder lemma 1} is bounded above by
    \begin{equation}\label{eq: dirder lemma 23}
        (d(\pi_{\alpha},\rho_{\alpha})+\theta) \left((1-\epsal)(\alpha^{-1} +\alpha d(\rho_{\alpha},\gamma_{\alpha})) +\epsal\left( c_1 d(\rho_{\alpha},\nu_{0})+\sqrt{I(\rho_{\alpha})}\right)\right).
    \end{equation}
    Let us now turn the attention to \eqref{eq: dirder lemma 10}. Here, using that 
    \begin{equation*}
        d(\pi_{\alpha},\geo{\rho_\alpha}{\pi_\alpha}_{\theta}(s))\leq (1-s)d(\rho_{\alpha},\pi_{\alpha})+  d(\geo{\rho_\alpha}{\pi_\alpha}(s),\geo{\rho_\alpha}{\pi_\alpha}_{\theta}(s))\leq (1-s)d(\rho_{\alpha},\pi_{\alpha})+s\theta(1+o(1))
    \end{equation*}
    we find that \eqref{eq: dirder lemma 10} is bounded below by 
    \begin{equation}\label{eq: dirder lemma 24}
        \alpha d(\rho_{\alpha},\pi_{\alpha})(d(\rho_{\alpha},\pi_{\alpha}) -  \theta).
    \end{equation}
    Assembling together \eqref{eq: dirder lemma 23} with \eqref{eq: dirder lemma 24}, dividing by $d(\rho_{\alpha},\pi_{\alpha})$ and letting $\theta\rightarrow 0$ yields 
    \begin{equation*}
         \alpha d(\rho_{\alpha},\pi_{\alpha}) \leq \left((1-\epsal)(\alpha^{-1} +\alpha d(\rho_{\alpha},\gamma_{\alpha})) +\epsal\left( c_1 d(\rho_{\alpha},\nu_{0})+\sqrt{I(\rho_{\alpha})}\right)\right),
    \end{equation*}
    from which the bound \eqref{eq: dir der lemma 26} is obtained taking squares on both sides, using convexity of the square function on the right hand side and eventually dividing by two.

    Let us now proceed to the proof of the second inequality. We do the proof in detail as, even though it uses some arguments similar to those used to obtain the first estimate, there are also some non trivial differences.  We begin by noting that we can assume without loss of generality that $d(\rho_\alpha,\gamma_\alpha) > 0$. Next, define the auxiliary test function $\bar{\mathcal G}_\alpha(\cdot)$ by
    \begin{align*}
        \bar{\mathcal G}_\alpha(\cdot)&= -\mathcal G_\alpha(\pi_\alpha,\mu_\alpha,\rho_\alpha,\cdot)\\
            &= \frac{\alpha}{2(1+\epsal) }d^2(\cdot,\mu_\alpha) +\frac{{\alpha}}{2}d^2(\rho_\alpha,\cdot) +\frac \epsal {(1+\epsal)}\bar \cE(\cdot) {+ c}.
    \end{align*}
    We obtain from \eqref{eqn:Ekeland_optimality}, the definition of $\mathcal{B}_\alpha$ (see \eqref{eqn:Bdef}) and the Lipschitzianity of Tataru's distance that 
    \begin{equation*}
       \forall \gamma\in E, \quad \bar{\mathcal G}_\alpha(\gamma_\alpha)-\bar{\mathcal G}_\alpha(\gamma) \leq \alpha^{-1} d_{T}(\gamma,\gamma_\alpha) \leq \alpha^{-1} d(\gamma,\gamma_\alpha).
    \end{equation*}
   
    Let us now consider a geodesic from $\gamma_\alpha$ to $\rho_\alpha$,  $\geo{\gamma_\alpha}{\rho_\alpha}$, (Due to the fact that we don't have linearity and all the properties of the flow given in Assumption \ref{assumption:regularized_geodesics} are given with $\limsup$, we have to go from $\gamma_\alpha$ to $\rho_\alpha$  while for the other inequality we had to go from $\rho_\alpha$ to $\pi_\alpha$) fix a $\theta>0$ small enough, and consider the curve  $\geo{\gamma_\alpha}{\rho_\alpha}_{\theta}$ given by  Assumption \ref{assumption:regularized_geodesics}. Using the previous estimate, we have, for all  $s$ small enough,
    \begin{subequations}\label{eq: dirder lemma 1b}
        \begin{align}
        & \liminf_{s\downarrow0}\frac{\alpha}{2s}[d^2(\rho_\alpha,\gamma_\alpha)-d^2(\rho_\alpha,\geo{\gamma_\alpha}{\rho_\alpha}_{\theta}(s))] \label{eq: dirder lemma 1b0} \\
         & \qquad \leq\limsup_{s\downarrow 0} \frac{1}{\alpha s}d(\geo{\gamma_\alpha}{\rho_\alpha}_{\theta}(s),\gamma_\alpha) \label{eq: dirder lemma 1b1} \\
         & \qquad \qquad +\limsup_{s\downarrow 0} \frac{\alpha}{2s(1+\epsal)}\left(d^2(\geo{\gamma_\alpha}{\rho_\alpha}_{\theta}(s),\mu_\alpha)-d^2(\gamma_\alpha,\mu_\alpha)\right) \label{eq: dirder lemma 1b2} \\
         & \qquad \qquad + \liminf_{s\downarrow 0} \frac \epsal{s(1+\epsal)}[\bar\cE(\geo{\gamma_\alpha}{\rho_\alpha}_{\theta}(s))-\bar \cE(\gamma_\alpha)].\label{eq: dirder lemma 1b3}
         \end{align}
    \end{subequations}
    
    \begin{comment}
    \begin{subequations}\label{eq: dirder lemma 1b}
        \begin{equation}\label{eq: dirder lemma 1b0}
        \liminf_{s\downarrow0}\frac{\alpha}{2s}[d^2(\rho_\alpha,\gamma_\alpha)-d^2(\rho_\alpha,\geo{\gamma_\alpha}{\rho_\alpha}_{\theta}(s))]
         \end{equation}
        \begin{equation}\label{eq: dirder lemma 1b1}
         \leq\limsup_{s\downarrow 0} \frac{1}{\alpha s}d(\geo{\gamma_\alpha}{\rho_\alpha}_{\theta}(s),\gamma_\alpha)
         \end{equation}
         \begin{equation}\label{eq: dirder lemma 1b2}
         +\limsup_{s\downarrow 0} \frac{\alpha}{2s(1+\epsal)}\left(d^2(\geo{\gamma_\alpha}{\rho_\alpha}_{\theta}(s),\mu_\alpha)-d^2(\gamma_\alpha,\mu_\alpha)\right)
         \end{equation} 
         \begin{equation}\label{eq: dirder lemma 1b3}
         + {\liminf_{s\downarrow 0}}\frac \epsal{s(1+\epsal)}[\bar\cE(\geo{\gamma_\alpha}{\rho_\alpha}_{\theta}(s))-\bar \cE(\gamma_\alpha)].
         \end{equation}
    \end{subequations}
    \end{comment}

    In order to estimate all terms containing $d$ on the right hand side, we use the analogous of \eqref{eq: dirder lemma 20}, namely that for all $\sigma\in E$ 
    \begin{equation}\label{eq: dirder lemma 22}
    \begin{split}
        d(\geo{\gamma_\alpha}{\rho_\alpha}_{\theta}(s),\sigma)&\leq d(\gamma_\alpha,\sigma)+d(\geo{\gamma_\alpha}{\rho_\alpha}_{\theta}(s),\gamma_\alpha) \\
        &\leq d(\gamma_\alpha,\sigma) + sd(\gamma_{\alpha},\rho_{\alpha})+ d(\geo{\gamma_\alpha}{\rho_\alpha}_{\theta}(s),\geo{\gamma_\alpha}{\rho_\alpha}(s))\\
        &\leq d(\gamma_\alpha,\sigma) + sd(\gamma_{\alpha},\rho_{\alpha})+s\theta(1+o(1)).
    \end{split}
    \end{equation}
    Indeed, choosing $\sigma=\gamma_{\alpha}$ to bound the right hand side of \eqref{eq: dirder lemma 1b1}, $\sigma=\mu_{\alpha}$ to bound \eqref{eq: dirder lemma 1b2}, $\sigma=\nu_0$ to bound the distance term of \eqref{eq: dirder lemma 1b3} and
    \begin{equation}\label{eq: dirder lemma 2b}
        \liminf_{s \downarrow 0}  \frac{1}{s}[\cE(\geo{\gamma_\alpha}{\rho_\alpha}_{\theta}(s))-\cE(\gamma_\alpha)]  \stackrel{\eqref{eq: energy directional derivative}}{\leq} \sqrt{I(\gamma_\alpha)}(d(\gamma_\alpha,\rho_\alpha)+\theta)
    \end{equation}
    for the energy term of \eqref{eq: dirder lemma 1b3}, we obtain that the right hand side in \eqref{eq: dirder lemma 1b} is bounded above by
    \begin{equation}\label{eq: dirder lemma 23b}
        (d(\gamma_{\alpha},\rho_{\alpha})+\theta) \left(\alpha^{-1}+\frac{\alpha}{(1+\epsal)}d(\gamma_{\alpha},\mu_{\alpha})+\frac \epsal{(1+\epsal)}\left( c_1 d(\gamma_{\alpha},\nu_{0})+\sqrt{I(\gamma_{\alpha})}\right)\right)
    \end{equation}
    
    Let us now turn the attention to \eqref{eq: dirder lemma 1b0}. Here, using that 
    \begin{equation*}
        d(\rho_{\alpha},\geo{\gamma_\alpha}{\rho_\alpha}_{\theta}(s))\leq (1-s)d(\rho_{\alpha},\gamma_{\alpha})+  d(\geo{\gamma_\alpha}{\rho_\alpha}(s),\geo{\gamma_\alpha}{\rho_\alpha}_{\theta}(s))\leq (1-s)d(\rho_{\alpha},\gamma_{\alpha})+s\theta(1+o(1))
    \end{equation*}
    we obtain that \eqref{eq: dirder lemma 1b0} is bounded below by
    \begin{equation}\label{eq: dirder lemma 24b}
        \alpha d(\rho_{\alpha},\gamma_{\alpha})(d(\rho_{\alpha},\gamma_{\alpha}){-}\theta).
    \end{equation}
    Assembling together \eqref{eq: dirder lemma 23b} with \eqref{eq: dirder lemma 24b}, dividing by $d(\rho_{\alpha},\gamma_{\alpha})$ and letting $\theta\rightarrow 0$ yields 
    \begin{equation*}
        \alpha d(\rho_{\alpha},\gamma_{\alpha})-\alpha^{-1} \leq \frac{\alpha}{(1+\epsal)}d(\gamma_{\alpha},\mu_{\alpha})+\frac \epsal{(1+\epsal)}\left( c_1 d(\gamma_{\alpha},\nu_{0})+\sqrt{I(\gamma_{\alpha})}\right)
    \end{equation*}
    If $ \alpha d(\rho_{\alpha},\gamma_{\alpha})-\alpha^{-1}\geq 0$ the bound \eqref{eq: dir der lemma 25} is obtained taking squares on both sides, using convexity of the square function on the right hand side and the fact that $d(\rho_{\alpha},\gamma_{\alpha})$ is $o(1)$. If $ \alpha d(\rho_{\alpha},\gamma_{\alpha})-\alpha^{-1}<0$, it is easily seen that the right hand side of \eqref{eq: dir der lemma 25} is bounded above by a function that is $o(1)$, from which the desired conclusion follows. 
    
    Finally, the bound \eqref{eq:key2} is a consequence of \eqref{eq: dir der lemma 25},\eqref{eq: dir der lemma 26}, Proposition \ref{prop:quadruplication_Ekeland}-\ref{item:CIL} and the basic inequality 
    \begin{equation*}
      \frac{1}{2} \Big( c_1 d(\cdot,\nu_0)+ \sqrt{I(\cdot)}\Big)^2\leq  c_1^2d^2(\cdot,\nu_0)+ I(\cdot).
    \end{equation*}

\end{proof}

\section{Consequences of EVI and properties of the Tataru distances} \label{EVI-Tataru}

\subsection{Consequences of EVI}  \label{section:consequences_EVI}
In this section we deduce from EVI various estimates on the behavior of $d$, $\cE$ {and $I$} along the gradient flow.  These estimates play a fundamental role in the proof of the comparison principle and are be obtained with little effort from those of \cite{MuSa20}.

\begin{lemma}\label{lem:EVI_properties}
Let Assumption \ref{assumption:distance_and_energy}  and \ref{assumption:gradientflow}  hold (in particular EVI inequality \eqref{item:ass_EVI}). 
%Let $\rho,\mu\in E$ be such that $\mathcal{E}(\rho)<+\infty$ a 
{For $\mu\in E$,} let $(\mu(t))_{t\geq 0 }$ be the corresponding gradient flow starting at $\mu.$

Then the following holds: %the function  is bounded on $[0,+\infty]$ for all $\varepsilon>0$. Moreover, the following holds
\begin{enumerate}[(a)]
 %\item\label{item:EVI quadratic lower bound} 
 %\todo{R: 16-9-2021, move it above $-\kappa$ to allow for the new argument in comparison} 
 %For each $\tilde\varepsilon>0$ there exists $c_1 \in (-\kappa,-\kappa + \tilde{\varepsilon})$  if $\kappa \leq 0$, and $c_1 = 0$ if $\kappa > 0$, and for each $\nu\in E$ a $c_2 \in \bR$ such that
	%		\begin{equation}\label{eq:lower bound on energy}
%				\inf_{\pi\in E} \cE(\pi) + \frac{c_1}{2} d^2(\pi,\nu) + c_2 = 0. 
%            \end{equation} 
{ \item \label{item:EVI quadratic lower bound}  For each $c_1 >-\kappa$ and for each $\nu\in E$ there exist $c_2,\tilde{c}_2 \in \bR$ such that
if we set 
\begin{equation*}
    \forall \pi \in E,\quad \bar{\cE}(\pi):=\cE(\pi) + \frac{c_1}{2} d^2(\pi,\nu) + c_2, 
\end{equation*}
then we have
			\begin{equation}\label{eq:lower bound on energy}
				\inf_{\pi\in E} \bar{\cE}(\pi)=0
            \end{equation} 
            and
            \begin{equation*}
                \forall \pi \in E \quad \bar{\cE}(\pi)\geq \frac{\kappa + c_1}{2} d^2(\pi,\nu)+\tilde{c}_2.
            \end{equation*}
            
}
\item\label{item: EVI energy identity} For any $t>0$ we have
		\begin{equation}\label{eq_energy_identity}
		\cE(\mu(t)) - \cE(\mu) = -\int_0^t I(\mu(s)) \dd s.
		\end{equation}
		\item \label{item:domainIdense} The domain $\cD(I)$ is dense in $\cD(\cE)$ and dense in $E$. In particular, the domain $\cD(\cE)$ of $\cE$ is dense in $E$.
		\item \label{item:Irightcontinuous}
		For any $t>0$, we have $I(\mu(t)) < \infty$. The map $t \mapsto I(\mu(t))$ is right-continuous at any $t_0 \geq 0$ such that $I(\mu(t_0)) < \infty$.
        \item\label{item: EVI contraction}		Let $\nu\in E$  and let $(\nu(t))_{t\geq 0}$ be the corresponding gradient flow starting at $\nu$. Then we have 
		\begin{equation}\label{lemma:distance_contracting_under_gradient_flow}
	    d(\mu(t),\nu(t)) \leq e^{-{\kappa} t }d(\mu,\nu)\quad \forall t\in[0,+\infty).
		\end{equation}
		In particular, for a given $\mu\in E$, there is at most one solution of \eqref{item:ass_EVI} such that $\mu(t)\rightarrow \mu$ as $t\rightarrow0$.
\end{enumerate}
\end{lemma}

%\begin{remark}
%\begin{align*}
 %   d^2(\pi,\sigma)  & \leq (d(\pi,\hat \sigma) + d(\hat \sigma,\sigma))^2 \\
  %  & \leq d^2(\pi,\hat \sigma) + 2(1+ d^2(\pi, \hat \sigma)) + d^2(\sigma,\hat{\sigma})
%\end{align*}
%so that any $\sigma$ yields it for all others.
%Note that if \eqref{eq: lower bound on energy} holds for some $\sigma\in E$, then it also holds for any $\rho\in E$ with possibly larger constants.
%\end{remark}

\begin{proof}

We begin by observing that under the current hypothesis the triplet $(E,d,\cE)$ is a metric-functional system in the sense of \cite[Eq 3.1]{MuSa20}. This allows us to deduce most of the results we need to prove from Theorem 3.5 therein. 
\smallskip

Item \ref{item:EVI quadratic lower bound} is proven at \cite[Thm 3.5]{MuSa20}, see Eq (3.15) and the discussion surrounding its proof.

\smallskip

For the proof of \ref{item: EVI energy identity}, note that \cite[Thm 3.5, Eq 3.11]{MuSa20} and \cite[Theorem 2.1.7]{CaSi04} imply that $t \mapsto \cE(\mu(t))$ is locally Lipschitz and, hence, absolutely continuous on $(0,\infty)$. By \cite[Thm 3.5, Eq 3.17]{MuSa20} and the monotone convergence theorem, we obtain \ref{eq_energy_identity}.

\smallskip

Item \ref{item:domainIdense} follows from \cite[Thm 2.10]{MuSa20} and Assumption \ref{assumption:distance_and_energy}, item \ref{item:Irightcontinuous} follows from \cite[Thm 3.5, Eq 3.11 and Eq 3.12]{MuSa20} and item \ref{item: EVI contraction} from \cite[Thm 3.5, Eq 3.10]{MuSa20}. \end{proof}

\subsection{Properties of the Tataru distance}\label{appendixTataru}
We develop here the key results that hold for our adjusted Tataru distance. First of all, note that the infimum in the definition is attained.

{ \begin{remark}
    
Since  the gradient flow   (thanks to Assumption \ref{assumption:gradientflow}) and the distance $d$ are  continuous then the $\inf$ is attained.

Indeed, for all  $\mu,\nu\in E$, we have 
$0\leq d_T(\mu,\nu)\leq 0+ d(\mu, \nu(0) )=d(\mu, \nu )$. Let $(t_n)_{n\in \mathbb N}\in [0,+\infty)$ be a minimizing sequence, i.e.
$$
\lim_{n\to +\infty}t_n+ e^{\hat{\kappa}  t_n} d(\mu,\nu(t_n)) = d_T(\mu,\nu).
$$

Then, for all $n\in\mathbb N$ we have 
$$
0\leq t_n+ e^{\hat{\kappa}  t_n} d(\mu,\nu(t_n))\leq d(\mu, \nu ),
$$
hence $0\leq t_n\leq d(\mu, \nu )$ and $(t_n)_{n\in \mathbb N}$ is a bounded sequence. Passing to a subsequence, still called $(t_n)_{n\in \mathbb N}$ by an abuse of notation, we have $\lim_{n\to +\infty}t_n= \bar t$ for a $\bar t\geq 0$. 

Being the gradient flow $\nu(\cdot)$ and $d$  continuous we also have $$\lim_{n\to +\infty}e^{\hat{\kappa}  t_n} d(\mu,\nu(t_n))= e^{\hat{\kappa}  \bar t} d(\mu,\nu(\bar t)).$$ 

Therefore we must have 
$$
d_T(\mu,\nu)=\bar t + e^{\hat{\kappa} \bar t} d(\mu,\nu(\bar t)).
$$
\end{remark}
}

Secondly, we note that the EVI inequality \eqref{item:ass_EVI} leads to the control on the growth of the distance along two solutions of the gradient flow.

%\begin{proof}
	%The domain where $I$ is finite is dense and closed under the gradient flow. The metric $d$ is continuous and the flow is continuous in the starting point, thus, we can choose without loss of generality $\mu,\nu\in E$ such that $I(\mu) + I(\nu) < \infty$. 
	
%	Applying twice inequality \eqref{item:ass_EVI}, it follows that
 %   $$
%	     \frac{\dd}{\dd t} \left( d^2(\mu(t),\nu(t))\right)\leq - 2 \kappa d^2(\mu(t),\nu(t))\quad \forall t\in[0,+\infty).
%	$$
%	 The result then follows by Grönwall. 
%\end{proof}

\begin{lemma} \label{lemma:estimates_Tataru}
	We have for all $\mu,\hat{\mu}, \nu,\hat{\nu}\in E$ and $r > 0$ that
	\begin{enumerate}
		\item [(a)]
		\begin{equation*}
		d_T(\mu,\nu) - d_T(\hat{\mu},\hat{\nu}) \leq d(\mu,\hat{\mu}) + d(\nu,\hat{\nu})
		\end{equation*}
		\item [(b)]
		\begin{equation*}
		\frac{d_T(\nu(r),\hat{\nu}) - d_T(\nu,\hat{\nu})}{r} \leq 1.
		\end{equation*}
	\end{enumerate}
\end{lemma}

\begin{proof}

	 For (a) Let $t\in [0,+\infty)$ be optimal for $d_T(\hat{\mu},\hat{\nu})$, i.e.  
	\begin{equation*}
	    d_T(\hat{\mu},\hat{\nu})= t + e^{\hat{\kappa} t} d(\hat{\mu},\hat{\nu}(t)).
	\end{equation*}
    Then, we have
	\begin{align*}
	d_T(\mu,\nu) - d_T(\hat{\mu},\hat{\nu}) & \leq e^{\hat{\kappa} t} d(\mu,\nu(t)) - e^{\hat{\kappa} t} d(\hat{\mu},\hat{\nu}(t))  \\
	& \leq e^{\hat{\kappa} t}\left[  d(\mu,\hat{\mu}) + d(\hat{\mu},\nu(t)) - d(\hat{\mu},\hat{\nu}(t)) \right]  \\
	& \leq e^{\hat{\kappa} t}d(\mu,\hat{\mu}) + e^{\hat{\kappa} t} d(\nu(t),\hat{\nu}(t))  \\
	& \leq e^{\hat{\kappa} t}d(\mu,\hat{\mu}) +  e^{ (\hat \kappa-\kappa)t}d(\nu,\hat{\nu}) \\
	& \leq d(\mu,\hat{\mu}) + d(\nu,\hat{\nu}),
	\end{align*}
	where in line 4 we use equation \eqref{lemma:distance_contracting_under_gradient_flow}, in line 5 we use that $\hat{\kappa} \leq 0$ and $\hat{\kappa} - \kappa \leq 0$.
	
	For (b), let $t\in [0,+\infty)$ be optimal for $d_T(\nu,\hat{\nu})$. Then working with the sub-optimal $t+r$ for the first term, we obtain
	\begin{align*}
	\frac{d_T(\nu(r),\hat{\nu}) - d_T(\nu,\hat{\nu})}{r} & \leq \frac{e^{(t+r)\hat\kappa}d(\nu(r),\hat \nu(t+r)) + t+r - e^{t \hat\kappa}d(\nu,\hat \nu(t)) -t }{r} \\
	& \leq \frac{e^{(t+r)\hat\kappa}d(\nu(r),\hat \nu(t+r)) - e^{\hat\kappa t} d(\nu,\hat \nu(t))}{r} + 1  \\
	& \leq e^{\hat\kappa  t} \frac{ d(\nu,\hat \nu(t)) - d(\nu,\hat \nu(t))}{r} + 1 \\
	&  \leq 1 
	\end{align*}
	by equation \eqref{lemma:distance_contracting_under_gradient_flow} and the fact that $\hat{\kappa} - \kappa \leq 0$, $\hat \kappa \leq 0$.

\end{proof}

\begin{lemma} \label{lemma:triangle_inequality_Tataru}
	For $\rho, \mu,\nu\in E$, we have
	\begin{equation*}
	d_T(\rho,\nu) \leq d_T(\rho,\mu) + d_T(\mu,\nu).
	\end{equation*}
\end{lemma}

\begin{proof}
	We have
	\begin{align*}
	d_T(\rho,\nu) & = \inf_{t \geq 0} \left\{t + e^{\hat\kappa t} d(\rho,\nu(t)) \right\} \\
	& = \inf_{t,s \geq 0} \left\{t+s + e^{\hat\kappa (t+s)} d(\rho,\nu(t+s)) \right\} \\
	& \leq \inf_{t,s \geq 0} \left\{t+s + e^{\hat\kappa (t+s)}d(\rho,\mu(t)) + e^{\hat\kappa(t+s)}d(\mu(t),\nu(t+s)) \right\}. 
	\end{align*}
	We now use that, as $\hat\kappa \leq 0$ we have $e^{\hat\kappa(t+s)}d(\rho,\mu(t)) \leq e^{\hat\kappa t}d(\rho,\mu(t))$. For the term $e^{\hat \kappa(t+s)}d(\mu(t),\nu(t+s))$ we use equation \eqref{lemma:distance_contracting_under_gradient_flow} and the fact that $\hat{\kappa} - \kappa \leq 0$. This yields
	\begin{align*}
    d_T(\rho,\nu)& \leq \inf_{t,s \geq 0} \left\{t+s + e^{\hat\kappa t}d(\rho,\mu(t)) + e^{\hat\kappa s}d(\mu,\nu(s)) \right\} \\
	& \leq d_T(\rho,\mu) + d_T(\mu,\nu).
	\end{align*}
\end{proof}

\section{Examples}\label{sec: examples}

In this section, we treat three key examples:
\begin{itemize}
    \item Hilbert spaces, in particular in the context where $\cE$ is derived from a Dirichlet energy. This includes e.g. the linearly controlled Allen-Cahn equation.
    \item Finite dimensional spaces that are essentially Riemannian manifolds.
    \item The Wasserstein space $\cP_2(\bR^d)$.
\end{itemize}

In all the examples, the first step is the verification that the metric space satisfies Assumption \ref{assumption:distance_and_energy} and that there exists a gradient flow satisfying \eqref{item:ass_EVI}. 

We will argue this final point starting from $\kappa$-convexity of the functional $\cE$, see Definition \ref{definition:kappa_convex} below. In concrete examples, this property is typically easier to verify, and is strongly related to \eqref{item:ass_EVI}. Indeed, $\kappa$-convexity of $\cE$ is implied by the existence of a gradient flow satisfying \eqref{item:ass_EVI} by a result of  \cite{DaSa08}. The other implication is not established in general, but includes an extensive list of relevant examples, see the discussion in Section 3.4 of \cite{MuSa20}. For our first two examples we will argue via this route, while for the final example, we will use the methods of \cite{AmGiSa08} based on the $\kappa$-convexity of $\cE$ along generalized geodesics.

\begin{definition} \label{definition:kappa_convex}
Let $\kappa \in \bR$. We say that a lower semi-continuous functional $\cE : E \rightarrow \bR \cup \{\infty\}$ is $\kappa$-convex on a curve $\gamma : [0,1] \rightarrow \cD(\cE)$ if it satisfies for all $t \in [0,1]$ the inequality
\begin{equation*}
\cE\left(\gamma(t)\right) \leq (1-t)\cE(\gamma(0))) + t \cE(\gamma(1))) - \frac{\kappa}{2} t(1-t)d^2(\gamma(0),\gamma(1)).
\end{equation*}
If for any two points $\rho,\pi \in \cD(\cE)$, there exists a constant speed geodesic $\geo{\rho}{\pi}$ such that $\cE$ is $\kappa$-convex on $\geo{\rho}{\pi}$, then we call $\cE$ $\kappa$-convex. If $\cE$ is $\kappa$-convex on all geodesics, then we call $\cE$ strongly $\kappa$-convex. 
\end{definition} 

\begin{theorem}[Theorem 3.2 \cite{DaSa08}] \label{theorem:EVI_implies_convex}
    Consider a lower semi-continuous functional $\cE : E \rightarrow \bR \cup \{\infty\}$ on a geodesic space $(E, d)$ such that there exist a gradient flow satisfying \eqref{item:ass_EVI}. Then $\cE$ is strongly $\kappa$-convex. 
\end{theorem}

Therefore, in all examples below, we can outright assume that we are working with a $\kappa$-convex functional. In this context, the following proposition simplifies establishing Assumption \ref{assumption:regularized_geodesics}.

\begin{proposition} \label{proposition:energy_C1}
Consider the context of Assumption \ref{assumption:distance_and_energy}.
Consider $\rho,\pi$ such that $I(\rho) + \cE(\pi) < \infty$ and let $\geo{\rho}{\pi}$ be the constant speed geodesic between $\rho$ and $\pi$.

Suppose that for each $\theta > 0$ there is a curve $(\geo{\rho}{\pi}_\theta(t))_{t \in [0,1]}$, $\geo{\rho}{\pi}_\theta(0) = \rho$, $\geo{\rho}{\pi}_\theta(t) \neq \rho$ if $t \in [0,1]$ such that:
\begin{enumerate}[(a)]
    \item \label{item:prop:energy_C1_convex} $\cE$ is $\kappa$-convex along $\geo{\rho}{\pi}_\theta$,
    \item \label{item:prop:energy_C1_angle}  the angle condition \eqref{eq: angle condition} holds:
    \begin{equation*}
        \limsup_{t \downarrow 0}   \frac{d(\geo{\rho}{\pi}_\theta(t),\geo{\rho}{\pi}(t))}{t} \leq \theta,
    \end{equation*}
    \item \label{item:prop:energy_C1_key_limit} for all $t$ we have $\geo{\rho}{\pi}_\theta(t) \in \cD(|\partial \cE|)$ and 
    \begin{equation*}
        \lim_{t \downarrow 0} |\partial \cE| (\geo{\rho}{\pi}_\theta(t)) = |\partial \cE| (\rho).
    \end{equation*}
\end{enumerate}
Then Assumption \ref{assumption:regularized_geodesics} holds.
\end{proposition}

\begin{remark}
    Consider the context in which the approximating curves $\geo{\rho}{\pi}_\theta$ are themselves geodesics. Then by Theorem \ref{theorem:EVI_implies_convex} we obtain that $\cE$ is $\kappa$-convex along geodesics implying \ref{item:prop:energy_C1_convex}.
\end{remark}

\begin{remark}
    In a range of contexts, one finds that $I = |\partial \cE|^2$ is convex along geodesics inside $\cD(|\partial \cE|)$. As $|\partial \cE|$ is always lower semi-continuous, this implies \ref{item:prop:energy_C1_key_limit}.
\end{remark}

\begin{proof}
By assumption \ref{item:prop:energy_C1_angle}, it suffices the establish \eqref{eq: energy directional derivative} for the curves $\geo{\rho}{\pi}_\theta$. Due to the $\kappa$-convexity of $\cE$ along $\geo{\rho}{\pi}_\theta$ given in \ref{item:prop:energy_C1_convex}, we can apply Proposition 2.4.9 in \cite{AmGiSa08} to obtain
\begin{equation*}
    d(\geo{\rho}{\pi}_\theta(t),\rho) |\partial \cE(\geo{\rho}{\pi}_\theta(t))| \geq \cE(\geo{\rho}{\pi}_\theta(t)) - \cE(\rho) + \frac{\kappa}{2} d^2(\geo{\rho}{\pi}_\theta(t),\rho).
\end{equation*}
Rewriting the inequality yields
\begin{align*}
    \frac{\cE(\geo{\rho}{\pi}_\theta(t)) - \cE(\rho)}{t} \leq \frac{d(\geo{\rho}{\pi}_\theta(t),\rho)}{t} |\partial \cE(\geo{\rho}{\pi}_\theta(t))| - \frac{\kappa}{2t}d^2(\geo{\rho}{\pi}_\theta(t),\rho).
\end{align*}
Using the triangle inequality, and the angle condition of \ref{item:prop:energy_C1_angle}, and that $\geo{\rho}{\pi}$ is a geodesic, we find
\begin{equation*}
    \limsup_{t \downarrow 0}  \frac{d(\geo{\rho}{\pi}_\theta(t),\rho)}{t} \leq \limsup_{t \downarrow 0}  \frac{d(\geo{\rho}{\pi}_\theta(t),\geo{\rho}{\pi}(t)) + d(\geo{\rho}{\pi}(t), \rho)}{t} \leq \theta + d(\rho,\pi).
\end{equation*}
Combining the two above equations, we have
\begin{multline*}
    \liminf_{t \downarrow 0} \frac{\cE(\geo{\rho}{\pi}_\theta(t)) - \cE(\rho)}{t} \leq \left(\theta + d(\rho,\pi)\right) \liminf_{t \downarrow 0} |\partial \cE(\geo{\rho}{\pi}_\theta(t))| \\
    \leq \left(\theta + d(\rho,\pi)\right) \liminf_{t \downarrow 0} |\partial \cE(\rho)| 
\end{multline*}
establishing the claim.
\end{proof}

\subsection{Hilbert spaces} \label{subsection:Hilbert_spaces}

In this subsection, we assume that $(E,d) = (\cH,\vn{\cdot})$ is a Hilbert space. Below we will verify our Assumptions in two examples, one treats linearly controlled Ornstein-Uhlenbeck type Hamiltonians on general Hilbert spaces, the other treats $L^2(\bR^d)$ with an energy that yields the solution to the Allen-Cahn equation as a gradient flow. For another example where our our methods apply see \cite{FeMiZi21}.

We start out with a general existence result for \eqref{item:ass_EVI}.

\begin{theorem}[Brezis-Pazy, Theorem 3.1 \cite{AmGi13}] \label{theorem:Brezis-Pazy}
Let $\cE$ be $\kappa$-convex and lower semi-continuous. Then there is a unique solution to \eqref{item:ass_EVI} for $\cE$.
\end{theorem}

\subsubsection{The gradient flow constructed from a maximally dissipative operator}

As the main example representing a large class of flows, we consider
\begin{equation} \label{eqn:grad_flow_Hilbert}
    \dot{\rho} = \frac{1}{2}\Delta \rho - \kappa \rho
\end{equation}
on $L^2(\bR^d)$ which formally corresponds to the gradient flow of 
\begin{equation} \label{eqn:Dirichlet_energy_AllenCahn}
    \cE(\rho) = \frac{1}{2}\int |\nabla \rho(x)|^2 + \kappa |\rho(x)|^2 \dd x = - \frac{1}{2} \ip{\Delta \rho - \kappa \rho}{\rho}.
\end{equation}
We see that $\cE$ decomposes as a Dirichlet energy which is lower semi-continuous and convex, combined with $\kappa/2$ times the norm-squared. This implies $\cE$ is $\kappa$-convex and that the gradient flow satisfying \eqref{item:ass_EVI} represented by \eqref{eqn:grad_flow_Hilbert} exists by Theorem \ref{theorem:Brezis-Pazy}.

\smallskip

The use of the Laplacian or the specific form of the Hilbert space in this argument is not essential. The example thus generalizes immediately to the context where we consider a general Hilbert space $\cH$ and replace $\Delta$ in \eqref{eqn:grad_flow_Hilbert} by a maximally dissipative \textit{linear} self-adjoint operator $C$.

\smallskip

We introduce some definitions to take care of general maximally dissipative operators and establish their connection $0$-convex energy functionals.

\begin{definition}
We say that an operator $C \subseteq E \times E$ is dissipative, if for all $(\rho_1,\xi_1),(\rho_2,\xi_2) \in C$ we have
\begin{equation*}
    \ip{\xi_1 - \xi_2}{\rho_1-\rho_2} \leq 0.
\end{equation*}
If $C$ is a single-valued operator, dissipativity is equivalent to
\begin{equation*}
    \ip{C \rho_1 - C \rho_2}{\rho_1 - \rho_2} \leq 0
\end{equation*}
for all $\rho_1,\rho_2 \in \cD(C)$.

\smallskip

We say that an operator $C$ is maximally dissipative if any dissipative extension $B$ of the operator $C$ equals $C$.
\end{definition}

In the context of a maximally dissipative linear and self-adjoint operator, which include all self-adjoint generators of linear strongly continuous semigroups, we thus identify the flow of this semigroup as the gradient flow for the Dirichlet energy constructed from $C$.

\begin{proposition} \label{proposition:Dirichlet_energy_properties}
    Let $(C,\cD(C))$ be a at most single-valued linear self-adjoint and maximally dissipative operator on $\cH$ and let $\kappa \in \bR$.
    
    Let $\cE$ be the lower semi-continuous regularization of the functional
    \begin{equation*}
        \cE_0(\rho) := \begin{cases}
        -\frac{1}{2}\ip{C\rho}{\rho} + \frac{\kappa}{2} \vn{\rho}^2 & \text{if } \rho \in \cD(C), \\
        \infty & \text{otherwise}
        \end{cases}
    \end{equation*}
    Then the conclusion of Theorem \ref{theorem:comparison_principle_tildeoperators} hold for the Hilbert space $\cH$ and energy functional $\cE$.
\end{proposition}

For the proof, we turn to Theorem \ref{theorem:comparison_principle_tildeoperators} and verify Assumptions \ref{assumption:distance_and_energy}, \ref{assumption:gradientflow} and \ref{assumption:regularized_geodesics}. As the first assumption is immediate in this Hilbertian context, we focus on the other two assumptions. To facilitate the verification, we first study the convexity properties and the Frechét subdifferential of $\cE_0$ and $\cE$ in the case that $\kappa = 0$.

\begin{definition}
let $\phi : E \rightarrow \bR \cup \{\infty\}$ be a functional. The Frechét subdifferential $\partial \phi(x)$ at $x \in E$ is given by
\begin{equation} \label{eqn:subgradient_convex_functional}
    \partial \phi(\rho) := \left\{\xi \in E \, \middle| \, \liminf_{\pi \rightarrow \rho} \frac{\phi(\pi) - \phi(\rho) - \ip{\xi}{\pi - \rho}}{\vn{\pi - \rho}} \geq 0  \right\}.
\end{equation}
If $\phi$ is lower semi-continuous and convex then by Proposition 1.4.4 of \cite{AmGi13} also
\begin{equation} \label{eqn:subgradient_convex_functional1}
    \partial \phi(\rho) = \left\{\xi \in E \, \middle| \, \forall \, \pi \in E \colon \, \phi(\pi) - \phi(\rho) - \ip{\xi}{\pi-\rho} \geq 0  \right\}.
\end{equation}
\end{definition}

Note that the notation $|\partial \phi|(\rho)$ for the local slope of $\phi$ at $\rho$ should not be interpreted as the 'size' of $\partial \phi(\rho)$, although the local slope is related to the size of the smallest element in $\partial \phi(\rho)$. See Proposition 1.4.4 of \cite{AmGi13}.

\begin{lemma} \label{lemma:Dirichlet_form_properties}
    Consider the setting of Proposition \ref{proposition:Dirichlet_energy_properties} with $\kappa = 0$. We then have that
    \begin{enumerate}[(a)]
        \item \label{item:Dirichlet_form_preform} $\cE_0 \geq 0$ and for $\rho,\pi \in \cD(C)$ and $t \in [0,1]$ we have 
        \begin{gather}
            \cE_0(\pi) - \cE_0(\rho) - \ip{-C\rho}{\pi-\rho} = \cE_0(\pi-\rho) \geq 0, \label{eqn:Dirichlet_form_subgradient_pre} \\
            \cE_0(\rho + t(\pi-\rho)) = (1-t) \cE_0(\rho) + t\cE_0(\pi) - t(1-t) \cE_0(\pi-\rho). \label{eqn:Dirichlet_form_convexity_pre}
        \end{gather}
        \item \label{item:Dirichlet_form_convexity} $\cD(C) \subseteq \cD(\cE)$, $0 \leq \cE \leq \cE_0$ and $\cE = \cE_0$ on $\cD(C)$ and $\cE$ is $0$-convex. If $\rho$ is such that $\cE(\rho) < \infty$ then there are $\rho_n \in \cD(C)$ such that 
        \begin{equation}\label{eqn:Dirichlet_form_recovery_sequence}
            \lim_n \cE_0(\rho_n) = \lim_n \cE(\rho_n) = \cE(\rho).
        \end{equation}
        \item \label{item:Dirichlet_form_subgradient} $\cD(\partial \cE) = \cD(C)$ and for all $\rho \in \cD(C)$ we have $\partial\cE(\rho) = \{-C\rho\}$ and $|\partial \cE|(\rho) = \vn{C\rho}$.
    \end{enumerate}
    
\end{lemma}

\begin{proof}
For the proof of \ref{item:Dirichlet_form_preform}, note that due to dissipativity $\cE_0 \geq 0$. Next, consider $x,y \in \cD(C)$, then using the linearity of $C$ we obtain
\begin{equation*}
    \cE_0(\pi) - \cE_0(\rho) - \ip{-C \rho}{\pi-\rho} = \cE_0(\pi-\rho) + \frac{1}{2}\left(\ip{C\rho}{\pi} - \ip{C\pi}{\rho}\right).
\end{equation*}
As $C$ is self-adjoint, we have
\begin{equation*}
    \cE_0(\pi) - \cE_0(\rho) - \ip{-C\pi}{\pi-\rho} = \cE_0(\pi-\rho) \geq 0
\end{equation*}
establishing \eqref{eqn:Dirichlet_form_subgradient_pre}. The parallelogram rule in \eqref{eqn:Dirichlet_form_convexity_pre} follows by a direct computation.
We proceed to the second item. As $\cE$ is the lower semi continuous regularization of $\cE_0 \geq 0$, we find $0 \leq \cE \leq \cE_0$. Thus, let $\rho \in \cD(C)$ and consider $\rho_n \in \cD(C)$ such that $\rho_n \rightarrow \rho$. Then by \ref{item:Dirichlet_form_preform}, we have
\begin{equation*}
    \liminf_{n \rightarrow \infty} \cE_0(\rho_n) \geq \liminf_{n \rightarrow \infty} \cE_0(\rho) + \ip{-C\rho}{\rho_n-\rho} = \cE_0(\rho)
\end{equation*}
establishing that $\cE(\rho) = \cE_0(\rho)$. As a consequence, the $0$-convexity of $\cE$ follows from \eqref{eqn:Dirichlet_form_convexity_pre}. \eqref{eqn:Dirichlet_form_recovery_sequence} follows by construction.

\smallskip

To establish \ref{item:Dirichlet_form_subgradient}, first consider $\rho \in \cD(C)$. We verify that $-C\rho \in \partial \cE(\rho)$ by using \eqref{eqn:subgradient_convex_functional1}, in other words, we establish
\begin{equation*}
    \cE(\pi) - \cE(\pi) - \ip{-C\rho}{\pi-\rho} \geq 0
\end{equation*}
for any $\pi$. First note that if $\cE(\pi) = \infty$ there is nothing to prove. So consider $\pi$ such that $\cE(\pi) < \infty$. By \eqref{eqn:Dirichlet_form_recovery_sequence} there are $\pi_n \in \cD(C)$ converging to $\pi$ satisfying $\lim_n \cE(\pi_n) = \cE(\pi)$. Then by \eqref{eqn:Dirichlet_form_subgradient_pre} we have
\begin{multline*}
    \cE(\pi) - \cE(\rho) - \ip{C\rho}{\pi-\rho} \\
    \geq \lim_n \cE(\pi_n) - \cE(\rho) - \ip{-C\rho}{\pi_n-\rho} \geq \lim_n \cE(\pi_n - \rho) \geq 0
\end{multline*}
so that $\rho \in \cD(\partial \cE)$ and $-C\rho \in \partial \cE(\rho)$. It follows that the graph of $C$ is contained in the dissipative operator $- \partial \cE$ and as $C$ is maximally dissipative $C = - \partial \cE$. We thus find that $\partial \cE(\rho) = \{-C\rho\}$ which implies by Proposition 1.4.4 of \cite{AmGiSa08} that $|\partial \cE|(\rho) = \vn{C \rho}$.
\end{proof}

\begin{proof}[Proof of Proposition \ref{proposition:Dirichlet_energy_properties}] 
It suffices to verify Assumptions \ref{assumption:distance_and_energy}, \ref{assumption:gradientflow} and \ref{assumption:regularized_geodesics}. First note that Assumption \ref{assumption:distance_and_energy} is immediate. We next turn to Assumption \ref{assumption:gradientflow} and establish the existence of the gradient flow satisfying \eqref{item:ass_EVI}. 

As the map $\rho \mapsto \frac{\kappa}{2} \vn{\rho}^2$ is $\kappa$-convex, it follows by Lemma \ref{lemma:Dirichlet_form_properties} that $\cE$ is $\kappa$-convex. Thus, Theorem \ref{theorem:Brezis-Pazy} implies the existence of a solution to \eqref{item:ass_EVI} establishing Assumption \ref{assumption:gradientflow}.

 We will verify Assumption \ref{assumption:regularized_geodesics} by means of Proposition \ref{proposition:energy_C1}. Consider $\rho,\pi$ such that $I(\rho) + \cE(\pi) < \infty$. We approximate the geodesic $\geo{\rho}{\pi}(t) = (1-t)\rho + t \pi$ between $\rho$ and $\pi$ by the geodesic $\geo{\rho}{\pi}_\theta(t) := (1-t)\rho + t S(\theta) \pi$ between $\rho$ and $S(\hat{\theta}) \pi$, where $t \mapsto S(t)\pi$ is used to denote the gradient flow started from $\pi$ and where $\hat{\theta}$ is chosen such that $\vn{S(\hat{\theta}) \pi - \pi} \leq \theta$.

\smallskip

To verify the angle condition \ref{item:prop:energy_C1_angle} of Proposition \ref{proposition:energy_C1}, note that
\begin{align*}
    \frac{\vn{\geo{\rho}{\pi}_\theta(t) -\geo{\rho}{\pi}(t)}}{t} & = \frac{\vn{(1-t)\rho + t S(\hat{\theta}) \pi - \left((1-t)\rho + t \pi \right)}}{t} \\
    & = \vn{S(\hat{\theta})\pi - \pi}
\end{align*}
which by choice of $\hat{\theta}$ is smaller than $\theta$.

For the second property, note that by Lemma \ref{lemma:Dirichlet_form_properties} we have
\begin{equation*}
\lim_{t \downarrow 0} |\partial \cE| (\geo{\rho}{\pi}_\theta(t)) = \lim_{t \downarrow 0} \vn{(1-t) C \rho + t C S(\hat{\theta}) \pi}  = \vn{C\rho} = |\partial \cE|(\rho)
\end{equation*}
so that \eqref{eq: energy directional derivative} follows by Proposition \ref{proposition:energy_C1}.

\end{proof}

\subsubsection{The Allen-Cahn equation}

In the context of more concrete Hilbert spaces, we can introduce more general energy functionals. We will not aim for an exhaustive list, but rather consider a single example of interest: the energy functional associated to the Allen-Cahn equation on $\cH = L^2(\bR^d)$:
\begin{equation} \label{eqn:Allen_Cahn}
    \dot{\rho} = \frac{1}{2} \Delta \rho - F'(\rho) - \kappa \rho.
\end{equation}
Here $\kappa \in \bR$ and $F : \bR^d \rightarrow [0,\infty)$ is a non-negative convex $C^1$ function such that $F(0) = 0$. By Remark 2.3.9 and Corollary 1.4.5 in \cite{AmGiSa08}, we can represent this equation as the gradient flow of the energy
\begin{equation} \label{eqn:Allen_Cahn_energy}
    \cE(\rho) = \frac{1}{2} \int |\nabla \rho(x)|^2 + \kappa |\rho(x)|^2 \dd x + \int F(\rho(x)) \dd x.
\end{equation}

\begin{proposition}
    Consider the Hilbert space $\cH = L^2(\bR^d)$ and energy functional $\cE$ of \eqref{eqn:Allen_Cahn_energy}, where $\kappa \in \bR$ and  where $F : \bR^d \rightarrow [0,\infty)$ is a non-negative convex $C^1$ function such that $F(0) = 0$.
    
    Then the conclusion of Theorem \ref{theorem:comparison_principle_tildeoperators} hold.  
\end{proposition}

\begin{proof}
It suffices to verify Assumptions \ref{assumption:distance_and_energy}, \ref{assumption:gradientflow} and \ref{assumption:regularized_geodesics}.

By construction, $\cE$ is $\kappa$-convex. By Theorem \ref{theorem:Brezis-Pazy} the gradient flow for $\cE$ exists and satisfies \eqref{item:ass_EVI}. As in the proof of Proposition \ref{proposition:Dirichlet_energy_properties}, it thus suffices to establish Assumption \ref{assumption:regularized_geodesics}. We do so as above. First note that by (3.4.14) of Remark 2.3.9 and Corollary 1.4.5 in \cite{AmGiSa08} we have
\begin{equation*}
    \partial \cE(\rho) = \begin{cases}
    \{\Delta \rho - F'(\rho) - \kappa \rho\} & \text{if } \Delta \rho, F'(\rho) \in L^2(\bR^d), \\
    \emptyset & \text{otherwise}.
    \end{cases}
\end{equation*}
We thus obtain that
\begin{equation*}
    |\partial \cE| (\rho) = \vn{\Delta \rho - F'(\rho) - \kappa \rho}.
\end{equation*}
We next establish the conditions for Proposition \ref{proposition:energy_C1}, and we do so on the basis of the same curves $\geo{\rho}{\pi}_\theta(t) = (1-t) \rho + t S(\hat{\theta})\pi$ as in the proof of Proposition \ref{proposition:Dirichlet_energy_properties}. $\geo{\rho}{\pi}_\theta$ is therefore the linear interpolation between two elements in $\cD(|\partial \cE)|$. As $F'$ is increasing and $\Delta$ is linear, it follows that $\geo{\rho}{\pi}_\theta(t) \in \cD(|\partial \cE|)$ for all $t \in [0,1]$. We next establish that
\begin{equation} \label{eqn:Allen_Cahn_convergence_information}
    \lim_{t \downarrow 0} |\partial \cE|^2(\geo{\rho}{\pi}_\theta(t)) = |\partial\cE|^2(\rho).
\end{equation}
We will establish this result by the use of the dominated convergence theorem. First of all
\begin{equation*}
    |\partial \cE|^2(\geo{\rho}{\pi}_\theta(t)) = \int |\Delta \geo{\rho}{\pi}_\theta(t)(x) - F'(\geo{\rho}{\pi}_\theta(t)(x)) - \kappa \geo{\rho}{\pi}_\theta(t)(x)|^2 \dd x
\end{equation*}
and as $\geo{\rho}{\pi}_\theta(t) \rightarrow \rho$ point-wise as $t \downarrow 0$, it suffices to find a integrable dominating function. Elementary point-wise estimates yield
\begin{align*}
    & |\Delta \geo{\rho}{\pi}_\theta(t)(x) - F'(\geo{\rho}{\pi}_\theta(t)(x)) - \kappa \geo{\rho}{\pi}_\theta(t)(x)|^2 \\
    & \qquad \leq 3|\Delta \geo{\rho}{\pi}_\theta(t)|^2 +3 |F'(\geo{\rho}{\pi}_\theta(t)(x)|^2 + 3\kappa^2 |\geo{\rho}{\pi}_\theta(t)(x)|^2 \\
    & \qquad \leq 3|\Delta \rho(x)|^2 + 3 |\Delta S(\hat{\theta}) \pi(x)|^2 +3 |F'(\rho(x))|^2 \\
    & \hspace{4cm} + 3 |F'(S(\hat{\theta})\pi(x)) |+ 3\kappa^2 |\rho(x)|^2 + 3 \kappa^2 |S(\hat{\theta})\pi(x)|^2
\end{align*}
as $F'$ is increasing, and all six terms are integrable by assumption. Thus \eqref{eqn:Allen_Cahn_convergence_information} follows by dominated convergence. Thus Assumption \ref{assumption:regularized_geodesics} follows by an application of Proposition \ref{proposition:energy_C1}.
\end{proof}

\subsection{Almost Riemannian manifolds}

In our second set of examples, we consider spaces that are essentially Riemannian manifolds. To illustrate what we are aiming for, consider the Hamiltonian 
\begin{equation} \label{eqn:CIR_Hamiltonian}
    Hf(x) = (\mu-x) f'(x) + \frac{1}{2}x (f'(x))^2, \qquad x \geq 0
\end{equation}
for some constant $\mu >0$. This Hamiltonian arises in the study of Freidlin-Wentzell type large deviation analysis of the Cox-Ingersoll-Ross model in finance \cite{CoInRo85,DFL11}. Following \cite{DFL11}, we study the Hamilton--Jacobi equation using a Riemannian point of view, where the Riemannian metric is generated by the quadratic part of the Hamiltonian. Arguing that the Hamiltonian is a map on the co-tangent bundle, we obtain a metric on the tangent bundle that satisfies $\ip{v}{w}_{g(x)} = x^{-1}vw$ with the metric $g(x) = x^{-1}$ being singular in $0$. 

We will show, however, that by interpreting the drift in \eqref{eqn:CIR_Hamiltonian} as the gradient flow of a functional $\cE$ that satisfies $\cE(0) = \infty$, we can work around the singularity of the metric at the boundary.

\smallskip

The framework that we will be working in is the following.

\begin{assumption} \label{assumption:Riemannian_manifold}
Let $(E,d,\cE)$ be a triple of a complete space $(E,d)$ together with an energy $\cE : E \rightarrow (-\infty,\infty]$. Assume that the following are satisfied.
\begin{enumerate}[(a)]
    \item $E_0 := \cD(\cE)$ is dense in $E$ and the restriction of $d$ to $E_0$ is such that $(E_0,d)$ is a smooth Riemannian manifold. 
    \item $\cE$ is continuously differentiable on $E_0$.
    \item  $\cE$ is $\kappa$-convex along geodesics in $E_0$. 
\end{enumerate}
\end{assumption}

\begin{proposition} \label{proposition_assumption_Riemannian_manifold}
    Suppose that Assumption \ref{assumption:Riemannian_manifold} is satisfied, then the conclusion of Theorem \ref{theorem:comparison_principle_tildeoperators} hold.
\end{proposition}

Before giving the proof, we start with an auxiliary result that relates the slope to directional derivatives.
\begin{definition}
Let $\phi$ be a lower semi-continuous functional. Suppose $x \in \cD(\phi)$. For a geodesic $\geo{x}{y}$ denote the directional derivative of $\phi$ along the geodesic $\geo{x}{y}$ by
\begin{equation*}
    \phi'(x,\geo{x}{y}) := \liminf_{t \downarrow 0} \frac{\phi(\geo{x}{y}(t)) - \phi(x)}{t}.
\end{equation*}
\end{definition}

\begin{lemma} \label{lemma:expression_slope_Riemann}
If $\phi$ is $\kappa$-convex on geodesics, then
\begin{equation*}
    |\partial \phi| (x) = \vn{\grad \phi(x)}_{T_{x}E_0}.
\end{equation*}
and Assumption \ref{assumption:regularized_geodesics} is satisfied for $\phi$.
\end{lemma}

\begin{proof}
For any two points $x,y \in \cD(\cE)$ we will derive \eqref{eq: angle condition} and \eqref{eq: energy directional derivative} with $\theta = 0$ for a geodesic $\geo{x}{y}$. Using the $\kappa$-convexity of $\phi$ on geodesics, we derive as in \cite[Section 2.3]{MuSa20} that
\begin{equation*}
    |\partial \phi| (x) := \sup_{y \in \cD(\phi), \text{ geodesics } \geo{x}{y}} \frac{\phi'(x,\geo{x}{y})}{d(x,y)}.
\end{equation*}
As $\cE$ is continuously differentiable on the domain of $\cE$, we can can obtain an upper bound on the directional derivative by using the Cauchy-Schwarz inequality 
\begin{equation*}
    \phi'(x,\geo{x}{y}) = \lim_{t \downarrow 0} \frac{\phi(\geo{x}{y}(t)) - \phi(x)}{t} = \ip{\grad \phi(x)}{\geod{x}{y}(0)} \leq \vn{\grad \phi(x)}_{T_xE_0} \vn{{\geod{x}{y}}(0)}_{T_x E_0}. 
\end{equation*}
As $\geo{x}{y}$ is a length-minimizing geodesic, we have $\vn{\geod{x}{y}(0)}_{T_x E_0} = d(x,y)$, so that
\begin{equation*}
    |\partial \phi| (x) \leq \vn{\grad \phi(x)}_{T_xE_0}.
\end{equation*}
To establish the converse inequality, recall that on a Riemannian manifold geodesics are locally length minimizing. Thus there is some $\delta > 0$ such that the geodesic (in the Riemannian sense of the word) $\gamma : [0,1] \rightarrow E_0$ started at $x$ in the direction $\grad \phi(x)$ of length $\delta$ satisfies $d(\gamma(0),\gamma(1)) = \delta$, and is thus a geodesic in our sense of the word. A direct computation yields that
\begin{equation*}
    \dot{\gamma}(0) = \frac{\delta}{\vn{\grad \phi(x)}_{T_xE_0}} \grad \phi(x)
\end{equation*}
which implies
\begin{equation*}
    \phi'(x,\geo{x}{y}) = \lim_{t \downarrow 0} \frac{\phi(\gamma(t)) - \phi(x)}{t} = \ip{\grad \phi(x)}{\dot{\gamma}(0)} =  \delta \vn{\grad \cE(x)}_{T_xE_0}.
\end{equation*}
We can conclude that $|\partial \phi| (x) \leq \vn{\grad \phi(x)}_{T_xE_0}$.

\bigskip

For the proof of Assumption \ref{assumption:regularized_geodesics}, we can take for all $x,y$ and $\theta$ the geodesic $\geo{x}{y}$ so that \eqref{eq: angle condition} is satisfied. Note that \eqref{eq: energy directional derivative} can be verified using Cauchy-Schwarz as in the first part of this proof.
\end{proof}

\begin{proof}[Proof of Proposition \ref{proposition_assumption_Riemannian_manifold}]
It suffices to verify Assumptions \ref{assumption:distance_and_energy}, \ref{assumption:gradientflow} and \ref{assumption:regularized_geodesics}. Assumption \ref{assumption:distance_and_energy} is immediate. The gradient flow for $\cE$ can be constructed by local arguments and by construction it remains in $\cD(\cE)$. Assumption \ref{assumption:gradientflow}, or in other words, that the gradient flow satisfies \eqref{item:ass_EVI}, follows by Proposition 23.1 in \cite{Vi09}. Assumption \ref{assumption:regularized_geodesics} follows from Lemma \ref{lemma:expression_slope_Riemann}. 
\end{proof}

For completeness, we verify the assumptions corresponding to the Hamiltonian of \eqref{eqn:CIR_Hamiltonian}. 

\begin{lemma} 
Assumption \ref{assumption:Riemannian_manifold} is satisfied for $E = \bR^+$, $\cE(x) = - \mu \log(x) + x - (\mu - \mu \log \mu)$ and $d(x,y) = 2|\sqrt{x}-\sqrt{y}|$.  
\end{lemma}

Note that the Hamiltonian of \eqref{eqn:CIR_Hamiltonian} is indeed represented by this choice of objects. In particular, note that $\grad \cE(x) = g^{-1}(x) \cE'(x) = x (-\mu/x + 1) = x - \mu$.

\begin{proof}
The functional $\cE$ is smooth and finite on $E_0 := (0,\infty)$. Working in the natural global chart, we can define a Riemannian metric using $g(x) = x^{-1}$, or equivalently $\ip{v}{w}_{g(x)} :=  x^{-1} vw$ on the tangentbundle at $x$. This local metric indeed gives the global metric $d$ of the lemma on $(0,\infty)$, which can then be extended by continuity to the boundary $0$.

\begin{comment}
Set $\gamma_x^{y}(t) = \left(\sqrt{x} +t \left(\sqrt{y} - \sqrt{x} \right)\right)^2$. Then for any $s,t \in (0,1)$ we have
\begin{align*}
    d\left(\gamma_x^{y}(s),\gamma^{x,y}(t) \right) & = 2\left|\left(\sqrt{x} + s \left(\sqrt{y} - \sqrt{x} \right)\right) - \left(\sqrt{x} + t\left(\sqrt{y} - \sqrt{x} \right) \right) \right| \\
    & = 2 \left|(s-t) \left(\sqrt{y}-\sqrt{x}\right) \right| \\
    & = |s-t|d(x,y)
\end{align*}
establishing that $\gamma_x^{y}$ is the constant speed geodesic from $x$ to $y$. 

\end{comment}

\smallskip

We next verify the convexity of $\cE$. As $\cE(0) = \infty$, it suffices to consider geodesics that remain in $(0,\infty)$. Working infinitesimally and considering the geodesic from $x$ to $y$, see Proposition 16.2  of \cite{Vi09}, we verify 
\begin{align*}
& \ip{- \grad{\cE}(x)}{\grad_x \frac{1}{2}d^2(x,y)}_{g(x)} - \ip{- \grad{\cE}(y)}{- \grad_y \frac{1}{2}d^2(x,y)}_{g(y)} \\
& \quad =   2\left(\mu-x\right) \left(1 - \frac{\sqrt{y}}{\sqrt{x}}\right) - 2 \left(- \left(\mu-y\right)\left(1 - \frac{\sqrt{y}}{\sqrt{x}}\right)\right) \\
& \quad  = - 2 \left(\frac{\mu}{\sqrt{xy}} + 1 \right) \left(\sqrt{x} - \sqrt{y}\right)^2 \\
& \leq -2 \left(\sqrt{x} - \sqrt{y}\right)^2 = - \frac{1}{2} d^2(x,y),
\end{align*}
implying that $\cE$ is $1$-convex. 
\end{proof}

\subsection{The Wasserstein space}

We consider $E=\cP_2(\R^d)$, which we equip with the Kantorovich-Wasserstein distance $W_2(\cdot,\cdot)$ of order two, defined by
\begin{equation*}
    W^2_2(\mu,\nu) = \inf_{\pi\in\Pi(\mu,\nu)} \int |x-y|^2 \pi(\De x\De y).
\end{equation*}
Following \cite{AmGiSa08} we consider an energy functional $\cE$ which is the sum of an internal energy, a potential energy and an interaction energy term. More precisely, we consider functions $F:\R_{+}\rightarrow \R$, $V:\R^d\rightarrow\R$, $W:\R^d\rightarrow\R$ such that
\begin{assumption}[McCann's condition]\label{ass: OT energy functional}
\begin{enumerate}[(a)] 
    \item $F:[0,+\infty)\rightarrow \R$ is convex, differentiable with superlinear growth. It satisfies the doubling condition    
    \[\exists C>0: \quad F(z+w) \leq C(1+F(z)+F(w)), \quad \forall z,w\geq 0. \] 
    Moreover we assume that 
    \[ s\mapsto s^dF(s^{-d}) \quad \text{is convex and increasing on $(0,+\infty)$}\]
        and 
        \[ F(0)=0, \quad \lim_{s\rightarrow 0} F(s)/s^{-\alpha}>-\infty, \quad \text{for some $\alpha>\frac{d}{d+2}$}.\]
    \item $V:\R^d\rightarrow (-\infty,+\infty]$ is lower semi-continuous, $\kappa_V$-convex for some $\kappa_V\in \R$, with proper domain that has nonempty interior.
    \item $W:\R^d\rightarrow[0,\infty)$ is an even continuously differentiable $\kappa_W$-convex function for some $\kappa_{W} \geq 0$\footnote{We impose $\kappa_W\geq 0$ as this condition allows us to directly apply the results of \cite{AmGiSa08}. However, it is very likely that this assumption is not necessary and that $\kappa_{W}\in\R$ is enough for Theorem \ref{thm: OT comparision}  to hold.} and satisfies the doubling condition 
    \[ \exists C>0: \quad W(x+y) \leq C(1+W(x)+W(y)), \quad \forall x,y\in \R^d.\]
\end{enumerate}
\end{assumption}
We define our energy functional $\cE$ by

\begin{equation}\label{eq: OT energy}
    \cE(\rho) := \int  F\Big( \frac{\De \rho}{\De \mathscr{L}^d}(x)\Big)\De x + \int V(x)\rho(\De x) + \frac12\int W(x-y) \rho(\De x)\otimes\rho(\De y), 
\end{equation}
setting $\cE(\rho)=+\infty$ as soon as $\rho$ is not absolutely continuous w.r.t the Lebesgue measure $\mathscr{L}^d$. 
The gradient flow of functionals satisfying McCann's condition has attracted lots of interest over the past two decades, because of their connection with PDEs. Indeed, the gradient flow of Boltzmann's entropy $F(s)=s\log s $provides with a variational interpretation of the heat equation \cite{JoKiOt98}, whereas the gradient flow of R\'eny's entropy ($F(s)=\frac{1}{\alpha-1}s^{\alpha}$) relates to the porous medium equation in the same way \cite{Ot01}.

\begin{theorem}\label{thm: OT comparision}
Let $(E,d)=(\cP_2(\R^d),W_2(\cdot,\cdot))$ and $\cE$ be defined by \eqref{eq: OT energy} with $F,V,W$ satisfying Assumption \ref{ass: OT energy functional}. Then the conclusion of Theorem \ref{theorem:comparison_principle_tildeoperators} hold with $\kappa = \kappa_V + \kappa_W$. 
\end{theorem}

The fact that the hypothesis of Theorem \ref{theorem:comparison_principle_tildeoperators} are verified under Assumption \ref{ass: OT energy functional} is a consequence of well-known results, that we essentially take from \cite{AmGiSa08}. For the identification that $\kappa = \kappa_V + \kappa_W$, see Proposition 3.33 in \cite{AmGi13}.

\begin{proof}
We verify the hypothesis of Theorem \ref{theorem:comparison_principle_tildeoperators} one by one.
\begin{itemize}
    \item \underline{\emph{Verification of \ref{assumption:distance_and_energy}}}
    The completeness of $(\cP_2(\R^d),W_2(\cdot,\cdot))$ is proven at \cite[Prop. 7.1.5]{AmGiSa08}. The fact that it is a geodesic space is proven at \cite[Thm 2.10]{AmGi13}. 
    \item \underline{\textit{Verification of Assumption \ref{assumption:gradientflow}}} The existence of an \ref{item:ass_EVI} gradient flow on $\cP_2(\R^d)$ is granted by \cite[Theorems 11.2.1 and 11.2.8]{AmGiSa08}.
    \item \underline{\textit{Verification of Assumption \ref{assumption:regularized_geodesics}}}
    Let us proceed to verify condition \eqref{eq: angle condition}. Given $\rho$ s.t. $I(\rho)<+\infty$ we know that against the Lebesgue measure $\rho$ is regular in the sense of \cite[Def. 6.2.2]{AmGiSa08}. Thus, we can apply \cite[Thm 6.2.4]{AmGiSa08} to obtain the existence of a map $\mathbf{r}$  such that the (unique) geodesic $\geo{\rho}{\pi}$ takes the form
    \begin{equation*}
       \geo{\rho}{\pi}(t) = (\bm{i} + t(\bm r- \bm i))_{\#}\rho \quad \forall t\in[0,1],
    \end{equation*}  
    where $\bm i$ denotes the identity map. Moreover, thanks to \cite[Thm 6.1  ii)]{AmGi13} for any $\theta>0$ we can find $\varphi^{\theta}\in C^{\infty}_c(\R^d)$ such that  
     \begin{equation}\label{eq: OT comparison 1}
         |\nabla\varphi^{\theta}-(\mathbf{r}-\bm i)|_{L^2_{\rho}}\leq \theta. 
     \end{equation}
     Using either a direct calculation or the isometry property of  \cite[Thm 6.1]{AmGi13} we also find that if we define $\geo{\rho}{\pi}_{\theta}(u)= (\bm i + u  \nabla \varphi^{\theta})_{\#}\rho$ for $u$ small enough, then 
    \begin{equation*}
        \lim_{u\downarrow 0}\frac{W_2(\geo{\rho}{\pi}_{\theta}(u),\geo{\rho}{\pi}(u))}{u} \leq |\nabla\varphi^{\theta}-(\mathbf{r}-\bm i)|_{L^2_{\rho}}\leq \theta,
    \end{equation*}
    which is \eqref{eq: angle condition}. We now proceed to verify \eqref{eq: energy directional derivative}. By \cite[Thm 10.4.13]{AmGiSa08} we know that if $I(\rho)<+\infty$, then setting
    \begin{equation*}
        L_F(z)=zF'(z)-F(z)
    \end{equation*}
    we have that $L_F\big(\frac{\De \rho}{\De\mathscr{L}^d} \big)$ belongs to $W^{1,1}_{loc}$. Combining  \cite[Lemma 10.4.4 and Eqs (10.4.58), (10.4.59)]{AmGiSa08}\footnote{In particular, one can check that the hypothesis of Lemma 10.4.4 are verified with $\bm r_t=(1-t)\bm i+ t \nabla \varphi^{\theta}$ using, among other things, the fact that for $t$ small enough $\bm r_t$ is invertible, smooth, strongly convex and $(\bm r_t)_{\#}\rho \ll \mathscr{L}^d$. } 
    \begin{equation*}
    \begin{split}
        \lim_{u\downarrow 0} \frac{\cE(\geo{\rho}{\pi}_{\theta}(u)) - \cE(\rho)}{u} = \int -L_F\big(\frac{\De \rho}{\De\mathscr{L}^d} \big) \Delta\varphi^{\theta} \De \mathscr{L}^d\\
        +\int\langle \nabla V, \nabla\varphi^{\theta} \rangle \De\rho + \int \langle\nabla W * \rho, \nabla\varphi^{\theta}\rangle  \De\rho\\
        =\int \Big\langle\frac{1}{\frac{\De \rho}{\De\mathscr{L}^d}}\nabla L_F\big(\frac{\De \rho}{\De\mathscr{L}^d}\big) +\nabla V + \nabla W * \rho,\nabla\varphi^{\theta}\Big\rangle \De \rho.
    \end{split}
    \end{equation*}
    Applying again \cite[Thm 10.4.13]{AmGiSa08} we have that there exist $\bm w\in L^2_{\rho}$ such that 
    \begin{equation*}
    \begin{split}
        \int |\bm w|^2 \De\rho &= I(\rho),\\
        \bm w &= \frac{1}{\frac{\De \rho}{\De\mathscr{L}^d}}\nabla L_F\big(\frac{\De \rho}{\De\mathscr{L}^d}\big) +\nabla V + \nabla W * \rho \quad \rho\text{-a.e.}
        \end{split}
    \end{equation*} 
    But then by Cauchy Schwartz we find
    \begin{equation*}
        \lim_{u\downarrow 0} \frac{\cE(\geo{\rho}{\pi}_{\theta}(u)) - \cE(\rho)}{u} \leq \sqrt{I(\rho)}|\nabla \varphi^{\theta}|_{L^2_{\rho}} \leq  \sqrt{I}(\rho)(W_2(\rho,\pi)+\theta),
    \end{equation*}
    where to obtain the last inequality we used \eqref{eq: OT comparison 1}, the triangular inequality and the fact that $W_2(\rho,\pi) = \int|\bm r -\bm i|^2 \De \rho$. The proof of \eqref{eq: energy directional derivative} is now complete.
\end{itemize}

\end{proof}

\appendix
\section{Appendix}\label{appendix A}
\subsection{Ekeland's principle}

\begin{lemma}[Ekeland's principle] \label{lemma:Ekeland}
Let $K$ be an abstract set  and $\mathcal B:K\times K \rightarrow [0,+\infty)$ a function with the following properties: 
\begin{enumerate}[(i)]
    \item \label{item:Ekland_assumption1} $\mathcal B(x,x)=0$ for all $x\in K$
    \item \label{item:Ekland_assumption2} $ \mathcal B(x,z)\leq \mathcal B(x,y)+\mathcal B(y,z)$ for all $x,y,z\in K$.
    \item \label{item:Ekland_assumption3} For any sequence $(x_n)_{n\in \bN}\in K$ satisfying $\sum_{n\in\mathbb{N}}\mathcal B(x_{n+1},x_n)<+\infty$, there exists $x\in K$ such that $\lim_{n \rightarrow \infty}\mathcal B(x,x_n)=0$.
\end{enumerate}
Let $\mathcal G:K\rightarrow[-\infty,+\infty)$ be a bounded from above function, i.e.  $\sup_{x\in K}  \mathcal G (x)<+\infty $, such that: 
\begin {itemize}
\item if $(x_n)_{n\in \mathbb{N}}, x \in K$,  $\sum_{n\in \mathbb{N}} \mathcal B(x_{n+1},x_n)<+\infty$
 and   $\lim_{n\rightarrow +\infty}\mathcal B(x,x_n)=0$
then $$\mathcal G(x)\geq \limsup_{n}\mathcal G(x_n).$$
\end{itemize}

Then for each $\delta>0$ and any $\hat x\in K$ such that  $\mathcal G(\hat x)\neq-\infty$ there exists $x_\delta\in  K$ such that 
\begin{enumerate}[(1)]
    \item \label{item:Ekeland1} $\mathcal G(\hat x)+ \frac{1}{2}\delta {\mathcal B(x_\delta,\hat x)}\leq \mathcal G (x_\delta)$,
    \item \label{item:Ekeland2} $\sup_x \left\{\mathcal G(x) - \frac{1}{2} \delta {\mathcal B(x,x_\delta)}\right\} \leq \mathcal G (x_\delta)$. 
\end{enumerate}
Let us note as a corollary that the above statements have the following consequences
\begin{enumerate}[(a)]
    \item \label{item:Ekeland_delta_to_optimizer} Suppose that $\mathcal G(\hat x) \geq \sup_{x\in K} \mathcal G(x) - \frac{1}{2}\delta^2$, then $\mathcal B(x_\delta,\hat x)\leq \delta$.
    \item \label{item:Ekeland_unique_optimizer} For all $x \neq x_\delta$ we have $\mathcal G(x) - \delta {\mathcal B(x,x_\delta)} < \mathcal G (x_\delta)$.
    \item \label{item:Ekeland_unique_optimizer_convergence} Suppose that $(x_n)_{n\in \mathbb{N}}\in K$ is such that $\lim_{n \rightarrow \infty} \mathcal G(x_n) - \delta {\mathcal B(x_n,x_\delta)} = \mathcal G(x_\delta)$, then 
    
    $\lim_{n\rightarrow \infty} {\mathcal B(x_n,x_\delta)} = 0 $ and $\lim_{n\rightarrow \infty} \mathcal G(x_n) = \mathcal G(x_\delta)$.
\end{enumerate}
\end{lemma}

\begin{remark} In particular, from \ref{item:Ekeland1} we deduce $\mathcal G(x_\delta)>-\infty$ and from \ref{item:Ekeland_unique_optimizer} we deduce that $x_\delta$ is the unique optimizer of $ \mathcal G(x) - \delta \mathcal B(x ,x_\delta)$.
\end{remark}

\begin{proof}
The statements \ref{item:Ekeland1} and \ref{item:Ekeland2} follow as in \cite{Ta92}, { using as $\mathcal B(x,y):= \mathcal B (y,x) $, $u (x):=-\mathcal G(x)$}, multiplying all terms by $-1$ and replacing $\delta$ by $\frac{1}{2}\delta$. From \ref{item:Ekeland1} and \ref{item:Ekeland2}, the consequences \ref{item:Ekeland_delta_to_optimizer} and \ref{item:Ekeland_unique_optimizer} follow immediately. We are left to prove \ref{item:Ekeland_unique_optimizer_convergence}.

Let $(x_n)_{n\in \mathbb{N}}\in K$ be as in \ref{item:Ekeland_unique_optimizer_convergence}. Then by statement  \ref{item:Ekeland2}, we have
\begin{equation*}
    0 \leq \mathcal G(x_\delta) - \mathcal G(x_n) + \frac{1}{2} \delta {\mathcal  B(x_n,x_\delta)}.
\end{equation*}
Thus,
\begin{align*}
    0 & \leq \frac{1}{2} \delta {\mathcal B(x_n,x_\delta)} \\
    & \leq \mathcal G(x_\delta) - \mathcal G(x_n) + \delta {\mathcal B(x_n,x_\delta)}.
\end{align*}
By assumption, the right hand side converges to $0$. Therefore, we also have  $$\lim_{n\to \infty}{\mathcal B(x_n,x_\delta)} = 0.$$ 

%{PROBLEM: to use the main assumption we should have $\lim_{n\to \infty}{\mathcal B(x_\delta, x_n)} = 0$ and we have the contrary $\lim_{n\to \infty}{\mathcal B(x_n,x_\delta)} = 0.$

%However that assumption is probably not needed since}

{Using again \ref{item:Ekeland2}, 
$$\mathcal G (x_\delta)\geq \mathcal G(x_n) - \frac{1}{2} \delta {\mathcal B(x_n,x_\delta)}\geq  \limsup_{n\to\infty}
\mathcal G(x_n) - \frac{1}{2} \delta {\mathcal B(x_n,x_\delta)}=\limsup_{n\to\infty}
\mathcal G(x_n).$$}

%By the  main assumptions of the lemma, we find $\mathcal G(x_\delta) \geq \limsup_{n\to \infty} \mathcal G(x_n)$.
Moreover, by the assumption on the sequence $(x_n)_{n\in \mathbb{N}}$, we also have 
\begin{equation*}
    \liminf_{n\to\infty} \mathcal G(x_n) \geq \liminf_{n\to\infty} \mathcal G(x_n) - \delta {\mathcal B(x_n,x_\delta)} = \mathcal G (x_\delta).
\end{equation*}
We then  conclude that $\lim_{n\to\infty} \mathcal G(x_n) = \mathcal G(x_\delta)$.
\end{proof}

Let us show in the following lemma that Ekeland's principle can be applied to the Tataru distance.

\begin{lemma}\label{lemma: tataru good for ekeland}
The Tataru distance satisfies the assumptions of Lemma \ref{lemma:Ekeland}.
\end{lemma}

\begin{proof}
\ref{item:Ekland_assumption1} is trivial and  \ref{item:Ekland_assumption2} has been verified in Lemma \ref{lemma:triangle_inequality_Tataru}. Let us show 
\ref{item:Ekland_assumption3}.

Let $(\mu_n)_{n\in \mathbb{N}}\in E$ be such that $\sum_n d_T(\mu_{n+1},\mu_n) < \infty$,

Recall  we have seen that 

\begin{equation*}
    d_T(\mu,\nu) = \min_{t \geq 0} \left\{ t + e^{\hat{\kappa} t} d(\mu,\nu(t))\right\}
\end{equation*}

Thus, there exists a sequence $(t_n)_{n\in\mathbb{N}}\in [0, +\infty)$ such that

\begin{equation} \label{eqn:dT_for_Ekeland_main_estimate}
    \sum_n t_n + e^{\hat{\kappa}t_n} d(\mu_{n+1},\mu_n(t_n)) < \infty.
\end{equation}
For all $n\in \mathbb{N}$, set  $s_n : = \sum_{k=n}^{\infty} t_k$. Note that \eqref{eqn:dT_for_Ekeland_main_estimate}  implies that
\begin{gather}
    s_n \leq s_0 = \sum_n t_n =: T < \infty, \notag \\ \sum_n d(\mu_{n+1},\mu_n(t_n))=
    \sum_n e^{-\hat{\kappa}t_n}e^{\hat{\kappa}t_n}d(\mu_{n+1},\mu_n(t_n)) \leq e^{-\hat{\kappa}T} \sum_n e^{\hat{\kappa}t_n}  d(\mu_{n+1},\mu_n(t_n)) < \infty. \label{eqn:dT_for_Ekeland_control_on_distances}
\end{gather}
(Remember that $\hat{\kappa}\leq 0$.)

Let us consider the sequence $(\nu_n)_{n\in\mathbb{N}}\in [0, +\infty)\in E$ given by  $\nu_n := \mu_n(s_n)$ for all $n\in\mathbb N$.
It follows by equation \eqref{lemma:distance_contracting_under_gradient_flow} that
\begin{align*}
    \sum_n d(\nu_n,\nu_{n+1}) & = \sum_n d(\mu_n(s_n),\mu_{n+1}(s_{n+1})) \\
    & \leq \sum_n e^{-\kappa s_{n+1}} d(\mu_n(t_n),\mu_{n+1}) \\
    &\leq \sum_n e^{-\hat{\kappa} s_{n+1}} d(\mu_n(t_n),\mu_{n+1}) \\
    & \leq e^{- \hat{\kappa} T}  \sum_n  d(\mu_n(t_n),\mu_{n+1}) \\
    & < \infty.
\end{align*}
Therefore $(\nu_n)_{n\in\mathbb{N}}$ is a Cauchy sequence and converges  to a $\nu\in E$, i.e. $\lim_{n\to\infty}d(\nu,\mu_n(s_n) )=0$. Moreover 

\begin{align*}
0\leq \lim_{n\to \infty} d_T(\nu, \mu_{n})&=\lim_{n\to \infty} \inf_{t \geq 0} \left\{ t + e^{\hat{\kappa} t} d(\nu ,\mu_{n}(t))\right\}\\&\leq \lim_{n\to \infty} s_n + e^{\hat{\kappa}s_n} d(\nu ,\mu_n(s_n))=0.
\end{align*}

\end{proof}
\subsection{From optimizing sequences to optimizing points}
The following Lemma relates Definition \ref{definition:viscosity_solutions_HJ_sequences} to the classical definition stated in terms of optimizing points. We use the lemma in combination with Ekeland's principle in the proof of the comparison principle.

\begin{lemma} \label{lemma:visc_sol_optimizers_sequence_to_point}
	Consider a viscosity subsolution $u$ of equation \eqref{eqn:differential_equation_H1}. Let $(f,g) \in A_\dagger$ {and $(\pi_n)_{n\in \mathbb N}\in E$, be the sequence given by the definition of viscosity subsolution.} %\todo{note in particular, $f$ is lsc, $g$ is usc}  
	Suppose  that: 
	\begin{itemize}
	    \item There exists  $\pi_0\in E$ such that  {$\lim_n \pi_n = \pi_0$} and
	    \begin{equation*}
	        u(\pi_0) - f(\pi_0) = \sup_\pi u(\pi) - f(\pi).
	    \end{equation*}
	    
	    %If $(\pi_n)_{n\in \mathbb{N}}\in E$ is any sequence such that 
	    %\begin{equation*}
	     %   \lim_n u(\pi_n) - f(\pi_n) = %\sup_\pi u(\pi) - f(\pi)
	    %\end{equation*}
	    %then 
	    
	\end{itemize}
	    Then we have 
	    \begin{equation*}
	        u(\pi_0) - \lambda g(\pi_0) - h^\dagger(\pi_0) \leq 0.
	    \end{equation*}
	    
	    Consider a viscosity supersolution $v$ of equation \eqref{eqn:differential_equation_H2}. Let $(f,g) \in A_\ddagger$  {and $(\pi_n)_{n\in \mathbb N}\in E$, be the sequence given by the definition of viscosity supersolution.} Suppose  that: 
	\begin{itemize}
	    \item There exists  $\pi_0\in E$ such that  {$\lim_n \pi_n = \pi_0$} and
	    \begin{equation*}
	        v(\pi_0) - f(\pi_0) = \inf_\pi v(\pi) - f(\pi).
	    \end{equation*}
	    
	    %\item If $(\pi_n)_{n\in \mathbb{N}}\in E$ is any sequence such that 
	    %\begin{equation*}
	       % \lim_n v(\pi_n) - f(\pi_n) = \inf_\pi v(\pi) - f(\pi)
	    %\end{equation*}
	    %then $\lim_n \pi_n = \pi_0$.
	\end{itemize}
	    Then we have 
	    \begin{equation*}
	        v(\pi_0) - \lambda g(\pi_0) - h^\ddagger(\pi_0) \geq 0.
	    \end{equation*}
	\end{lemma}

	\begin{proof}
	We prove the statement for the subsolution case, the supersolution case works analogously.
	
	Let  $u$ be a subsolution to $f - \lambda  {A}_\dagger f = h^\dagger$, $(f,g) \in A_\dagger$ and $(\pi_n)_{n\in \mathbb N}\in E$ be as in the assumption of this lemma. Then in particular we have 
	%{By the viscosity subsolution property of $u$, there exists a sequence $(\pi_n)_{n\in \mathbb N}\in E$ such that}
	\begin{gather*}
	    \limsup_n u(\pi_n) - f(\pi_n) = \sup_\pi u(\pi) - f(\pi), \\
	    \limsup_n u(\pi_n) - \lambda g(\pi_n) - h^\dagger(\pi_n) \leq 0.
	\end{gather*}
    By assumption, there exists  $\pi_0\in E$ such that $u(\pi_0) - f(\pi_0) = \sup_\pi u(\pi) - f(\pi)$ and  $\pi_n \rightarrow \pi_0$.
	    
	 Being $u$  upper semi-continuous, we have $\limsup_n u(\pi_n) \leq u(\pi_0)$. On the other hand, being {$\lim_n u(\pi_n) - f(\pi_n) = u(\pi_0) - f(\pi_0)$}, we have
	    \begin{align*}
	        \liminf_n u(\pi_n) & = \liminf_n (u(\pi_n) - f(\pi_n) + f(\pi_n)) \\
	        & \geq u(\pi_0) - f(\pi_0) + \liminf_n f(\pi_n) \\
	        & \geq u(\pi_0) - f(\pi_0) + f(\pi_0) = u(\pi_0)
	    \end{align*}
	    due to the fact that  $f$ is {continuous}. We can then conclude that $\lim_n u(\pi_n) = u(\pi_0)$. On the other hand, being  $h^\dagger$  continuous and $g$ is upper semi-continuous, we find
	    \begin{align*}
	        0 & \geq \limsup_n (u(\pi_n) - \lambda g(\pi_n) - h^\dagger(\pi_n)) \\
	        & = u(\pi_0) - h^\dagger(\pi_0) + \limsup_n - \lambda g(\pi_n) \\
	        & \geq u(\pi_0) - h^\dagger(\pi_0) + \liminf_n - \lambda g(\pi_n) \\
	        & = u(\pi_0) - h^\dagger(\pi_0) - \lambda \limsup_n g(\pi_n) \\
	        & \geq u(\pi_0) - h^\dagger(\pi_0) - \lambda g(\pi_0).
	    \end{align*}

	\end{proof}

\subsection{A variant of the triangle inequality for the quadratic distance}

For the proof of Proposition \ref{prop:quadruplication_Ekeland}, we need the following combination of the triangle and Jensen inequality.

		\begin{lemma} \label{lemma:Jensen_on_distance}
		Let $\nu_1,\nu_2,\nu_3,\nu_4 \in E$ and  $\varepsilon, \varepsilon' \in (0,1/3)$, then
	\begin{equation}\label{eq:Jensen_1}
		\frac{1}{6} \frac{1}{1-\varepsilon'} \frac{1}{2} d^2(\nu_1,\nu_4) \leq \frac{1}{1-\varepsilon} \frac{1}{2}d^2(\nu_1,\nu_2) + \frac{1}{2}d^2(\nu_2,\nu_3)+ \frac{1}{1+\varepsilon} {\frac{1}{2}}d^2(\nu_3,\nu_4)
	\end{equation}
	\end{lemma}    
	
	\begin{proof}
		By the triangle inequality, we have
		\begin{equation*}
			d(\nu_1,\nu_4) \leq d(\nu_1,\nu_2) + d(\nu_2,\nu_3) + d(\nu_3,\nu_4)
		\end{equation*}
		so that by Jensens inequality, we have
		\begin{align*}
			\frac{1}{6} d^2(\nu_1,\nu_4) & \leq \frac{1}{3} \frac{1}{2} \left(d(\nu_1,\nu_2) + d(\nu_2,\nu_3) + d(\nu_3,\nu_4)\right)^2 \\
			& = 3\frac{1}{2}  \left(\frac{1}{3} d(\nu_1,\nu_2) + \frac{1}{3} d(\nu_2,\nu_3) + \frac{1}{3} d(\nu_3,\nu_4)\right)^2 \\
			& \leq \frac{3}{2} \left(\frac{1}{3}d^2(\nu_1,\nu_2) + \frac{1}{3}d^2(\nu_2,\nu_3) + \frac{1}{3}d^2(\nu_3,\nu_4) \right).
		\end{align*}
		 The second claim follows from this inequality, using that for $\varepsilon, \varepsilon' \in (0,1/3)$ 
		\begin{equation*}
			{\frac{1-\varepsilon}{1-\varepsilon'} \leq2},\qquad \frac{1}{2(1-\varepsilon')} \leq 1, \qquad  1 \leq \frac{1}{1 - \varepsilon}, \qquad  {\frac{1+\varepsilon}{1-\varepsilon'} \leq2}. 
		\end{equation*} 
	\end{proof}

\begin{comment}
\subsection{A convergence Lemma}

	\todo[inline]{R:Can we get a $\alpha,t$ separately version of this one? \\
	2-9-2024: To be replaced by a fully new argument}

The next classical result is stated for a general topological space.
	
	\begin{proposition}[Proposition 3.7 of \cite{CIL92}] \label{proposition:CIL:convergence_optimizers}
	Let $\cX$ be a topological space. Let $\Phi$ be upper semi-continuous and let $\Psi $ be lower semi-continuous such that $\Psi \geq 0$. Set
	\begin{equation*}
	    M_{\alpha} := \sup_\cX \left\{\Phi(x) - \alpha \Psi(x) \right\}.
	\end{equation*}
	Let $- \infty < \lim_{\alpha \rightarrow \infty} M_{\alpha} < \infty$. Let $x_{\alpha}\in\cX$ be such that 
	\begin{equation*}
	    \lim_{\alpha \rightarrow \infty} M_{\alpha} - \Phi(x_{\alpha}) - \alpha \Psi(x_{\alpha})  = 0.
	\end{equation*}
	
	Then the following hold
	\begin{enumerate}
	    \item $\lim_{\alpha \rightarrow \infty} \alpha \Psi(x_{\alpha})  = 0$.
	    \item Let $\hat{x}$ be a limit point of $x_{\alpha}$ as $\alpha \rightarrow \infty$. Then
	    \begin{gather*}
	       \Psi(\hat{x}) = 0, \\
	       \lim_{\alpha \rightarrow \infty} M_{\alpha} = \Phi(\hat{x}) = \sup_{\{x: \Psi(x) = 0 \}} \Phi(x).
	    \end{gather*}
	\end{enumerate}
	\end{proposition}

	\end{comment}

%\end{document}

%\input{Final_version_comparison/intro}
%\input{Final_version_comparison/main_res}
%\input{Final_version_comparison/Proof_of_comparison}
%\input{Final_version_comparison/EVI}
%\input{Final_version_comparison/examples}
%\input{Final_version_comparison/Appendix}

\printbibliography

\end{document}